\documentclass[12pt]{amsart}
\usepackage{setspace}
\usepackage{verbatim}

\usepackage{amsfonts,fancyhdr,ifthen,bm}
\usepackage{amscd,amssymb} 
\usepackage[letterpaper,left=1.35in,right=1.35in,top=1.35in,bottom=1.35in]{geometry}
 
\usepackage{times}
\usepackage{graphicx}
\usepackage{color}
\usepackage{epsfig}
\usepackage{mathrsfs}
\usepackage{paralist}
\usepackage{comment}
\usepackage{mathtools}
\usepackage{multirow}
\usepackage{tikz-cd}
\usepackage[foot]{amsaddr}
\usepackage{caption}

\usepackage[style=numeric, backend=bibtex, maxbibnames=99]{biblatex}
\renewbibmacro{in:}{}

\DeclareBibliographyCategory{Smib}
\addtocategory{Smib}{ Smi22b}

\DeclareFieldFormat{prefixnumber}{\ifcategory{Smib}{\mkbibbrackets{Smi2}}{\mkbibbrackets{#1}}} 
\DeclareFieldFormat{labelnumber}{\ifcategory{Smib}{Smi2}{#1}}

\usepackage{enumitem}

\usepackage{appendix}

\usepackage{tikz,ifthen,calc}
\usetikzlibrary{automata,arrows,positioning,calc}

\newtheorem{thm}{Theorem}[section]

\newtheorem*{thm*}{Theorem}
\newtheorem{prop}[thm]{Proposition}
\newtheorem*{prop*}{Proposition}
\newtheorem{cor}[thm]{Corollary}
\newtheorem*{cor*}{Corollary}

\newtheorem{lem}[thm]{Lemma}
\newtheorem*{lem*}{Lemma}

\newtheorem*{oquest*}{Open Question}
\newtheorem{heur}[thm]{Heuristic}
\newtheorem{intres}[thm]{Intermediate Result}

\theoremstyle{remark}
\newtheorem{rmk}[thm]{Remark}

\theoremstyle{remark}
\newtheorem*{rmk*}{Remark}

\theoremstyle{definition}
\newtheorem{defn}[thm]{Definition}

\theoremstyle{definition}
\newtheorem{notat}[thm]{Notation}

\theoremstyle{definition}
\newtheorem{ass}[thm]{Assumption}

\theoremstyle{definition}

\theoremstyle{definition}
\newtheorem*{defn*}{Definition}

\theoremstyle{definition}
\newtheorem{ex}[thm]{Example}

\theoremstyle{definition}
\newtheorem{ctn}[thm]{Construction}

\theoremstyle{definition}

\numberwithin{equation}{section}%numbers equations by section

\newcommand{\vect}[1]{\bm{#1}}

 \usepackage{chngcntr}
\counterwithin{figure}{section}

\newcommand{\C}{\mathbb{C}}
\newcommand{\Z}{\mathbb{Z}}
\newcommand{\QQ}{\mathbb{Q}}

\DeclareMathOperator{\FFF}{\mathbb{F}}

\newcommand{\Sel}{\textup{Sel}}
\newcommand{\Gal}{\textup{Gal}}

\newcommand{\mfp}{\mathfrak{p}}
\newcommand{\mfq}{\mathfrak{q}}
\newcommand{\Vplac}{\mathscr{V}}
\newcommand{\Frob}{\textup{Frob}}

\newcommand{\Hom}{\textup{Hom}}

\newcommand{\isoarrow}{\xrightarrow{\,\,\,\sim\,\,\,}}

\newcommand{\msS}{\mathscr{S}}

\newcommand{\mfa}{\mathfrak{a}}
\newcommand{\mfb}{\mathfrak{b}}

\newcommand{\cores}{\textup{cor}}
\newcommand{\res}{\textup{res}}
\newcommand{\Cl}{\textup{Cl}\,}

\newcommand{\saaa}{s_{\textbf{a}}}
\newcommand{\sbbb}{s_{\textbf{b}}}

\newcommand{\ovp}{\overline{\mfp}}
\newcommand{\ovq}{\overline{\mfq}}

\DeclareFontFamily{U}{wncy}{}
\DeclareFontShape{U}{wncy}{m}{n}{<->wncyr10}{}
\DeclareSymbolFont{mcy}{U}{wncy}{m}{n}
\DeclareMathSymbol{\Sha}{\mathord}{mcy}{"58}

\newcommand{\e}{\vect{e}}

\newcommand\restr[2]{{% we make the whole thing an ordinary symbol
  \left.\kern-\nulldelimiterspace % automatically resize the bar with \right
  #1 % the function
  \vphantom{\big|} % pretend it's a little taller at normal size
  \right|_{#2} % this is the delimiter
  }}

\DeclareMathOperator*{\rank}{rank}

\newcommand{\Tine}{\textup{Tine}}
\newcommand{\TineF}[2]{\textup{Tine}_{#1}\,#2}
\newcommand{\FrobF}[2]{\textup{Frob}_{#1}\,#2}

\newcommand{\Sigstrict}{\Sigma_{\textup{strict}}}

\newcommand{\Mod}{\textup{Mod}}

\usepackage{scalerel}
\newcommand{\ran}{r_{\textup{an}}}

\newcommand{\ovQQ}{\overline{\mathbb{Q}}}

\newcommand{\mfB}{\mathfrak{B}}
\newcommand{\mfR}{\mathfrak{R}}

\newcommand{\symb}[2]{\left[ #1,\, #2\right]}
\newcommand{\class}[1]{\left[#1\right]}
\newcommand{\spin}[1]{\symb{#1}{#1}}

\newcommand{\zs}{\textup{zs}}
\newcommand{\Zlz}{\mathbb{Z}_{\ell}[\xi]}

\begin{document}
\title[Selmer groups in twist families I]{The distribution of $\ell^{\infty}$-Selmer groups in degree $\ell$ twist families I}

\author{Alexander Smith}
\email{asmith13@stanford.edu}
\date{\today}

\maketitle

\begin{abstract}
In this paper and its sequel, we develop a technique for finding the distribution of $\ell^{\infty}$-Selmer groups in degree $\ell$ twist families of Galois modules over number fields. Given an elliptic curve $E$ over a number field satisfying certain technical conditions, this technique can be used to show that $100\%$ of the quadratic twists of $E$ have rank at most $1$. Given a prime $\ell$ and a number field $F$ not containing $\mu_{2\ell}$, this method also shows that the $\ell^{\infty}$-class groups in the family of degree $\ell$ cyclic extensions of $F$ have a distribution consistent with the Cohen--Lenstra--Gerth heuristics.

For this work, we develop the theory of the fixed point Selmer group, which serves as the base layer of the $\ell^{\infty}$-Selmer group. This first paper gives a technique for finding  the distribution of $\ell^{\infty}$-Selmer groups in certain families of twists where the fixed point Selmer group is stable. In the sequel paper, we will give a technique for controlling fixed point Selmer groups.
\end{abstract}

\setcounter{tocdepth}{1}

\tableofcontents

\section{Introduction}
\label{sec:intro}
The aim of this paper and its sequel \cite{Smi22b} is to give a method for finding the distribution of $\ell^{\infty}$-Selmer groups in degree $\ell$ twist families of Galois modules over number fields for a given prime $\ell$. Such Selmer groups contain a base layer that we call the fixed point Selmer group, with $2$-Selmer groups holding this role in the case $\ell = 2$. The main goal of this paper is to prove Theorem \ref{thm:main_higher}, which controls the distribution of $\ell^{\infty}$-Selmer groups in certain special families of twists where the fixed point Selmer group is stable. The goal of \cite{Smi22b} is to study fixed point Selmer groups in natural twist families and break these natural families into subsets of twists where Theorem \ref{thm:main_higher} can be applied. 

The advantage of splitting our work in this way is that it prioritizes the theory allowing us to control higher Selmer groups, which is the crux of this project. This work appears in Sections \ref{sec:algcomb} and \ref{ssec:TAR} of this paper. The disadvantage of this arrangement is that we can only state and prove the overarching theorem of this project in the second part; it appears as \cite[Theorem 2.14]{Smi22b}.  This pair of papers developed out of the revision of the preprint \cite{Smi17} and constitute the final form of that work.

We will start by listing some of the consequences of the overarching theorem of this project. The proof that \cite[Theorem 2.14]{Smi22b} implies these results will be given in \cite[Section 3]{Smi22b}.

\subsection{Ranks of elliptic curves in twist families}
\label{ssec:elliptic_curves}
Choose an abelian variety $A$ over a number field $F$. Given an integer $n > 1$, the $n$-Selmer group 
\[\Sel^n (A/F)\] 
is a certain effectively calculable finite abelian group that contains the quotient $A(F)/nA(F)$ of the Mordell--Weil group as a subgroup. Defining the $n$-Selmer rank $r_n(A/F)$ to be the maximal integer $r$ so that there is some injection
\[(\Z/n\Z)^r \hookrightarrow \Sel^n(A/F)\big/\text{im}\left(A(F)_{\text{tors}}\right),\]
 it is clear that the rank of $A$ is no larger than $r_n(A/F)$. 

Given a prime $\ell$, the $\ell$-power Selmer ranks of $A/F$ form a nonincreasing sequence
\[r_{\ell}(A/F) \ge r_{\ell^2}(A/F) \ge r_{\ell^3}(A/F) \ge \dots.\]
We define the \emph{$\ell^{\infty}$-Selmer corank} $r_{\ell^{\infty}}(A/F)$ to be the limit of this sequence. If the Shafarevich--Tate conjecture holds for $A/F$, it would follow that
\[\rank(A/F) =r_{\ell^{\infty}}(A/F)\]
for every prime $\ell$. Without this assumption, we still have the inequality
\[\rank(A/F) \le r_{\ell^{\infty}}(A/F).\]

Given nonzero $d$ in $F$, we take $A^d/F$ to be the quadratic twist of $A$ associated to the field extension $F(\sqrt{d})/F$. In the case that $A$ is an elliptic curve over $F$ with Weierstrass form
\[y^2 = x^3 + ax + b,\]
the curve $A^d$ has Weierstrass form
\[y^2 = x^3 + d^2ax + d^3b.\]
In a 1979 paper \cite{Gold79}, Goldfeld conjectured that, among the quadratic twists of a given elliptic curve $A/\QQ$,
\begin{itemize}
\item $50\%$ have analytic rank $0$,
\item $50\%$ have analytic rank $1$, and
\item $0\%$ have any higher analytic rank.
\end{itemize}
Here, the analytic rank $\ran(A/\QQ)$ of the elliptic curve $A/\QQ$ is the order of vanishing of the $L$-function associated to $A$ at $s = 1$. This invariant is known as the analytic rank because, under the Birch and Swinnerton-Dyer conjecture, it is expected to equal the actual rank of the elliptic curve $A/\QQ$.

One major goal of this pair of papers is to prove an analogue of Goldfeld's conjecture for $2^{\infty}$-Selmer coranks. For technical reasons, we need to place certain restrictions on the elliptic curves we consider. These restrictions vary depending on the structure of the $2$-torsion subgroup $A(\QQ)[2]$ of $A(\QQ)$.

\begin{ass}
\label{ass:elliptic_main}
An elliptic curve $A/\QQ$ obeys this assumption if one of the following holds:
\begin{enumerate}
\item $A(\QQ)[2] = 0$; or
\item $A(\QQ)[2] \cong \Z/2\Z\,$ and, writing $\phi: A \rightarrow A_0$ for the unique $\QQ$-isogeny of degree $2$, we have
\[\QQ(A_0[2]) \ne \QQ\quad \text{and}\quad\QQ(A_0[2]) \ne \QQ(A[2]);\,\,\text{ or}\]
\item $A(\QQ)[2] \cong (\Z/2\Z)^2\,$ and $A$ has no cyclic degree $4$ isogeny defined over $\QQ$.
\end{enumerate}
\end{ass}
By \cite[Theorem 1]{Duke97}, most elliptic curves over $\QQ$ satisfy $A(\QQ)_{\text{tor}} = 0$ and hence satisfy this assumption. Our main result for the ranks of elliptic curves is the following.

\begin{thm}
\label{thm:elliptic_main}
Suppose $A/\QQ$ is an elliptic curve satisfying Assumption \ref{ass:elliptic_main}. Then, for $r \ge 0$,
\[\lim_{H \rightarrow \infty} \frac{\#\left\{d \in \Z^{\ne 0} \,:\,\,  |d| \le H \,\,\text{ and }\,\, r_{2^{\infty}}\left(A^d/\mathbb{Q}\right) = r\right\}}{2H} \,=\, \begin{cases} 1/2 &\text{ for }\, r = 0 \\ 1/2 &\text{ for }\, r= 1\\ 0 & \text{ for }\, r\ge 2.\end{cases}\]
\end{thm}

\begin{rmk}
\label{rmk:about_two_logs}
A more precise version of the zero-density portion of this theorem is the following: given $A/\QQ$ satisfying Assumption \ref{ass:elliptic_main}, there are positive constants $c, C > 0$ so, for $H > C$, we have
\[\frac{\#\left\{d \in \Z \,:\,\, 0 < |d| \le H \,\text{ and }\, r_{2^{\infty}}\left(A^d/\mathbb{Q}\right) \ge 2\right\}}{2H} \,\le\, \exp\left(- c \cdot \left(\log \log \log H\right)^{1/2} \right).\]
It is conjectured that the left hand side of this expression can be bounded by $H^{-1/4 + \epsilon}$ \cite[Section 3.3]{Park16}, which is about two logarithms better than the saving term we prove.
\end{rmk}

Our results for $2^{\infty}$-Selmer coranks are a consequence of our results on the distribution of $2^k$-Selmer groups, which use the following transition probabilities.

\begin{defn}
\label{defn:PAlt}
Given $n \ge j \ge 0$, take
\[P^{\text{Alt}}(j \,|\, n)\]
to be the probability that a uniformly selected alternating $n \times n$ matrix with entries in $\mathbb{F}_2$ has kernel of dimension exactly $j$. This probability is zero unless $j$  and $n$ have the same parity. A formula for this probability appears in \cite[Section 2]{Smi22b}.

We will also define
\begin{align*}
&P^{\text{Alt}}(j \,|\,\infty) = \lim_{n \rightarrow \infty} \tfrac{1}{2} \left(P^{\text{Alt}}(j \,|\, 2n) + P^{\text{Alt}}(j \,|\, 2n + 1 ) \right).
\end{align*}
This averaged limit accounts for the fact that $P^{\text{Alt}}(j \,|\,n)$ is $0$ if $j$ and $n$ have different parities.
\end{defn}

\begin{thm}
\label{thm:elliptic_Selmer_13}
Suppose $A/\QQ$ is an elliptic curve that fits into either case (1) or (3) of Assumption \ref{ass:elliptic_main}. 
Given any nonincreasing sequence
\[r_2 \ge r_4 \ge \dots \ge r_{2^k} \ge \dots \]
of nonnegative integers, we have
\begin{align*}
&\lim_{H \rightarrow \infty} \frac{\# \{d \in \Z^{\ne 0} \,:\,\,  |d| < H\,\text{ and }\, r_{2^k}(A^d) = r_{2^k}\,\text{ for all } k \ge 1\}}{2H} \\
& \qquad =  P^{\textup{Alt}}(r_2\,|\,\infty) \cdot \prod_{k= 2}^{\infty} P^{\textup{Alt}}(r_{2^{k}}\, |\, r_{2^{k-1}})
\end{align*}
\end{thm}
To put it another way, as $d$ varies, the sequence $(r_2(A^d), r_4(A^d), \dots)$ behaves like a time-homogeneous Markov chain. We give a representation of this Markov chain in the left part of Figure \ref{fig:Markov}. The probability of starting in an even state is $50\%$, and the probability of starting in an odd state is $50\%$. The absorbing states of this process are $0$ and $1$, and we derive the first and third cases of Theorem \ref{thm:elliptic_main} as a consequence. 

The appearance of alternating matrices in these transition probabilities is explained by the fact that the Cassels--Tate pairing for Selmer groups of elliptic curves is alternating. Indeed, this theorem is consistent with the Cassels--Tate pairings being uniformly distributed among all alternating possibilities, in the sense of Heuristic \ref{heur:random_matrix}. We also note that this result is consistent with the heuristics  given in \cite{BKLPR15}, with the group $\Sel^{2^{\infty}} A^d$ having the distribution predicted by \cite[Conjecture 1.3]{BKLPR15} as $d$ varies.

\begin{figure}

\begin{minipage}[t]{0.5\textwidth}
\begin{center}
\begin{tikzpicture}[->, >=stealth', auto, semithick, node distance=3cm]
\tikzstyle{every state}=[fill=white,draw=black,,text=black,scale=1]
\node[state]    (A)    at (0pt, 0pt)    {$0$};
\node[state]    (B) at (60pt, 35pt) {$1$};
\node[state]    (C) at (0pt, 70pt) {$2$};
\node[state]    (D) at (60pt, 105pt) {$3$};
\node[state]    (E) at (0pt, 140pt) {$4$};
\path
(A) edge[loop below, min distance=5mm] node{$1$} (A)
(B) edge[loop below, min distance=5mm] node{$1$} (B)
(C) edge[loop below, min distance=5mm] node{$\frac{1}{2}$} (C)
(C) edge [bend right = 40] node[label = left: {$\frac{1}{2}$}]{} (A)
(D) edge[loop below, min distance=5mm]  node{$\frac{1}{8}$} (D)
(D) edge [bend left = 33] node{$\frac{7}{8}$} (B)
(E) edge[loop below, min distance=5mm]  node{$\frac{1}{64}$} (E)
(E) edge[bend right = 40] node[label = left : {$\frac{35}{64}$}]{} (C)
(E) edge[bend left = 27] node { $\frac{28}{64}$} (A);

\end{tikzpicture}
\end{center}
\end{minipage}\begin{minipage}[t]{0.5\textwidth}
\begin{center}
\begin{tikzpicture}[->, >=stealth', auto, semithick, node distance=3cm]
\tikzstyle{every state}=[fill=white,draw=black,,text=black,scale=1]
\node[state]    (A)    at (0pt, 0pt)    {$0$};
\node[state]    (B) at (60pt, 35pt) {$1$};
\node[state]    (C) at (0pt, 70pt) {$2$};
\node[state]    (D) at (60pt, 105pt) {$3$};
\path
(A) edge[loop below, min distance=5mm] node{$1$} (A)
(B) edge[loop below, min distance=5mm] node{$\frac{1}{2}$} (B)
(B) edge node[label =below: {\hspace{-0pt}$\frac{1}{2}$}]{} (A)
(C) edge[loop left, min distance=5mm] node{$\frac{1}{16}$} (C)
(C) edge node{\hspace{-7pt}$\frac{9}{16}$} (B)
(C) edge node{$\frac{6}{16}$} (A)
(D) edge[loop above, min distance=5mm]  node{$\frac{1}{512}$} (D)
(D) edge node[label = above:  {\hspace{-11pt}$\frac{49}{512}$}]{} (C)
(D) edge node[label = right: {\hspace{-8pt}$\frac{294}{512}$}]{} (B)
(D) edge [out= 165, in = 150,  looseness = 1.8] node {\hspace{-32pt}$\frac{168}{512}$} (A);

\end{tikzpicture}
\end{center}
\end{minipage}
\caption{Diagrams for the Markov chains that model $2^k$-Selmer ranks (on the left) and $2^k$-class ranks (on the right). Each diagram omits infinitely many possible higher-rank states. }
\label{fig:Markov}
\end{figure}
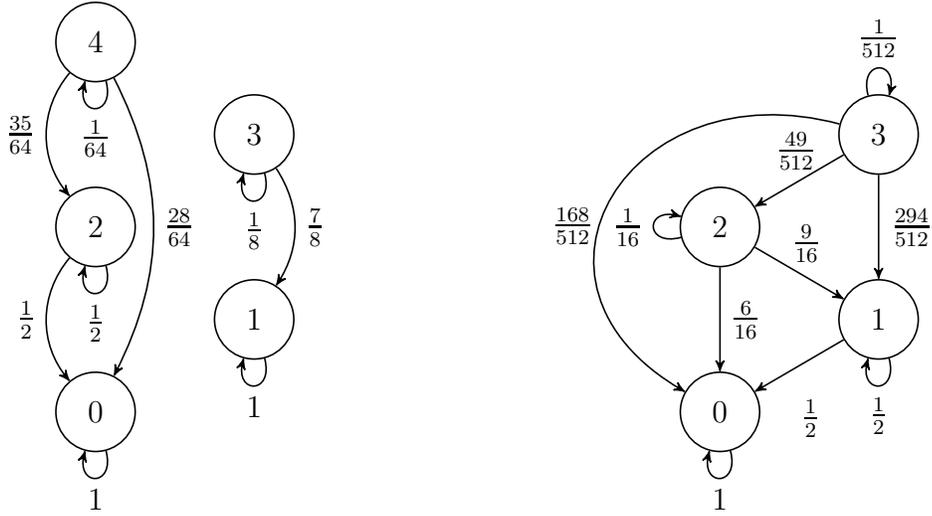

\begin{rmk}
The variant of Theorem \ref{thm:elliptic_Selmer_13} that applies to elliptic curves satisfying Assumption \ref{ass:elliptic_main} (2) will be given as \cite[Theorem 1.4]{Smi22b}.

Neither the restriction to abelian varieties of dimension $1$ nor the restriction to the base field $\QQ$ are necessary to prove a result like Theorem \ref{thm:elliptic_Selmer_13}. However, the required technical conditions become more complicated when these restrictions are removed. For some examples with higher dimensional abelian varieties, see \cite[Section 3]{Smi22b}.

\end{rmk}

\subsection{Class groups and Cohen--Lenstra--Gerth heuristics}
The simplest case of our results for class groups gives the distribution of $2$-primary class torsion of imaginary quadratic fields. 
\begin{notat}
Given an imaginary quadratic field $K$, take $r_2(K) \ge r_4(K) \ge \dots $ to be the unique sequence of nonnegative integers satisfying
\[ \Cl K[2^{\infty}]  \,\cong \,(\Z/2\Z)^{r_2(K) - r_4(K)} \oplus (\Z/4\Z)^{r_4(K) - r_8(K)} \oplus \dots\]
and $\lim_{k \rightarrow \infty} r_{2^k}(K) = 0$.
\end{notat}

As before, we need some notation for our probability distribution. We have no alternating restriction this time, which will be consistent with the fact that every class group is finite.
\begin{defn}
\label{defn:PMat}
For $n \ge j \ge 0$, take
\[P^{\text{Mat}}(j \,|\, n)\]
to be the probability that a uniformly selected $n \times n$ matrix with entries in $\mathbb{F}_2$ has kernel of rank exactly $j$. 

We also use the notation
\begin{align*}
&P^{\text{Mat}}(j \,| \, \infty) = \lim_{n \rightarrow \infty} P^{\text{Mat}}(j\, | \,n).
\end{align*}
\end{defn}

\begin{thm}
\label{thm:main_Cl_quad}
Given any nonincreasing sequence
\[r_4 \ge r_8 \ge \dots \ge r_{2^k} \ge \dots \]
of nonnegative integers, we have
\begin{align*}
&\lim_{H \rightarrow \infty} \frac{\# \left\{d \in \Z^{> 0}\,:\,\, d < H \,\text{ and }\, r_{2^k}\left(\QQ\left(\sqrt{-d}\right)\right) = r_{2^k}\,\text{ for }\, k \ge 2\right\}}{H} \\
&\qquad =   P^{\textup{Mat}}(r_4\,|\,\infty) \cdot \prod_{k= 3}^{\infty} P^{\textup{Mat}}(r_{2^{k}}\, |\, r_{2^{k-1}}).
\end{align*}
\end{thm}
The distribution of sequences of $2^k$-class ranks is again given by a Markov chain, and we give a representation of this Markov chain in the right part of Figure \ref{fig:Markov}. The distribution for $4$-class ranks was calculated by Fouvry and Kl{\"u}ners in \cite{Fouv07}. Theorem \ref{thm:main_Cl_quad} is consistent with what is predicted by Gerth's extension of the Cohen-Lenstra heuristic for the distribution of class groups \cite{Gert84, CoLe84}. It is the third major result towards proving this heuristic for imaginary quadratic fields, after the result of Davenport-Heilbronn on $3$-torsion \cite{DaHe71} and the result of Fouvry and Kl{\"u}ners on $4$-class ranks.

We note that $\Cl \QQ(\sqrt{-d})[2]$ has unbounded average size by Gauss's genus theory, which is why we remove it from consideration above. A similar consideration explains our notation below.

\begin{notat}
\label{notat:gen_class_result}
Take $F$ to be a number field, take $\ell$ to be a rational prime, and take $K$ to be a degree $\ell$ Galois extension of $F$. Then $\Cl K$ is a $\Gal(K/F)$ module, and we take $(\Cl K[\ell^{\infty}])^{\Gal(K/F)}$ to be the submodule of $\Cl K[\ell^{\infty}]$ fixed by $\Gal(K/F)$. 

Take $\xi$ to be the image of $x$ in the quotient ring
\[\Z_{\ell}[x]\big/ \left(1 + x + \dots +x^{\ell - 1}\right).\]
Take $\langle \xi \rangle$ to be the multiplicative group generated by $\xi$. This group has order $\ell$, so we may choose some isomorphism from $\langle \xi \rangle$ to $\Gal(K/F)$. Under this isomorphism,
\[\Cl K[\ell^{\infty}]/ (\Cl K[\ell^{\infty}])^{\Gal(K/F)}\]
is a $\Z_{\ell}[\xi]$ module. Writing $\omega = \zeta - 1$, there is a unique sequence of nonnegative integers 
\[r_{\omega}(K) \ge r_{\omega^2}(K) \ge \dots \]
with limit zero for which there is some isomorphism
\[\Cl K[\ell^{\infty}]/ (\Cl K[\ell^{\infty}])^{\Gal(K/F)} \,\cong\, (R/\omega R)^{r_{\omega}(K) - r_{\omega^2}(K)} \oplus  (R/\omega^2R)^{r_{\omega^2}(K) - r_{\omega^3}(K)}\oplus\dots,\]
where $R$ is taken to be $\Z_{\ell}[\xi]$. This defines the sequence of $\omega^k$-class ranks of our field extension.
\end{notat}

For this more general situation, we will need more general notation for our distribution.
\begin{defn}
\label{defn:PMat}
Choose an integer $u$ and a rational prime $\ell$. Given nonnegative integers $n \ge j \ge 0$, and also supposing $n \ge u$, take
\[P^{\text{Mat}}_{u, \ell}(j \,|\, n)\]
to be the probability that a uniformly selected $(n- u) \times n$ matrix with entries in $\mathbb{F}_\ell$ has kernel of rank exactly $j$. In the case that $n <u$, we take this probability to be $0$.  A formula for this probability appears in \cite[Section 2]{Smi22b}.

We also use the notation
\begin{align*}
&P^{\text{Mat}}_{u, \ell}(j \,| \, \infty) = \lim_{n \rightarrow \infty} P^{\text{Mat}}_{u, \ell}(j\, | \,n).
\end{align*}
\end{defn}
\begin{thm}
\label{thm:Cl}
Take $F$ to be a number field with $r_1$ real embeddings and $r_2$ conjugate pairs of complex embeddings. Take $\ell$ to be a rational prime such that
\[\mu_{2\ell} \not\subset F.\]
If $\ell = 2$, take $r_1'$ to be an integer satisfying $0 \le r_1' \le r_1$. If $\ell > 2$, take $r_1' = r_1$. We define
\[u = - r_2 - r_1'.\]

For $H > 0$, define
\[X_{F, \ell, r_1'}(H)\, = \,\Big\{ K/F \text{ Gal. of deg. } \ell\,:\,\, |\Delta_K| \le H,\,\, K/F \text{ splits at exactly } r_1' \text{ real places}\Big\},\]
where $\Delta_K$ denotes the discriminant of $K/\QQ$.

Then, given  any nonincreasing sequence of integers
\[r_{\omega} \ge r_{\omega^2} \ge \dots \ge r_{\omega^k} \ge \dots,\]
we have
\begin{align*}
&\lim_{H \rightarrow \infty} \frac{\#\left\{K \in X_{F, \ell, r_1'}(H)\,:\,\, r_{\omega^k}(K) = r_{\omega^k}\,\text{ for }\, k \ge 1\right\}}{\# X_{F, \ell, r_1'}(H)} \\
&\qquad = P^{\textup{Mat}}_{u, \ell}(r_{\omega}\, |\, \infty) \cdot\prod_{k=2}^{\infty}P^{\textup{Mat}}_{u, \ell}(r_{\omega^k}\, |\, r_{\omega^{k-1}}).
\end{align*}
\end{thm} 
This theorem verifies Gerth's supplement to the Cohen--Lenstra heuristic in the sense codified by Wittman in \cite{Witt05}. Conditionally on GRH, this result was previously known for Galois extensions of $\QQ$ due to work of Koymans and Pagano \cite{Koym18}. This work used the method of \cite{Smi17} and built off base-case work of Klys that was also conditional on GRH \cite{Klys16}.

\subsection{An overview of the method}

\subsubsection{$2$-Selmer groups in quadratic twist families}
Given an abelian variety $A$ over a number field $F$, and given $d$ in $F^{\times}$, we find that the $2$-torsion subgroups of $A$ and $A^d$ are isomorphic over $F$. More generally, given a module $N$ acted on continuously by the absolute Galois group $G_F$ of $F$, and given a continuous homomorphism $\chi: G_F \to \pm 1$, we can define the quadratic twist $N^{\chi}$ of $N$, and we find that there is an equivariant isomorphism
\begin{equation}
\label{eq:2_guy}
N^{\chi}[2] \cong N[2]
\end{equation}
of the $2$-torsion submodules of these Galois modules.

The isomorphism \eqref{eq:2_guy} may be exploited to study the distribution of $2$-Selmer groups in quadratic twist families of a given Galois module. This was first done by Heath-Brown, who found the distribution of $2$-Selmer ranks in the quadratic twist family of the congruent number curve $A\colon y^2 = x^3 -x$ over $\QQ$ \cite{Heat94}. This method was adapted to the study of $4$-class ranks of families of quadratic fields by Fouvry and Kl{\"u}ners \cite{Fouv07}. Heath--Brown's result was later generalized to quadratic twist families of elliptic curves obeying Assumption \ref{ass:elliptic_main} (3) by Kane \cite{Kane13}.

Other previous results on the distribution of $2$-Selmer groups apply to less natural families of twists.
This includes the work of Gerth \cite{Gert84}, Swinnerton--Dyer \cite{Swin08}, and Klagsbrun, Mazur, and Rubin \cite{KMR14}. The families appearing in these three papers are quite different, but they share one crucial element: in each of the families, twists are built up from the sets of primes where they are ramified. This idea reappears in our work as the notion of a \emph{grid of twists}. Over $\QQ$, a grid is defined from a collection of disjoint sets of rational primes $X_1, \dots, X_r$, and consists of the squarefree integers in the set
\[\big\{p_1\cdot p_2 \cdot \dots \cdot p_r \,:\,\, (p_1, \dots, p_r) \in X \big\},\]
where we have taken the notation $X = \prod_{i \le r} X_i$.
One key step in our work is gridding, where we carve the set of twists up to a given height into grids where the distribution of $2$-Selmer ranks can be calculated.

\begin{rmk}
Note that $N[2]$ is the submodule of $N$ fixed under the automorphism $-1$. More generally, given an equivariant automorphism $\zeta$ of $N$, we may consider the submodule $N[\zeta - 1]$ of points in $N$ fixed by $\zeta$. If we then consider a twist $N^{\chi}$ of $N$ corresponding to some homomorphism $\chi$ in $\Hom(G_F, \langle \zeta \rangle)$, there will be an equivariant isomorphism.
\[N^{\chi}[\zeta - 1]  \cong N[\zeta - 1]\]
between the submodules of fixed points of $\zeta$. The analogue of the $2$-Selmer group of $N$ in this more general context is the $\zeta -1$-Selmer group of $N$, which we will also refer to as the fixed point Selmer group since it consists of cocycle classes valued in the fixed points of $\zeta$ on $N$. If $\zeta$ has order a power of $\ell$, we find that the $\ell^{\infty}$-Selmer groups in this twist family can often be controlled by similar methods to the $2^{\infty}$-Selmer groups in the quadratic twist family, so we expand the scope of this paper to include such examples.
\end{rmk}

\subsubsection{Moving to higher Selmer groups}
For simplicity at this point, we will focus on the case of quadratic twist families.

The key innovation of our work is that the isomorphism \eqref{eq:2_guy} has analogues for higher powers of $2$, and that these can be used to study higher Selmer groups. To be explicit, choose a sequence of homomorphisms $\chi_1, \chi_2, \dots$ in $\Hom(G_F, \pm 1)$. Then we find that the $G_F$-module $N^{\chi_1 \cdot \chi_2}[4]$ can be constructed as an explicit subquotient of
\[N^{\chi_1}[4] \oplus N^{\chi_2}[4] \oplus N[4],\]
that $N^{\chi_1 \cdot \chi_2 \cdot \chi_3}[8]$ can be constructed as an explicit subquotient of 
\[N^{\chi_1 \cdot \chi_2}[8] \oplus  N^{\chi_1 \cdot \chi_3}[8]  \oplus N^{\chi_2 \cdot \chi_3}[8]  \oplus N^{\chi_1}[8] \oplus N^{\chi_2}[8] \oplus N^{\chi_3}[8] \oplus N[8],\]
and so on. The work needed to prove these results can be found in Section \ref{sec:algcomb}, the only section of this paper entirely devoid of number theory. In particular, see Example \ref{ex:four_subquo}.

It should be clear that these relations can be applied fruitfully to grids of twists. Choose an abelian variety $A/\QQ$, and take $X =  \prod_{i \le r} X_i$ to be the grid of twists considered above. Suppose that 
\[\Sel^2 A^{d} \cong \Sel^2 A^{d'}\]
under the standard isomorphism of $A^d[2]$ and $A^{d'}[2]$ for any $d, d'$ in the grid. Then, supposing $k \le r$, and supposing some other technical assumptions are satisfied, we find that the $2^k$-Selmer groups of $A^d$ for every $d$ in the grid can be reconstructed from the $2^k$-Selmer groups over a certain sparse subset $Y$ of $X$. This can be proved using the method of Lemma \ref{lem:closure_noram}.

\subsubsection{The three parts of forcing equidistribution}
This observation about $2^k$-Selmer groups in grids still leaves us far away from proving an equidistribution result like Theorem \ref{thm:elliptic_Selmer_13}. Without further work, if the $2^k$-Selmer ranks of the $A^d$ have a lopsided distribution over $Y$, the same can be true for the whole grid. To prove an equidistribution result, we need to arrange it so that, however the $2^k$-Selmer groups behave over $Y$, the result will still be an equidistributed set of $2^k$-Selmer groups. The process of forcing equidistribution breaks into an algebraic part, an analytic part, and a combinatorial part.

The algebraic part is given as Theorem \ref{thm:TAR}. In this result, we show that the behavior of higher Selmer groups in a grid partially depends on a symbol $\symb{\ovp_a}{\ovp_b}$ defined for certain primes $\ovp_a, \ovp_b$ of $\overline{\QQ}$. The symbol encodes information about how $\ovp_a$ behaves at $\ovp_b$ relative to a fixed number field; we will define this object formally in Section \ref{sec:symbols}.

The analytic part is the bilinear equidistribution result for symbols appearing in Section \ref{sec:bilinear}. This is a result we end up using three times: in addition to its application towards forcing equidistribution of higher Selmer ranks, it also appears in the regridding of Section \ref{sec:regrid} and in our calculations for the distribution of fixed point Selmer ranks in \cite{Smi22b}. However, it is only the application towards forcing equidistribution that requires the full strength of this result, in the form of Corollary \ref{cor:an}.

Combining these two ideas gives us substantial control on how the Selmer groups over $X$ behave relative to the Selmer groups over $Y$. We need to engineer this control so that, no matter how the Selmer groups over $Y$ behave, the result is an approximately equidistributed collection of Selmer groups over the whole grid $X$. This is the role that the combinatorial result Proposition \ref{prop:bye_Ramsey} fills.

\subsection{The organization of this paper}
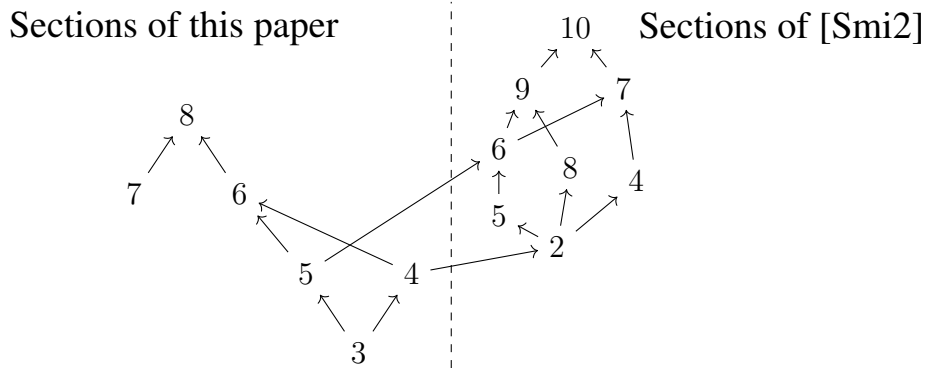
\begin{figure}
\begin{center}
\begin{tikzpicture}
 \tikzstyle{every path}=[->]      
\node  (A)    at (-10pt, 0pt)    {$3$};
\node  (B) at (10pt, 30pt) {$4$};
\node    (C) at (-30pt, 30pt) {$5$};
\node  (D) at (-55pt, 60pt) {$6$};
\node    (E) at (-95pt, 60pt) {$7$};
\node    (F) at (-75pt, 90pt) {$8$};
\node (G) at (65pt, 40pt) {$2$};
\node (H) at (95pt, 65pt) {$4$};
\node (I) at (43pt, 52pt) {$5$};
\node (J) at (43pt, 77pt) {$6$};
\node (K) at (90pt, 100pt) {$7$};
\node (L) at (70pt, 70pt) {$8$};
\node (M) at (52pt, 100pt){$9$};
\node (N) at (72pt, 123pt){$10$};

\node(O) at (25pt, -10pt){};
\node(P) at (25pt, 137pt){};

\node(Q) at (150pt, 123pt){\text{\large{Sections of \cite{Smi22b}}}};
\node(R) at (-80pt, 123pt){\text{\large{Sections of this paper}}};

\path

(A) edge (B)
(A) edge (C)
(B) edge (D)
(C) edge (D)
(D) edge (F)
(E) edge (F)
(B) edge (G)

(G) edge (H)
(G) edge (I)
(I) edge (J)
(H) edge (K)
(J) edge (K)
(G) edge (L)
(J) edge (M)
(C) edge (J)
(L) edge (M)
(M) edge (N)
(K) edge (N);
 \tikzstyle{every path}=[-]   
\draw[dashed] (O) -- (P);

\end{tikzpicture}
\end{center}
\caption{An approximate dependency diagram for the proof of \cite[Theorem 2.14]{Smi22b}. The dashed line separates the sections of this paper on the left from the sections of \cite{Smi22b} on the right.}
\label{fig:Leit}
\end{figure}

After going through some necessary notation in Section \ref{sec:notat}, we introduce the notions of classes, spins, and symbols of primes in Section \ref{sec:symbols}. These objects suffice to calculate the fixed point Selmer group of a twist coming from a grid, as we will show with Propositions \ref{prop:class_spin_symbol} and \ref{prop:class_spin_symbol_exact} after defining twistable modules, Selmer groups, and grids. Given a twist $\chi$ in the grid, we call the collection of twists in the grid whose classes, spins, and symbols all agree with those of $\chi$ the grid class of $\chi$.

We will control higher Selmer ranks by controlling certain pairings derived from the Cassels--Tate pairing on a specified grid class. We develop this theory in Section \ref{ssec:CT} and give a heuristic for how these pairings are distributed in Section \ref{ssec:heur}. This section ends with Theorem \ref{thm:main_higher}, the main result of this paper, which gives a sufficient set of conditions on the grid class for our heuristic to hold.

In Section \ref{sec:bilinear}, we prove an equidistribution result for symbols $\symb{\ovp}{\ovq}$ for pairs of primes $(\ovp, \ovq)$ taken from a large product $X_1 \times X_2$.

As mentioned above, our algebraic methods for manipulating higher Selmer groups requires a grid of twists with isomorphic fixed point Selmer groups. When proving Theorem \ref{thm:main_higher}, we run into the problem that the grid class of a twist is usually not a grid. Our solution to this is to carve the given grid class into smaller grids, a process we complete in Section \ref{sec:regrid}. This section also reduces the scope of Theorem \ref{thm:main_higher} to the behavior of a single matrix coefficient of the Cassels--Tate pairing.

In the self-contained Section \ref{sec:algcomb}, we introduce the algebra and combinatorics needed to mainpulate Selmer groups in grids of twists. We then apply this theory in Section \ref{sec:gov} to finish the proof of Theorem \ref{thm:main_higher}.

\subsection*{Acknowledgements}
This work was largely done while I was a graduate student at Harvard University. I would like to thank my advisors, Noam Elkies and Mark Kisin, for their support there. I also would like to thank  Dorian Goldfeld, Andrew Granville, Melanie Matchett Wood, and Shou-Wu Zhang for their encouragement and guidance from outside the Harvard mathematics department.

I want to acknowledge Dimitris Koukoulopoulos, Peter Koymans, Adam Morgan, Jesse Thorner, and David Yang for specific insights that had a major impact on this work.

I would like to thank the many people who have had comments on previous versions of this work, including Brandon Alberts, Robin Ammon, Alex Bartel, Magnus Carlson, Stephanie Chan, Brian Conrad, Andrea Gallese, Kenz Kallal, Seoyoung Kim, Zev Klagsbrun, Alison Miller, Djordjo Milovic, Sebastian Monnet, Evan O'Dorney, Carlo Pagano, Sun Woo Park, Ross Paterson, Manami Roy, Ye Tian, John Voight, Jiuya Wang, Ariel Weiss, Boya Wen, and Xiuwu Zhu.

I would like to thank George Boxer, Jordan Ellenberg, Wei Ho, Barry Mazur, Morgan Opie, Hector
Pasten, Bjorn Poonen, Karl Rubin, Peter Sarnak, Geoffrey Smith, Richard Taylor, Frank Thorne, Nicholas
Triantafillou, Ila Varma, Xinyi Yuan, and the participants in the virtual 2022 learning seminar for helpful discussions over the course of this project.

This research was partially conducted during the period the author served as a Clay Research Fellow. Previously, the author was supported in part by National Science Foundation grant DMS-2002011.

\section{Notation}
\label{sec:notat}
Fix an algebraic closure $\overline{\QQ}$ of $\QQ$. All the number fields appearing in this paper will be taken to lie inside this field.  For any integer $m \ge 1$, $\mu_m$ will denote the group of $m^{th}$ roots of unity inside $\overline{\QQ}$.

Given a characteristic-$0$ field $K$ with fixed algebraic closure $\overline{K}$, we take $G_K$ to be the absolute Galois group $\Gal(\overline{K}/K)$. In particular, for a number field $F$, we have $G_F = \Gal(\overline{\QQ}/F)$.

For us, the places of a given number field are its primes and its archimedean absolute value functions. For each place $v$ of a number field $F$, $F_v$ will denote the completion of $F$ at $v$.  We take $\mathcal{O}_F$ to be the ring of integers of $F$.

Given a prime $\ovp$ of $\overline{\QQ}$, we take $G_{F, \ovp}$ to be the associated decomposition group of $G_F$, $I_{F, \ovp}$ to be the associated inertia subgroup, and $I^{\text{w}}_{F, \ovp}$ to be the associated wild inertia subgroup. We similarly can define the decomposition group $G_{F, \overline{v}}$ for an archimedean place $\overline{v}$ of $\overline{F}$. 

Given a place $v$ of $F$, we will take $G_v$ to be equal to $G_{F, \overline{v}}$ for some place $\overline{v}$ of $\overline{F}$ over $v$, with $I_v$ defined similarly. We will use this notation in cases where the specific choice of $\overline{v}$ does not matter.

Given $F$ and a prime $\mfp$, we fix a Frobenius element $\FrobF{F}{\ovp}$ inside $G_{F, \ovp}$. This will be any element whose image in $G_{F, \ovp}/I_{F, \ovp}$ induces the Frobenius endomorphism on the residue field of $\overline{\QQ}_{\ovp}$. This definition determines the element up to an element of  $I_{F, \ovp}$.

Take $\widehat{\Z}(1)$ to be the inverse limit ${\varprojlim}_{m \ge 1} \,\mu_m$. Fix once and for all a topological generator $\overline{\zeta}$ of $\widehat{\Z}(1)$.

With this fixed, given $F$ and $\ovp$, we choose an element $\TineF{F}{\ovp}$ in $I_{F, \ovp}$ so that, for any  $\alpha \in \mathcal{O}_F$ divisible by $\mfp = \ovp \cap F$ but not by $\mfp^2$, and given any positive integer $m$ indivisible by $\mfp$, the ratio
\[\frac{(\TineF{F}{\ovp})( \alpha^{1/m})}{\alpha^{1/m}} \]
equals the image of $\overline{\zeta}$ in $\mu_m$.

With $\overline{\zeta}$ fixed, this definition determines $\Tine_F\, \ovp$ up to an element of $I^{\text{w}}_{F, \ovp}$. As suggested by the notation, $\Tine_F\, \ovp$ is a generator for the tame inertia group $I_{F, \ovp}/ I^{\text{w}}_{F,\ovp}$.

Given a finite discrete $G_F$-module $M$, we take $M(1)$ to be the Tate twist $M \otimes_{\widehat{\Z}} \widehat{\Z}(1)$ and $M(-1)$ to be the Tate twist
\[M(-1)  = \Hom_{\widehat{\Z}}\left(\widehat{\Z}(1),\, M\right).\]
We also define $M^{\vee}$ to be the twisted Pontryagin dual $\Hom_{\widehat{\Z}}(M, \QQ/\Z)(1)$.

Given a topological group $G$ and a discrete $G$-module $M$, we take $H^i(G, M)$ to be the continuous cochain cohomology group for $i \ge 0$, which we define as a quotient of the set of continuous $i$-cocycles $Z^i(G, M)$. The term $M^G$ will denote the set of invariants of $M$ under $G$, making it synonymous with $H^0(G, M)$, and the term $M_G$ will denote the set of coinvariants of $M$. Given $\sigma \in G$, we take $M^{\sigma}$ and $M_{\sigma}$  as alternative notation for $M^{\langle \sigma \rangle}$ and $M_{\langle \sigma \rangle}$.

\section{Ramification sections and symbols}
\label{sec:symbols}
Our first major result about $\omega$-Selmer groups is Proposition \ref{prop:class_spin_symbol}. This proposition generalizes the following result, which can be found in Monsky's appendix to \cite{Heat94}.
\begin{prop}
\label{prop:Monsky}
Take $A$ to be the elliptic curve over $\QQ$ with Weierstrass form $y^2 = x^3 - x$. Choose an integer $r \ge 1$ and two collections of $r$ distinct odd primes $p_1, \dots, p_r$ and $q_1, \dots, q_r$. Suppose
\begin{itemize}
\item The primes $p_i$ and $q_i$ are equal mod $8$ for all $i \le r$, and
\item We have an equality
\[\left(\frac{p_i}{p_j}\right) = \left(\frac{q_i}{q_j}\right)\]
of Legendre symbols for all $i < j \le r$.
\end{itemize}
Then, taking $d = p_1 \cdot \dots\cdot  p_r$ and $e = q_1 \cdot \dots\cdot  q_r$, the $2$-Selmer groups of the quadratic twists $A^d$ and $A^e$ are isomorphic.
\end{prop}
The proof of Proposition \ref{prop:Monsky} relies on the fact that $A[2]$ is isomorphic to $(\Z/2\Z)^2$ as a $G_{\QQ}$-module. To handle more complicated Galois modules in Proposition \ref{prop:class_spin_symbol}, we need to find an appropriate replacement for the Legendre symbols appearing above. The goal of this section is to develop the theory of this replacement for the Legendre symbol, with the Legendre symbol itself reappearing in Example \ref{ex:Legendre}. We will turn to the fundamental definitions and results for Selmer groups in the next section.

\begin{defn}
Given a number field $F$, a discrete $G_F$ module $M$, and a set of places $\Vplac$ of $F$ and its subfields, we define the set of cocycle classes unramified away from $\Vplac$ as the kernel
\[\msS_{M/F}(\Vplac) = \ker\left(H^1(G_F, M) \to \prod_{v \,\nmid\, \Vplac} H^1(I_v, M)\right).\]
\end{defn}
\begin{defn}
\label{defn:starting_tuple}
Choose a tuple $(K/F, \Vplac_0, e_0)$  consisting of
\begin{enumerate}
\item A Galois extension of number fields $K/F$,
\item A finite set of places $\Vplac_0$ of $F$, and
\item An integer $e_0 \ge 1$.
\end{enumerate}

We assume that $K$ contains $\mu_{e_0}$ and that $\Vplac_0$ contains all archimedean places, all primes dividing $e_0$, and all places where $K/F$ ramifies. We then call $(K/F, \Vplac_0, e_0)$ a \emph{starting tuple}. 

\end{defn}

\begin{defn}
\label{defn:unpacked}
Choose a starting tuple $(K/F, \Vplac_0, e_0)$. 
We take $\Mod(K/F, e_0)$ to be the full subcategory of the category of finite $\Gal(K/F)$-modules whose objects are the modules of exponent dividing $e_0$.

Given a prime $\ovp$ of $\ovQQ$ not dividing any prime in $\Vplac_0$,  and given $M$ in $\Mod(K/F, e_0)$, we can define a natural \emph{ramification-measuring} homomorphism
\begin{equation}
\label{eq:mfR}
\mfR_{\ovp, M}: H^1(G_F, M) \xrightarrow{\quad} M(-1)^{G_{F, \ovp}}
\end{equation}
taking $\phi \in H^1(G_F, M)$ to the map in $\Hom\big(\widehat{\Z}(1), M\big)$ sending $\overline{\zeta}$ to $\phi(\TineF{F}{\ovp})$. This definition does not depend on the choice of $\overline{\zeta}$. We need to check that this homomorphism is fixed by the action of $G_{F, \ovp}$, but this follows from the identity
\begin{equation}
\label{eq:frob_tine}
\FrobF{F}{\ovp} \cdot \TineF{F}{\ovp}\cdot  (\FrobF{F}{\ovp})^{-1} = \TineF{F}{\ovp}^{a}\quad\text{mod } I_{F, \ovp}^{e_0}\cdot I^w_{F, \ovp},
\end{equation}
where $a \in (\Z/e_0\Z)^{\times}$ is chosen so $\FrobF{F}{\ovp} (\zeta) = \zeta^a$ for all $\zeta$ in $\mu_{e_0}$.

Taking $\mfp = \ovp \cap F$, the maps $\mfR_{\ovp}$ fit an exact sequence
\begin{equation}
\label{eq:pre_unpacked}
0 \to \msS_{M/F}(\Vplac_0) \to \msS_{M/F}(\Vplac_0 \cup \{\mfp\}) \xrightarrow{\,\,\,\mfR_{\ovp}\,\,\,} M(-1)^{G_{F, \ovp}}
\end{equation}
for any $M$ in $\Mod(K/F, e_0)$. If the final map in this sequence is surjective for every $M$ in $\Mod(K/F, e_0)$ and every $\ovp$ not over $\Vplac_0$, we call $(K/F, \Vplac_0, e_0)$ an \emph{unpacked starting tuple}.
\end{defn}

Our definition of an unpacked starting tuple invokes Galois cohomology, but we can check that a tuple is unpacked using more elementary algebraic number theory. This starts with the following definition.
\begin{notat}
\label{notat:approx_gen}
Given a starting tuple $(K/F, \Vplac_0, e_0)$, choose a subfield $E$ of $K$ containing $F$ so that $\Gal(K/E)$ is cyclic, and choose a prime $\mfp$ of $E$ that is inert in $K/E$ and that is not over any place in $\Vplac_0$. Given $\alpha_{\mfp}$ in $E^{\times}$, we say that $\alpha_{\mfp}$ is an approximate generator for $\mfp$ if there is an integral ideal $\mfa$ of $E$ divisible only by primes over $\Vplac_0$ and a fractional ideal $\mfb$ of $E$ so that
\[(\alpha_{\mfp}) = \mfp \mfa \mfb^{e_0}.\]
If such an $\alpha_{\mfp}$ can be found for all $E$ and $\mfp$, we say that the tuple $(K/F, \Vplac_0, e_0)$ \emph{has approximate generators}.
\end{notat}

\begin{defn}
\label{defn:class}
Fix a starting tuple $(K/F, \Vplac_0, e_0)$. Given any set of places $\Vplac$ of $F$, define $K(\Vplac)$ to be the maximal abelian extension of $K$ of exponent dividing $e_0$ ramified only over places in $\Vplac$. Choose a prime $\ovp$ of $\ovQQ$ not over $\Vplac_0$, and take $E$ to be the minimal extension of $F$ so that $K/E$ is inert at $\ovp \cap E$. Take $E(\Vplac_0)$ to be the maximal abelian extension of $E$ of exponent dividing $e_0$ ramified only over places in $\Vplac_0$.

Then we define the \emph{class of $\ovp$} with respect to $(K/F, \Vplac_0, e_0)$ to be the set of all primes $\ovp'$ of $\ovQQ$ not over any prime of $\Vplac_0$ for which
\[\FrobF{F}{\ovp'} \equiv \FrobF{F}{\ovp}\,\text{ mod } G_{E(\Vplac_0)}.\]
We denote the class of $\ovp$ by $\class{\ovp}$.

We define an equivalence relation $\sim$ on $\Gal(K(\Vplac_0)/F)$ by saying $\sigma \sim \tau$ if there are primes $\ovp, \ovq$ in the same class so $\FrobF{F}{\ovp}$ restricts to $\sigma$ and $\FrobF{F}{\ovq}$ restricts to $\tau$. We then define $\mathscr{C}$ to be the quotient $\Gal(K(\Vplac_0)/F)/\sim$.
\end{defn}
\begin{rmk}
\label{rmk:class}
Given $M$ in $\Mod(K/F, e_0)$ and $\phi$ in $\msS_{M/F}(\Vplac_0)$, and given $\sigma$ and $\tau$ in the same class of $\Gal(K(\Vplac_0)/F)/\sim$, we see that the coinvariant groups $M_{\sigma}$ and $M_{\tau}$ are equal and that $\phi(\sigma)$ equals $\phi(\tau)$ in this group of coinvariants.
\end{rmk}

\begin{notat}
\label{notat:approx_gen_assign}
\label{notat:Eovp}
Fix a starting tuple $(K/F, \Vplac_0, e_0)$. Given a prime $\ovp$ of $\ovQQ$ not over $\Vplac_0$, we take $E(\ovp)$ to be the minimal extension of $F$ so that $K/E(\ovp)$ is inert at $\ovp \cap E(\ovp)$.

Suppose this tuple has approximate generators. For each prime $\ovp$ of $\ovQQ$ not over $\Vplac_0$, choose some approximate generator $\alpha_{\ovp} \in E(\ovp)^{\times}$ for $\ovp \cap E(\ovp)$. We call this collection an \emph{assignment of approximate generators} if, for every pair of primes $\ovp, \ovp'$ of $\ovQQ$ lying in the same class with respect to the starting tuple, the ratio $\alpha_{\ovp} \cdot \alpha_{\ovp'}^{-1}$ lies in $(E(\ovp)_w)^{e_0}$ for every place $w$ over a place in $\Vplac_0$. By class field theory and the definition of a class, we see that an assignment of approximate generators can be found so long as the starting tuple has approximate generators.
\end{notat}

\begin{prop}
\label{prop:approx_gen}
A given starting tuple $(K/F, \Vplac_0, e_0)$ is unpacked if and only if it has approximate generators.
\end{prop}
\begin{proof}
Choose a prime $\ovp$ of $\ovQQ$ not over $\Vplac_0$, and take $E = E(\ovp)$ to be the field defined in Notation \ref{notat:Eovp}. Take $M_0 = \Z[G_F] \otimes_{\Z[G_E]} \mu_{e_0}$, and take $m_0 \in M_0(-1)^{G_{F, \ovp}}$ to be the homomorphism taking $\overline{\zeta}$ to $[1] \otimes \zeta$, where $\zeta$ is the image of $\overline{\zeta}$ in $\mu_{e_0}$. Given any $M$ in $\text{Mod}(K/F, e_0)$ and $m$ in $M(-1)^{G_{F, \ovp}}$, we see there is a unique $G_F$-equivariant homomorphism $\rho: M_0 \to M$ so that the Tate twisted morphism $\rho(-1)$ takes $m_0$ to $m$. So we find that the final map in \eqref{eq:pre_unpacked} is surjective for all $M$ in this category if and only if there is some 
\[\phi_{\ovp} \in \msS_{M_0/F}(\Vplac_0 \cup \{\ovp \cap F\})\]
 with $\mfR_{\ovp}(\phi_{\ovp}) = m_0$. 

By Shapiro's lemma and Hilbert 90, we have isomorphisms
\[H^1(G_F, M_0) \cong H^1(G_E, \mu_{e_0}) \cong E^{\times}/(E^{\times})^{e_0},\]
and we find that a given $\phi_{\ovp}$ satisfies the above conditions if and only if any/every lift of the corresponding element in $E^{\times}/(E^{\times})^{e_0}$ to $E^{\times}$ is an approximate generator for $\ovp \cap E$. This gives the proposition.
\end{proof}
The proof of Proposition \ref{prop:approx_gen} gives a method for constructing sections to $\mfR_{\ovp, M}$ from approximate generators. These particular sections turn out to have some nice properties that we isolate with the following definition.

\begin{defn}
\label{defn:ram_sec}
Given a starting tuple $(K/F, \Vplac_0, e_0)$, a \emph{ramification section} $\mfB$ is a collection of homomorphisms 
\[\mfB_{\ovp, M}:M(-1)^{G_{F, \ovp}} \xrightarrow{\quad} \msS_M(\Vplac_0 \cup \{\mfp\}) \]
indexed by primes $\ovp$ of $\ovQQ$ not over $\Vplac_0$ and objects $M$ in $\Mod(K/F, e_0)$ so that
\begin{enumerate}
\item The map $\mfB_{\ovp, M}$ is a section for $\mfR_{\ovp, M}$, i.e. $\mfR_{\ovp, M} \circ \mfB_{\ovp, M}$ is the identity on $M(-1)^{G_{F, \ovp}}$ for every $\ovp$ and $M$.
\item For every fixed $\ovp$, the collection of maps $\mfB_{\ovp, M}$ defines a natural transformation.
\item Given any primes $\ovp$ and $\ovp'$ in the same class, any object $M$, and any $m$ in $M(-1)^{G_{F, \ovp}}$, the difference
\[\mfB_{\ovp, M}(m) - \mfB_{\ovp', M}(m)\]
is trivial when restricted to $G_v$ for all $v$ in $\Vplac_0$.
\end{enumerate}
\end{defn}

Choose an unpacked starting tuple $(K/F, \Vplac_0, e_0)$, and choose an assignment of approximate generators $(\alpha_{\ovp})_{\ovp}$. Choose a prime $\ovp$ of $\ovQQ$ not over $\Vplac_0$ and take $E = E(\ovp)$ as in Notation \ref{notat:Eovp}. Take $M_0 =  \Z[G_F] \otimes_{\Z[G_E]} \mu_{e_0}$, and take $m_0 \in M_0(-1)$ to be the homomorphism taking $\overline{\zeta}$ to $[1] \otimes \zeta$, where $\zeta$ is the image of $\overline{\zeta}$ in $\mu_{e_0}$. Using the method of the proof of Proposition \ref{prop:approx_gen}, we may associate $\alpha_{\ovp}$ to an element $\phi_{\ovp}$ in $H^1(G_F, M_0)$ with $\mfR_{\ovp}(\phi_{\ovp}) = m_0$. If $\ovp'$ is in $\class{\ovp}$, we may also define $\phi_{\ovp'}$ in $H^1(G_F, M_0)$. We find that the difference $\phi_{\ovp} - \phi_{\ovp'}$ is trivial at every place in $\Vplac_0$; indeed, this is equivalent to the compatibility of $\alpha_{\ovp}$ and $\alpha_{\ovp'}$ given as part of Notation \ref{notat:Eovp}.

Now, given $M$ in $\text{Mod}(K/F, e_0)$ and $m$ in $M(-1)^{G_{F, \ovp}}$, we  define $\mfB_{\ovp, M}(m)$ to equal $\rho(\phi_{\ovp})$, where $\rho$ is the unique equivariant homomorphism from $M_0$ to $M$ for which the Tate twisted morphism $\rho(-1)$ takes $m_0$ to $m$. This is a natural transformation: given a morphism $\beta: M \to M'$ in $\text{Mod}(K/F, e_0)$, we see that $\beta \circ \rho$ is the unique morphism whose Tate twist takes $m_0$ to $\beta(m)$, so
\[\mfB_{\ovp, M'}(\beta(m)) = \beta\left(\mfB_{\ovp, M}(m)\right).\]
A similar argument shows that $\mfB_{\ovp, M}$ is a homomorphism, and it is now straightforward to check that the collection of maps $\mfB_{\ovp}$ define a ramification section. By using the naturality condition, we can also check that any ramification section may be constructed in this way from some assignment of approximate generators.

We summarize the above work with the following proposition.

\begin{prop}
\label{prop:unpacked}
Given a starting tuple $(K/F, \Vplac_0, e_0)$, the following are equivalent:
\begin{enumerate}
\item There is a ramification section $\mfB$ defined with respect to the starting tuple.
\item The starting tuple is unpacked.
\item For every $M$ in $\Mod(K/F, e_0)$, we have
\[\ker\left(\msS_{M/F}(\Vplac_0) \to \prod_{v \in \Vplac_0} H^1(G_v, M) \right) \,=\, \ker\left(H^1(G_F, M) \to \prod_{v \text{ of } F } H^1(G_v, M) \right).\]
\item The starting tuple has approximate generators.
\end{enumerate}
\end{prop}
\begin{proof}
We see from the above work that (1), (2), and (4) are equivalent, so we just need to prove (3) is equivalent to the others. But, from Poitou--Tate duality, we see that the equality of kernels of condition (3) holds  for a given $M$ and prime $\ovp$ if and only if the map $\mfR_{\ovp, M^{\vee}}$. This suffices to show that (2) and (3) are equivalent.
\end{proof}

Our methods generally require that we work with unpacked starting tuples. Fortunately, this is not much of a burden.
\begin{prop}
\label{prop:unpack_it}
Choose a starting tuple $(K/F, \Vplac_0, e_0)$, and take $\Cl\, K$ to be the class group of $K$. Then there is a set of places $\Vplac$ of $F$ of cardinality at most $[K: F] + \log_2\left(\# \Cl\, K\right)$ so that $(K/F, \Vplac_0 \cup \Vplac, e_0)$ is an unpacked starting tuple.
\end{prop}
\begin{proof}
For $\sigma$ in $\Gal(K/F)$, choose a prime $\ovp_{\sigma}$ of $\ovQQ$ not over $\Vplac_0$ such that $\FrobF{F}{\ovp}$ projects to $\sigma$. Also choose a minimal set of generators $A$ for the abelian group $\Cl\,K/e_0 \Cl\, K$, and choose a prime $\mfp_a$ of $K$ whose ideal class projects to $a$ for each $a \in A$. Take
\[\Vplac = \big\{\ovp_{\sigma} \cap F\,:\,\, \sigma \in \Gal(K/F)\big\} \cup  \big\{\mfp_{a} \cap F\,:\,\, a \in A\big\}.\]
This set has cardinality bounded by $[K: F] + \log_2\left(\# \Cl\, K\right)$. To check that this choice gives an unpacked tuple, we use Condition (3) of Proposition \ref{prop:unpacked}. First, given an unramified element $\phi$ in $H^1(G_K,\frac{1}{e_0}\Z/\Z)$, if $\phi$ is trivial at $\FrobF{F}{(\mfp_a \cap F)}$ for each $a$ in $A$, we can use Artin reciprocity to conclude that $\phi$ is trivial. From the inflation-restriction exact sequence, we can conclude that
\begin{align*}
&\ker\left(\msS_{M/F}(\Vplac_0)  \to \prod_{v \in \Vplac_0 \cup \Vplac} H^1(G_v, M)\right)  \\
& \qquad =\,\,\ker\left(H^1(\Gal(K/F), M) \to \prod_{v \in \Vplac_0 \cup \Vplac} H^1(G_v, M)\right).
\end{align*}
From triviality at the $\ovp_{\sigma}$, we see that a cocycle class in the second kernel must be trivial on restriction to any cyclic subgroup of $\Gal(K/F)$. Since $\Vplac_0$ contains all places where $K/F$ is ramified, we thus find that any class in this second kernel has trivial restriction to $G_v$ for every place $v$ of $F$. So the conditions of Proposition \ref{prop:unpacked} hold for $(K/F, \Vplac_0 \cup \Vplac, e_0)$.
\end{proof}

For the rest of the section, we fix an unpacked starting tuple $(K/F, \Vplac_0, e_0)$ and an assignment of approximate generators $(\alpha_{\ovp})_{\ovp}$, and we take $\mfB$ to be the ramification section associated to $(\alpha_{\ovp})_{\ovp}$.
\begin{defn}
\label{defn:symbol}
Choose primes $\ovp, \ovq$ of $\ovQQ$ not over places of $\Vplac_0$. We define $E(\ovp)$ and $E(\ovq)$ as in Notation \ref{notat:Eovp}. For any $\tau$ in $G_F$, define the composite field
\[L_{\tau} = L_{\tau}(\ovp, \ovq)= E(\ovp) \cdot \tau E(\ovq),\]
take $\mfp_{\tau} = \ovp \cap L_{\tau}(\ovp, \ovq)$, and take $m_{\tau}$ to be the maximal divisor of $e_0$ such that $\mu_{m_{\tau}}$ is a subgroup of $L_{\tau}^{\times}$. Given any $\beta_1, \beta_2 \in L^{\times}_{\tau}$, we can consider the order $m_{\tau}$ Hilbert symbol $(\beta_1, \beta_2)_{m_{\tau}, \mfp_{\tau}}$ evaluated in the local field $(L_{\tau})_{\mfp_{\tau}}$, with our conventions as in \cite[Section V.3]{Neuk99}.

With this setup, we define the \emph{symbol} of $\ovp$ and $\ovq$ to be the function
\[\symb{\ovp}{\ovq}: G_F \to \mu_{e_0}\]
given by 
\[\symb{\ovp}{\ovq}(\tau) \, =\, \left(\alpha_{\ovp}, \,\tau \alpha_{\ovq}\right)_{m_{\tau}, \mfp_{\tau}}\,\,\text{ for } \,\tau \in G_F.\]

\end{defn}

\begin{ex}
\label{ex:Legendre}
Restricting to the case that $K = F = \QQ$ and $e_0 = 2$, we find that $(\QQ/\QQ, \{2, \infty\}, 2)$ is an unpacked starting tuple. One possible assignment of approximate generators takes a prime $\ovp$ of $\ovQQ$ not over $2$ to the unique positive generator for the ideal $\ovp \cap \QQ$. Given distinct odd rational primes $p$ and $q$ and any primes $\ovp$ and $\ovq$ in $\ovQQ$ dividing these primes, we find that
\[\symb{\ovp}{\ovq}(\tau) = \left(\frac{q}{p}\right)\quad\text{for all } \tau \in G_{\QQ},\]
 where the right term is a Legendre symbol.
\end{ex}

\begin{rmk}
\label{rmk:non_can}
The symbol $\symb{\ovp}{\ovq}$ depends on the choice of assignment of approximate generators, but this dependence is superficial. To explain this, take $\Vplac_0'$ to be a set of places of $F$ containing $\Vplac_0$, and take $\symb{\,\,}{}'$ to be a symbol defined with respect to some assignment of approximate generators for $(K/F, \Vplac_0', e_0)$. Define $\mathscr{C}$ as in Definition \ref{defn:class}. Then we find that there is some function
\[f: \mathscr{C} \times  \mathscr{C} \to \text{Map}(G_F,  \mu_{e_0})\]
such that, for any primes $\ovp, \ovq$ of $\ovQQ$ not over $\Vplac_0'$ with $\ovp \cap F \ne \ovq \cap F$, we have
\[\symb{\ovp}{\ovq} = \symb{\ovp}{\ovq}' \cdot f\left(\FrobF{F}{\ovp},\, \FrobF{F}{\ovq}\right).\]
That is, the ratio between these symbols is determined by the class of $\ovp$ and $\ovq$ over $(K/F, \Vplac_0', e_0)$. We can also determine the ratio between $\spin{\ovp}$ and $\spin{\ovp}'$ just from the class of $\ovp$.

It would be nice to package the same information in a more canonical invariant, but none of the alternatives the author has considered have been satisfactory. 
\end{rmk}

We list some basic properties of the symbol.
\begin{prop}
\label{prop:symbol_properties}
Take $\ovp, \ovq$ and $\ovp', \ovq'$ to be primes of $\ovQQ$ not over places in $\Vplac_0$ satisfying $\class{\ovp} = \class{\ovp'}$ and  $\class{\ovq} = \class{\ovq'}$. Define $E(\ovq)$ as before, and choose $\tau$ in $G_F$.
\begin{enumerate}
\item For $\sigma \in G_{E(\ovq)}$, we have $\symb{\ovp}{ \ovq}(\tau \sigma) = \symb{\ovp}{\ovq}(\tau)$.
\item For $\sigma \in G_{K}$, we have
\[\symb{\sigma \ovp}{\ovq}(\tau) = \symb{ \ovp}{\sigma \ovq}(\tau) = \symb{ \ovp}{\ovq}(\tau) .\]
\item Suppose $\ovp \cap K \ne \ovq \cap K$. Then
\[\tau \big(\symb{\ovp}{ \ovq}(1)\big) = \symb{\tau\ovp}{\ovq}(\tau). \]
\item We have 
\[ \symb{ \ovp}{\ovq}(\tau) \cdot \symb{ \ovp}{\tau\ovq}(1)^{-1} \,=\,   \symb{ \ovp}{\ovq'}(\tau) \cdot \symb{\ovp}{\tau \ovq'}(1)^{-1}.\]
\item Suppose $\ovp \cap K \ne  \ovq \cap K$ and $\ovp' \cap K \ne \ovq' \cap K$. Then
\[\symb{\ovp}{\ovq}(1)\cdot \symb{\ovq}{\ovp}(1)^{-1} \, =\, \symb{\ovp'}{\ovq'}(1)  \cdot \symb{\ovq'}{\ovp'}(1)^{-1}.\] 
\end{enumerate}
\end{prop}
\begin{proof}
Part (1) is immediate, and part (2) follows from the fact that $\sigma \ovp$ and $\sigma\ovq$ are in the classes of $\ovp$ and $\ovq$, respectively. For part (3), we note that $\tau$ gives an isomorphism between the local fields $(L_1(\ovp, \ovq))_{\mfp_1}$ and $(\tau L_1(\ovp, \ovq))_{\tau \mfp_1}$, so we have an identity of Hilbert symbols
\[\tau \left( \alpha_{\ovp},\, \alpha_{\ovq}\right)_{m_1, \mfp_1} = \left( \tau \alpha_{\ovp},\, \tau \alpha_{\ovq}\right)_{m_1,\tau  \mfp_1}.\]
The part follows since $\tau \alpha_{\ovq}$ and $\alpha_{\tau \ovp}^{-1} \tau \alpha_{\ovp}$ are both unramified at $\tau\mfp_1$. Part (4) is similar to (2), and part (5) follows from Hilbert reciprocity.
\end{proof}

In particular, we see that a symbol $\symb{\ovp}{\ovq}$ can be determined from the classes $\class{\ovp}$ and $\class{\ovq}$ together with the value of the symbol on any set of representatives for the collection of double cosets $G_{E(\ovp)} \backslash G_F / G_{E(\ovq)}$. However, this is as far as we can push it, as we have the following proposition.

\begin{prop}
\label{prop:chinese}
Choose primes $\ovp_0$ and $\ovq$ of $\ovQQ$ not over places in $\Vplac_0$. Choose a set of representatives $B $ in $G_F$ for $G_{E(\ovp_0)} \backslash G_F / G_{E(\ovq)}$, and choose any $\zeta_{\tau} \in \mu_{m_{\tau}}$ for each $\tau \in B$, where $m_{\tau}$ is defined as in Definition \ref{defn:symbol}. 

Then there is a prime $\ovp$ in the class of $\ovp_0$ not dividing $\ovq \cap F$ so
\[[\ovp, \ovq](\tau) = \zeta_{\tau} \quad\text{for all }\, \tau \in B.\]
\end{prop}
We will prove this in Section \ref{ssec:symb_proofs}.

\begin{notat}
\label{notat:symb_set}
Given classes $\class{\ovp_0}$ and $\class{\ovq_0}$, we take $\symb{\class{\ovp_0}}{\class{\ovq_0}}$ to be the set of all functions of the form $\symb{\ovp}{\ovq}$ for some $\ovp$ in $\class{\ovp_0}$ and $\ovq$ in $\class{\ovq_0}$ with $\ovp \cap F \ne \ovq \cap F$. If we take $B$ and $m_{\tau}$ as in Proposition \ref{prop:chinese}, we see that this set has size
\[\# \symb{\class{\ovp_0}}{\class{\ovq_0}} = \prod_{\tau \in B} m_{\tau}.\]
\end{notat}

Our main interest in symbols lies in the fact that they may be used to encode local behavior of ramification sections, as per the following proposition.

\begin{prop}
\label{prop:same_symbs}
Choose primes $\ovp, \ovq$ of $\ovQQ$ not over $\Vplac_0$, and choose primes $\ovp'$ in $\class{\ovp}$ and $\ovq'$ in $\class{\ovq'}$. Suppose $\ovp \cap F \ne \ovq \cap F$ and $\ovp' \cap F \ne \ovq \cap F$. Then the following conditions are equivalent:
\begin{enumerate}
\item There is an identity of symbols
\[\symb{\ovp}{\ovq} = \symb{\ovp'}{\ovq'}.\]
\item  For every $M$ in $\Mod(K/F, e_0)$ and $m$ in $M(-1)^{G_{F, \ovq}}$, 
\[\mfB_{\ovq, M}(m)\left(\FrobF{F}{\ovp}\right) \, =\, \mfB_{\ovq', M}(m)\left(\FrobF{F}{\ovp'}\right)\,\,\text{ in }\,\, M_{G_{F, \ovp}}.\]
\end{enumerate}
\end{prop}
\begin{proof}
This is an immediate consequence of Proposition \ref{prop:eval_at_Frob}, which we will prove in Section \ref{ssec:symb_proofs}.
\end{proof}

The symbol $\symb{\ovp}{\ovq}$ can be interpreted as a description of how $\ovp \cap F$ behaves at the prime $\ovq \cap F$ over the extension $K/F$. Given this interpretation, it is perhaps counterintuitive that the symbol $\symb{\ovp}{\ovp}$ should encode anything at all.  But while it is true that $\symb{\ovp}{\ovp}(1)$ is fully determined by the class of $\ovp$, as can be shown by an application of Hilbert reciprocity, this is not necessarily true for $\symb{\ovp}{\ovp}(\tau)$ for $\tau$ outside $G_{E(\ovp)}$. Intuitively, this can be explained by saying this symbol describes how $\ovp \cap K$ behaves at $\tau  \ovp \cap K$, which is a more reasonable thing to describe than how $\ovp \cap F$ behaves at itself.

The symbol $\symb{\ovp}{\ovp}$ is a generalization of the notion of the spin of a prime ideal, a definition that was first introduced in \cite{FIMR13} and studied in more depth in \cite{KoMi21}. Following this work, we call $\symb{\ovp}{\ovp}$ the \emph{spin} of $\ovp$. As per Remark \ref{rmk:non_can}, the dependence of this definition on a choice of assignment of approximate generators, while irritating, is ultimately superficial.

The analogue of Proposition \ref{prop:same_symbs} for spins is the following:
\begin{prop}
\label{prop:same_spins}
Choose primes $\ovp$ and $\ovp'$ of $\ovQQ$ not over any place in $\Vplac_0$. We assume that $\class{\ovp} = \class{\ovp'}$. Then the following conditions are equivalent:
\begin{enumerate}
\item There is an identity of spins
\[ \spin{\ovp} = \spin{\ovp'}.\]
\item For all $M$ in $\Mod(K/F, e_0)$ and $m$ in $M(-1)^{G_{F, \ovp}}$,
\[ \mfB_{\ovp, M}(m)\left(\FrobF{F}{\ovp}\right) \,=\,\mfB_{\ovp', M}(m)\left(\FrobF{F}{\ovp'}\right) \,\,\text{ in }\, (M/\textup{im } m)_{G_{F, \ovp}}.\]
Here, $\textup{im } m$ denotes the image of $m \in M(-1)$ in $M$.
\end{enumerate}
\end{prop}
\begin{proof}
This is also an immediate consequence of Proposition \ref{prop:eval_at_Frob}.
\end{proof}

\subsection{The double coset formula}
\label{ssec:symb_proofs}
This subsection gives an explicit method for computing terms of the form $\mfB_{\ovq, M}(m)\left(\FrobF{F}{\ovp}\right)$. This will be used heavily in \cite{Smi22b}, but is only used in this first part to prove Propositions \ref{prop:same_symbs}, \ref{prop:same_spins}, and \ref{prop:chinese}, so the reader might consider continuing to Section \ref{sec:selmer} at this point.

It is convenient to introduce a slight variant of the symbols defined above.
\begin{defn}
\label{defn:alt_symb}
Fix an unpacked starting tuple $(K/F, \Vplac_0, e_0)$ and an assignment of approximate generators $(\alpha_{\ovp})_{\ovp}$ for this tuple, and take $\mfB$ to be the ramification section associated to $(\alpha_{\ovp})_{\ovp}$.

Given a prime $\ovq$ of $\ovQQ$, take $E(\ovq)$ as in Notation \ref{notat:Eovp}, and take $\psi_{\ovq}$ to be the image of $\alpha_{\ovq}$ under the natural isomorphism
\[E(\ovq)^{\times}/(E(\ovq)^{\times})^{e_0} \isoarrow H^1(G_{E(\ovq)},\, \mu_{e_0}).\]
We may evaluate $\psi_{\ovq}$ at an element $\sigma$ of $G_{E(\ovq)}$; after accounting for coboundaries, this will be a well-defined element in $(\mu_{e_0})_{\sigma}$. With this in mind, given primes $\ovp, \ovq$ of $\ovQQ$ not over $\Vplac_0$, we may define an \emph{alternative symbol} by
\[\symb{\ovp}{\ovq}' = \psi_{\ovq}\left(\FrobF{E(\ovq)}{\ovp}\right) \in (\mu_{e_0})_{G_{E(\ovp) \cdot E(\ovq)}}.\]

Take $m$ to be the maximal divisor of $e_0$ so $\mu_m$ lies in the composite field $E(\ovp)\cdot E(\ovq)$. Note that there is an isomorphism
\[(\mu_{e_0})_{G_{E(\ovp) \cdot E(\ovq)}} \isoarrow \mu_m\]
given by raising to the $e_0/m$ power. Using this isomorphism, the identity
\[\symb{\ovp}{\ovq}'  =  \symb{\ovp}{\ovq}(1)^{m/e_0} \quad\text{if }\, \ovp \cap K \ne \ovq \cap K\]
between the alternative symbol and the original symbol follows from the definitions.

If $\ovp \cap K$ equals $\ovq \cap K$, the alternative symbol $\symb{\ovp}{\ovq}'$ depends on the choice of $\Frob_{E(\ovq)}{\ovp}$. As such, it encodes no information about the primes $\ovp$ and $\ovq$. Such symbols appear only transiently in our calculations.
\end{defn}

Choose primes $\ovp, \ovq$ of $\ovQQ$ not over $\Vplac_0$. For convenience, take $E = E(\ovp)$ and $L = E(\ovq)$, and take $B$ to be a set of representatives for the collection of double cosets $G_{L}\backslash G_F / G_E$. If $\ovp \cap F = \ovq \cap F$, we will assume that $\ovp = \ovq$ and that $B$ contains $1$.

\begin{prop}
\label{prop:eval_at_Frob}
Choose $M$ in $\textup{Mod}(K/F, e_0)$, and choose $m \in M(-1)^{G_L}$. Then, in the group of coinvariants $M_{G_E}$, we have
\begin{equation}
\label{eq:eval_at_Frob}
\mfB_{\ovq, M}(m)(\FrobF{F}{\ovp}) = \sum_{\tau \in B} \tau^{-1} \left(m\left(\symb{\tau \ovp}{\ovq}'\right)\right).
\end{equation}
Here, we evaluate $m$ at $\symb{\tau\ovq}{\ovp}'$ by evaluating at any element in $\widehat{\Z}(1)$ whose image in $\mu_{e_0}$ maps to $\symb{\tau \ovp}{\ovq}'$ in $(\mu_{e_0})_{G_{\tau E + L}}$.
\end{prop}

\begin{proof}
Since ramification sections are natural transformations, it suffices to consider the case where $M =  \Z[G_F] \otimes_{\Z[G_L]} \mu_{e_0}$ and $m$ is the map taking $\overline{\zeta}$ to $[1] \otimes (\overline{\zeta})_{e_0}$, where $(\overline{\zeta})_{e_0}$ is the image of $\overline{\zeta}$ in $\mu_{e_0}$. Shapiro's isomorphism gives an isomorphism between $H^1(G_E, \mu_{e_0})$ and $H^1(G_F, M)$, and we defined $\mfB_{\ovq, M}(m)$ to correspond to the map $\psi_{\ovq}$ under this isomorphism, where $\psi_{\ovq}$ is defined as in Definition \ref{defn:alt_symb}.

We can apply the double coset formula  \cite[Proposition I.5.6]{Neuk08} to calculate
\[\res_{G_E} \mfB_{\ovq, M}(m) =  \sum_{\tau \in B} \cores_{G_E}^{G_{\tau^{-1} L} \cap G_E}\circ \tau^{-1} \circ \res^{G_L}_{G_{L} \cap G_{\tau E}}([1] \otimes \psi_{\ovq}),\]
with $\res$ and $\cores$ denoting restriction and corestriction. The group $G_{\tau^{-1} E \cdot L}$ is a normal subgroup in $G_L$, and the quotient $G_L/G_{\tau^{-1} E \cdot L}$ is cyclic. Corestriction behaves well for such extensions: if $G$ is a profinite subgroup and $H$ is an open subgroup so $G/H$ is cyclic, and if $\phi$ lies in $H^1(H, N)$ for a $G$ module $N$, we have
\[\cores_G^H \phi(\sigma) = \phi\left(\sigma^{[G: H]}\right) \quad\text{in }\, N_{\sigma}.\]
Applying this and the previous identity at $\FrobF{F}{\ovp}$ gives
\[ \mfB_{\ovq, M}(m)  (\FrobF{F}{\ovp}) \equiv \sum_{\tau \in B}  [\tau^{-1}] \otimes \psi_{\ovq}\left(\tau \cdot \FrobF{\tau^{-1} L}{\ovp} \cdot \tau^{-1}\right)\, \text{ mod }\, \left(\FrobF{F}{\ovp} -1\right) M,\]
which gives \eqref{eq:eval_at_Frob}.
\end{proof}

To prove Proposition \ref{prop:chinese}, the following result extending a special case of Proposition \ref{prop:same_symbs} will be useful.

\begin{prop}
\label{prop:same_symbs_old}
Choose primes $\ovp$, $\ovp'$, and $\ovq$ of $\ovQQ$. We assume none of these primes are over places in $\Vplac_0$, that neither $\ovp \cap F$ nor $\ovp' \cap F$ equals $\ovq \cap F$, and that $\class{\ovp} = \class{\ovp'}$. Take $E = E(\ovp)$ as in Notation \ref{notat:Eovp}.

Then the following conditions are equivalent:
\begin{enumerate}
\item There is an identity of symbols
\[\symb{\ovp}{\ovq} = \symb{\ovp'}{\ovq}.\]
\item For all $M$ in $\Mod(K/F, e_0)$ and $m$ in $M(-1)^{G_{F, \ovp}}$, the difference 
\[\mfB_{\ovp, M}(m) - \mfB_{\ovp', M}(m)\]
has trivial restriction to $G_{F, \ovq}$.
\item For all $M$ in $\Mod(K/F, e_0)$ and $m$ in $M(-1)^{G_{F, \ovq}}$, 
\[\mfB_{\ovq, M}(m)\left(\FrobF{F}{\ovp}\right) \, =\, \mfB_{\ovq, M}(m)\left(\FrobF{F}{\ovp'}\right)\,\,\text{ in }\,\, M_{G_{F, \ovp}}.\]
\item Take $\mfq$ to be $\ovq \,\cap K$, and take $E(\Vplac_0 \cup \{\mfq\})$ to be the maximal abelian extension of $K$ of exponent dividing $e_0$ that is ramified only over places in $\Vplac_0 \cup \{\mfq\}$. Then
\[\FrobF{F}{\ovp} \equiv \FrobF{F}{\ovp'} \quad\textup{mod }\, G_{E(\Vplac_0 \cup \{\mfq\})}.\]
\end{enumerate}

\end{prop}
\begin{proof}
The equivalence of (1) and (2) follows from Proposition \ref{prop:same_symbs}.

The equivalence of condition (2) and condition (3) follows from Poitou--Tate duality. Specifically, we find that the condition of (2) holds for a given $M$ if and only if the condition of (3) holds for $M^{\vee}$.

To show that  (3) implies (4), take $M = \Gal(E(\Vplac_0 \cup \{\mfq\})/E(\Vplac_0))$. Considered with the conjugation action, this is an object in $\Mod(K/F, e_0)$. Take $m \in M(-1)$ to be the element taking $\overline{\zeta}$ to the image of $\TineF{F}{\ovq}$ in this Galois group. This element is invariant under $G_{F, \ovq}$, as can be seen from \eqref{eq:frob_tine}. Then $\mfB_{\ovq, M}(m)$ is a class of cocycles
\[\mfB_{\ovq, M}(m):\Gal(K(\Vplac_0 \cup\{\mfq\})/F) \,\xrightarrow{\quad}\, \Gal(K(\Vplac_0 \cup \{\mfq\})/K(\Vplac_0))\]
that all take $\TineF{F}{\ovq}$ to itself. The cocycle condition forces $\tau \TineF{F}{\ovq} \tau^{-1}$ to map to itself for all $\tau \in G_F$; since these elements generate $M$, we see that any cocycle in this class restricts to the identity on $M$. Then condition (3) and the cocycle condition imply that $\FrobF{F}{\ovp} \left(\FrobF{F}{\ovp'}\right)^{-1}$ maps to $0$ in $M_{G_{F, \ovp}} = M$, which is equivalent to (4).

Now suppose (4) holds and we wish to show (1). From Artin reciprocity applied to $E(\Vplac_0 \cup \{\mfq\})/E$, we find that  the approximate generators $\alpha_{\ovp}$ and $\alpha_{\ovp'}$ can be rechosen so their ratio is locally an $e_0^{th}$ power at every place over $\Vplac_0 \cup \{\mfq\}$. From condition (3) of Proposition \ref{prop:unpacked}, we thus have that $\alpha_{\ovp}/\alpha_{\ovp'}$ is locally an $e_0^{th}$ power at every place dividing $\mfq$, and this condition implies (1).
\end{proof}

\begin{proof}[Proof of Proposition \ref{prop:chinese}]
Take $E = E(\ovp_0)$, take $\mfq = \ovq \cap F$, and define $K(\Vplac_0)$, $E(\Vplac_0)$, and $E(\Vplac_0 \cup \{\mfq\})$ as in Definition  \ref{defn:class}. From Shapiro's lemma, we know that $(K/E, \Vplac_0', e_0)$ is unpacked, where $\Vplac_0'$ is the set of places above $\Vplac_0$. Note that the map taking $\tau$ to $\tau \ovq \cap E$ is a bijection between $B$ and the primes of $E$ dividing $\ovq \cap F$.  Taking $C = \tfrac{1}{e_0}\Z/\Z$, we find that the ramificaiton-measuring homomorphism defines an isomorphism
\[\msS_{C /E}\left(\Vplac_0' \cup \{\tau \ovq \cap E\,:\,\, \tau \in B\}\right)/ \msS_{C /E}\left(\Vplac_0'\right) \isoarrow \prod_{\tau \in B} C(-1)^{G_{E, \tau \ovq}}.\]
From this isomorphism, we get the identity
\[\big[E(\Vplac_0 \cup \{\mfq\})\,:\,\, E(\Vplac_0)\big] = \prod_{\tau \in B} m_{\tau},\]
where the $m_{\tau}$ are defined as in Definition \ref{defn:symbol}. The result then follows from Proposition \ref{prop:same_symbs_old} and the Chebotarev density theorem.
\end{proof}

\section{Selmer groups in twist families}
\label{sec:selmer}

\begin{defn}
\label{defn:twistable}
Fix a rational prime $\ell$ and a positive integer $k_0$. Taking $\Z_{\ell}[x]$ to be the polynomial ring over $\Z_{\ell}$, we take $\xi$ to be the image of $x$ in the quotient ring
\[\Z_{\ell}[x]\Big/\left(1 + x^{\ell^{k_0 - 1}} + x^{2 \cdot \ell^{k_0 - 1}} + \dots + x^{(\ell - 1) \cdot \ell^{k_0 - 1}}\right).\]
We will write this quotient ring as $\Z_{\ell}[\xi]$ and will endow it with the trivial $G_{\QQ}$ action. Take $\omega = \xi - 1$, and note that $\omega$ is a uniformizer for the discrete valuation ring $\Z_{\ell}[\xi]$. Finally, write $\FFF = \langle \xi\rangle$ for the multiplicative group generated by $\xi$. This is a cyclic group of order $\ell^{k_0}$.

Now fix a number field $F$. A \emph{twistable module} consists of
\begin{itemize}
\item A discrete topological $\Z_{\ell}[\xi]$-module $N$ isomorphic as a topological group to some positive power of $\QQ_{\ell}/\Z_{\ell}$ and
\item A continuous action of $G_F$ on $N$ commuting with $\xi$ whose restriction to $I_v$ is trivial for all but finitely many places $v$ of $F$.
\end{itemize}

Given $\chi$ in $\Hom_{\text{cont}}(G_F, \langle \xi\rangle )$ and a twistable module $N$, we define the twist $N^{\chi}$ as a twistable module isomorphic to $N$ under a (typically non-equivariant) isomorphism
\[\beta_{\chi}: N^{\chi} \to N\]
of topological $\Z_{\ell}[\xi]$ modules, with the action of $G_F$ on $N^{\chi}$ defined so
\[\beta_{\chi}(\sigma n) = \chi(\sigma) \sigma \beta_{\chi}(n)\]
for all $\sigma \in G_F$ and $n \in N^{\chi}$.
\end{defn}

\begin{defn}
Take $N$ to be a twistable module defined over $F$ with respect to $\xi $. Given a finite Galois extension $K/F$ and a finite set of places $\Vplac_0$ of $F$, we say that $N$ is unpacked by $(K/F, \Vplac_0)$ if $(K/F, \Vplac_0, \#\FFF)$ is an unpacked starting tuple, if $G_K$ acts trivially on $N[\omega]$, and if $I_v$ acts trivially on $N$ for $v$ outside $\Vplac_0$.
\end{defn}

\begin{defn}
\label{defn:Sel}
Choose a twistable module $N$ over $F$, and choose $(K/F, \Vplac_0)$ unpacking $N$. For $v \in \Vplac_0$ and $\chi \in \Hom_{\text{cont}}(G_v,\FFF)$, fix an $\ell$-divisible subgroup $W_v(\chi)$  of $H^1(G_v, N^{\chi})$ that is fixed under the automorphism $\xi$. We call the collection of groups $\left(W_v(\chi)\right)_{v, \chi}$ a \emph{set of local conditions} for $N$ at the places in $\Vplac_0$.

For $v$ a place of $F$ outside $\Vplac_0$ and $\chi \in \Hom_{\text{cont}}(G_v, \,\FFF)$, we define a subgroup 
\[W_v(\chi) \subseteq H^1(G_v, N^{\chi})\]
to equal $H^1_{\textup{ur}}(G_v, N^{\chi})$ in the case that $\chi$ is trivial on $I_v$, and to equal $0$ in the case that $\chi$ is nontrivial on $I_v$. We call these the \emph{automatic local conditions}.

With this set, given $\chi \in \Hom_{\text{cont}}(G_F,\, \FFF)$, we define the \emph{Selmer group} of $N^{\chi}$ by
\[\Sel\left(N^{\chi}, (W_v)_{v \in \Vplac_0}\right) = \ker\left(H^1(G_F, N^{\chi}) \to \prod_{v \text{ of } F} H^1(G_v, N^{\chi})\Big/W_v\left(\res_{G_v\,} \chi\right)\right).\]
We will write this group as $\Sel\,N^{\chi}$ if the local conditions at the places in $\Vplac_0$ are clear.

For $k \ge 1$, we take $\Sel^{\omega^k} N^{\chi}$ to be the preimage of $\Sel\, N^{\chi}$ in $H^1\left(G_F, N^{\chi}\left[\omega^k\right]\right)$. Given $j \le k$, we take $\omega^j \Sel^{\omega^k} N^{\chi}$ to be the image of $\Sel^{\omega^k} N^{\chi}$ in $\Sel^{\omega^{k-j}} N^{\chi}$ under the projection
\[\omega^j\colon N^{\chi}\left[\omega^k\right] \to N^{\chi}\left[\omega^{k-j}\right].\]
We note that $\omega^{k-1} \Sel^{\omega^k} N^{\chi}$ always contains the image of $H^0(G_F, N^{\chi}[\omega])$ under the connecting map corresponding to
\[0 \to N^{\chi}[\omega] \to N^{\chi}\left[\omega^2\right] \xrightarrow{\,\,\cdot \omega\,\,} N^{\chi}[\omega] \to 0.\]
We call this the \emph{rational torsion portion} of the Selmer group of $N^{\chi}$. For each $k \ge 1$, we define the $\omega^k$-Selmer rank of $N^{\chi}$ by
\[r_{\omega^k}(N^{\chi}) = \dim \omega^{k-1}\Sel^{\omega^k}N^{\chi}\Big/\text{im} \,H^0(G_F, N^{\chi}[\omega]).\]
\end{defn}

With these definitions set, we now show that Theorems \ref{thm:elliptic_Selmer_13} and \ref{thm:Cl} reduce to theorems about the distribution of Selmer groups in twist families of twistable modules.
\begin{ex}
\label{ex:abelian_Selmerable}
Take $A/F$ to be an abelian variety of positive dimension over a number field. Given a continuous homomorphism $\chi \colon G_F \to \pm 1$, take $A^{\chi}/F$ to be the corresponding quadratic twist. The torsion submodule
\[N = A[2^{\infty}]\]
is then a twistable module, and we see that $N^{\chi} = A^{\chi}[2^{\infty}]$. This twistable module is decorated with the local conditions
\[W_v(\chi) = \ker\big(H^1(G_v, N) \to H^1(G_v, A)\big)\]
for each place $v$ of $F$. So long as $\Vplac_0$ contains the primes of bad reduction of $A$, the archimedean primes, and the primes dividing $2$, these local conditions agree with the automatic local conditions for $v$ outside $\Vplac_0$.

The usual $2^{\infty}$-Selmer group on $A^{\chi}$ is then equal to the $2^{\infty}$-Selmer group of $N^{\chi}$, and $r_{2^k}(N^{\chi})$ equals the $2^k$-Selmer rank of $A^{\chi}$ over $F$ as defined in Section \ref{ssec:elliptic_curves}.
\end{ex}

\begin{ex}
\label{ex:class_Selmerable}
Choose a rational prime $\ell$, and define $\Z_{\ell}[\xi]$ as in Definition \ref{defn:twistable} with $k_0 = 1$. Consider
\[N = \QQ_{\ell}[\xi]/\Z_{\ell}[\xi]\]
endowed with the trivial $G_F$-action. This is a $\Z_{\ell}[\xi]$-module and hence a twistable module.

Choose  a cyclic degree $\ell$ extension $L/F$ and take $\chi: \Gal(L/F) \to \langle \xi \rangle$ to be a nontrivial character. We endow $N^{\chi}$ with local conditions
\[W_v(\chi) = \begin{cases} \text{inf}_{\,G_v}^{\,G_v/I_v}\left(H^1(G_v/I_v, N^{\chi})\right) &\text{ if } L/F \text{ is unramified at } v \\ 0 &\text{ otherwise.}\end{cases}\]
Then, from \cite[Proposition 7.4]{MS21b}, we find there is an injection
\begin{equation}
\label{eq:MS2_inj}
 \left(\Cl^* L\big/ (\Cl^* L)^{\Gal(L/F)}\right)[\ell^{\infty}] \hookrightarrow \Sel\, N^{\chi}
\end{equation}
of $\Z_{\ell}[\xi]$ modules, where $\Cl^* L = \Hom(\Cl\, L, \QQ/\Z)$ is the dual class group of $L$, and where $\xi$ acts on $\Cl\, L$ by $\chi^{-1}(\xi)$. The map \eqref{eq:MS2_inj} is defined by using the Shapiro isomorphism to pass from $\Cl^* L[\ell^{\infty}]$ to the Selmer group of $(\QQ_{\ell}/\Z_{\ell})[\Gal(L/F)]$ and then mapping to $\Sel \,N$ using the natural surjection from $ (\QQ_{\ell}/\Z_{\ell})[\Gal(L/F)]$ to $N$. Note that we have an equivariant perfect pairing
\[\Cl^* L\big/ (\Cl^* L)^{\Gal(L/F)} \times \Cl\, L\big/ (\Cl\, L)^{\Gal(L/F)} \to \QQ/\Z\]
given by $(x, y) \mapsto x \cdot (1 - \sigma) y$, where $\cdot$ denotes the evaluation pairing, so we find that $\Cl^* L\big/ (\Cl^* L)^{\Gal(L/F)}$ and $\Cl\, L\big/ (\Cl\, L)^{\Gal(L/F)}$ are (typically non-canonically) isomorphic as $\Z_{\ell}[\xi]$ modules.

We will show in \cite{Smi22b} that the cokernel of \eqref{eq:MS2_inj} vanishes for all but a negligible portion of fields $L$ in the set $X_{F, \ell, r_1'}(H)$ defined in Theorem \ref{thm:Cl}, so we may study the distribution of these Selmer groups to understand the distribution of the quotients $\Cl L[\ell^{\infty}]/ (\Cl L[\ell^{\infty}])^{\Gal(L/F)}$.
\end{ex}

Returning to the general situation of Definition \ref{defn:twistable}, it is clear that $N[\omega]$ is the subset of $N$ fixed by the automorphism $\xi$. For this reason, we sometimes refer to $\Sel^{\omega}\, N^{\chi}$ as the \emph{fixed point Selmer group} of the twist $N^{\chi}$. Since this subset is fixed by $\xi$, we see that the restriction
\[\beta_{\chi}: N^{\chi}[\omega] \to N[\omega]\] 
is $G_F$-equivariant. As a result, we may think of the fixed point Selmer group of any twist $N^{\chi}$ as a subgroup of $H^1(G_F, N[\omega])$ via $\beta_{\chi}$. More specifically, taking $\Vplac$ to be the set of places where $\chi$ ramifies, we see that this Selmer group maps to a subgroup of $\msS_{N[\omega]/F}(\Vplac_0 \cup \Vplac)$ cut out by certain local conditions. This allows us to control the fixed point Selmer group using the methods of Section \ref{sec:symbols}. One result is the following.

\begin{prop}
\label{prop:class_spin_symbol}
Fix a twistable module $N$ defined over $F$ with respect to $\FFF$, fix $(K/F, \Vplac_0)$ unpacking $N$, fix a set of local conditions at $\Vplac_0$, and fix a ramification section $\mfB$. Fix a positive integer $r$ and fix $h_i \in \FFF(-1)^{\times}$ for $i \le r$, where $\FFF(-1)^{\times}$ is the set of generators for $\FFF(-1)$. Also fix a homomorphism $\chi_0 \in \Hom(G_F, \FFF)$ that is unramified away from the places in $\Vplac_0$.

Choose primes $\ovp_1, \dots, \ovp_{r}$ and $\ovp_1', \dots, \ovp'_r$ of $\ovQQ$ not over $\Vplac_0$. We assume that the primes $\ovp_1 \cap F, \dots \ovp_{r} \cap F$ are distinct and that $\FrobF{F}{\ovp_i}$ acts trivially on $\FFF(-1)$ for each $i \le r$, and we make the analogous assumptions for $\ovp_1', \dots, \ovp_r'$.

Taking
\[\chi = \chi_0 + \sum_{i \le r} \mfB_{\ovp_i, \FFF}(h_i)\quad\text{and}\quad \chi' = \chi_0 + \sum_{i \le r} \mfB_{\ovp'_{i}, \FFF}(h_i),\]
we find that $\Sel^{\omega} N^{\chi}$ is isomorphic to $\Sel^{\omega} N^{\chi'}$ if
\begin{enumerate}
\item The class $\class{\ovp_i}$ equals $\class{\ovp_{i}'}$ for all $i \le r$,
\item The spin $\spin{\ovp_i}$  equals $\spin{\ovp_{i}'}$ for all $i \le r$, and
\item The symbol $\symb{\ovp_i}{\ovp_j}$ equals $\symb{\ovp_{i}'}{\ovp_{j}'}$ for all $i < j \le r$.
\end{enumerate}
\end{prop}

Proposition \ref{prop:class_spin_symbol} is a consequence of the next proposition, which lets us directly compare the fixed point Selmer groups of different twists. This proposition relies on the concept of a grid of twists, which we introduce now.

\begin{defn}
\label{defn:grid}
Take $N$ to be a twistable module defined over $F$ with respect to $\FFF$. Using Proposition \ref{prop:unpack_it}, we may choose $(K/F, \Vplac_0)$ unpacking $N$, and we fix a ramification section associated to the associated starting tuple.

Choose a nonempty finite set $S$, and take $X_s$ to be a finite nonempty set of primes of $\ovQQ$ for each $s \in S$. Taking $X_{F, s}$ to be the set of primes of $F$ divisible by a prime in $X_s$, we assume that the $X_{F, s}$ are disjoint sets of primes, and that these sets are disjoint from $\Vplac_0$.

For each $s \in S$, we assume that the primes in $X_s$ all have the same class, and we call this class $[X_s]$. That is to say, we assume there is some
\[\FrobF{F}{X_s} \in \Gal(K(\Vplac_0)/F)/\sim\]
so $\FrobF{F}{\ovp_s}$ projects to $\FrobF{F}{X_s}$ for all $\ovp_s \in X_s$, where $\Gal(K(\Vplac_0)/F)/\sim$ is defined as in Definition \ref{defn:class}. We will assume that $\FrobF{F}{X_s}$ acts trivially on $\FFF(-1)$ for all $s \in S$.

We then call the product space
\[X = \prod_{s \in S} X_s\]
a \emph{grid of ideals}. 

Given $M$ in $\Mod(K/F, |\FFF|)$ and $x \in X$, we then define a natural homomorphism
\[\mfB_{x, M}: \bigoplus_{s \in S} M(-1)^{\FrobF{F}{X_s}} \to H^1(G_F, M)\]
by
\[\mfB_{x, M}(m) = \sum_{s \in S} \mfB_{\pi_s(x), M}\left(\pi_s(m)\right),\]
where $\pi_s$ denotes projection onto the $s^{\text{th}}$ coordinate.

So choose $h \in \prod_{s \in S}\FFF(-1)^{\times}$, where $\FFF(-1)^{\times}$ is the subset of generators of $\FFF(-1)$, and choose a homomorphism $\chi_0 \in \Hom(G_F, \FFF)$ that is ramified only at places of $\Vplac_0$. Note that the assumptions on $\FrobF{F}{X_s}$ mean $\mfB_{x, \FFF}(h)$ is well defined. We define the \emph{grid of twists} corresponding to $(\chi_0, h, X)$ to be the set of twists
\[\{N^{\chi(x)}\,:\,\, x\in X\},\]
where $\chi(x)$ is shorthand for $\chi_0 + \mfB_{x, \FFF}(h)$.

Given $x_0$ in $X$, we define the \emph{grid class} $\class{x_0}$ of $x_0$ to be the set of $x \in X$ satisfying
\[\symb{\pi_s(x_0)}{\pi_t(x_0)} = \symb{\pi_s(x)}{\pi_t(x)}\]
for all $s, t \in S$.

Define
\begin{equation}
\label{eq:param_space_def}
\mathscr{M}(N) \,=\, \msS_{N[\omega]/F}(\Vplac_0) \,\oplus\, \bigoplus_{s \in S} N[\omega](-1)^{\FrobF{F}{X_s}}.
\end{equation}
From the definition of an unpacked starting tuple, we can define an isomorphism
\begin{equation}
\label{eq:param_Selmer}
\Psi_{x, N} :\mathscr{M}(N) \isoarrow \msS_{N^{\chi(x)}[\omega]/F}\big( \Vplac_0 \,\cup\, \{\pi_s(x) \cap F\,:\,\, s \in S\}\big) 
\end{equation}
by 
\[\Psi_{x, N}(\phi_0, n) =\beta_{\chi(x)}^{-1}\left(\phi_0 +  \mfB_{x, N[\omega]}(n)\right)\]
for any $x \in X$. We write $\mathscr{M}$ and $\Psi_x$ for these objects if $N$ is clear.

It can be verified from the naturality of the ramification section that the preimage of the rational torsion portion of $\Sel^{\omega} N^{\chi(x)}$ under $\Psi_x$ is fixed. We call this preimage $V_{\textup{tor}, N}$.
\end{defn}

The form of the automatic local conditions force the codomain of $\Psi_x$ to contain the $\omega$-Selmer group of $N^{\chi(x)}$, so we may use the maps $\Psi_x$ to parameterize Selmer elements. Our next proposition shows that the portion of $\mathscr{M}$ that maps into $\Sel^{\omega}N^{\chi(x)}$ is fixed as $x$ varies in a grid class.

\begin{prop}
\label{prop:class_spin_symbol_exact}
Take $N$ to be a twistable module over $F$ with respect to $\FFF$, choose $(K/F, \Vplac_0)$ unpacking $N$ and an associated ramification section $\mfB$, and consider the grid of twists of $N$ associated to some tuple $\left(\chi_0, h, X = \prod_{s \in S} X_s\right)$. Define $\mathscr{M}$ and $(\psi_x)_{x \in X}$ as in Definition \ref{defn:grid}.

Then, given $x \in X$ and $x'$ in the grid class of $x$, we have an equality of subgroups
\[ \Psi_x^{-1}\left(\Sel^{\omega} N^{\chi(x)}\right)\, = \,\Psi_{x'}^{-1}\left(\Sel^{\omega} N^{\chi(x')}\right)\quad\text{in }\, \mathscr{M}.\]
\end{prop}
 
\begin{proof}
Choose $(\phi_0, n)$ in $\mathscr{M}$, and take
\[\phi = \Psi_{x}(\phi_0, n)\quad\text{and}\quad \phi' = \Psi_{x'}(\phi_0, n).\]
We need to show that $\phi$ lies in $\Sel^{\omega} N^{\chi(x)}$ if and only if $\phi'$ lies in $\Sel^{\omega} N^{\chi(x')}$. 

We see that $\chi(x) - \chi(x')$ and $\beta_{\chi(x)}(\phi) - \beta_{\chi(x')}(\phi')$ have zero restriction to $G_v$ for $v \in \Vplac_0$. From this, we conclude that $\phi$ obeys the local condition at $v$ if and only if $\phi'$ obeys the local condition at $v$.

Next, for $s \in S$, we see that $\phi$ obeys the local condition at $v = \ovp_s \cap F$ if it lies in the kernel of the map $H^1(G_v, N^{\chi}[\omega]) \to H^1(G_v, N^{\chi(x)})$. This is equivalent to it lying in the image of the connecting map 
\[H^0(G_v, N^{\chi}[\omega]) = H^0(G_v, N^{\chi(x)}) \to H^1(G_v, N[\omega])\]
corresponding to the exact sequence
\[0 \to N^{\chi}[\omega] \to N^{\chi(x)} \to N^{\chi(x)} \to 0.\] 
Calling this boundary map $\delta_{\chi(x)}$, we find that
\[\beta_{\chi(x)}\left(\delta_{\chi(x)}\left(\beta_{\chi(x)}^{-1}(m)\right)\right) = \delta(m) + \chi(x) \cup m\]
for $m \in H^0(G_v, N[\omega])$, where $\delta$ is the boundary map associated to
\[0 \to N[\omega] \to N[\omega^2] \to N[\omega] \to 0\]
 and where the cup product is defined with respect to the bilinear pairing $\FFF \otimes N[\omega] \to N[\omega]$ given by $\xi \otimes m' \mapsto m'$.

There is a unique $m \in N[\omega]^{G_v}$ for which
\[\pi_s(n)\left(\overline{\zeta}\right) = \pi_s(h)\left(\overline{\zeta}\right) \cdot m.\]
We see this is the unique choice of $m$ for which $\phi - \delta(m) - \chi(x) \cup m$ is unramified at $v$. Equivalently, it is the unique choice of $m$ so $\phi' - \delta(m) - \chi'(x) \cup m$ is unramified at $\ovp'_s \cap F$.  The proposition then reduces to the claimed identity
\begin{align*}
&\phi'\left(\FrobF{F}{\ovp'_s}\right)  - \chi'(x)\left(\FrobF{F}{\ovp'_s}\right) \cdot m \\
&\qquad = \phi\left(\FrobF{F}{\ovp_s}\right)  - \chi(x)\left(\FrobF{F}{\ovp_s}\right) \cdot m \quad\text{in }\, N[\omega]_{G_v}.
\end{align*}
From the equality of classes and symbols, this reduces by Proposition \ref{prop:same_symbs} and Remark \ref{rmk:class} to the claim
\begin{align*}
&\mfB_{\ovp'_s, N[\omega]}(\pi_s(n)) \left(\FrobF{F}{\ovp'_s}\right) - \mfB_{\ovp'_s, \FFF}(\pi_s(h)) \left(\FrobF{F}{\ovp'_s}\right) \cdot m \\
&\qquad=  \mfB_{\ovp_s, N[\omega]}(\pi_s(n)) \left(\FrobF{F}{\ovp_s}\right) - \mfB_{\ovp_s, \FFF}(\pi_s(h)) \left(\FrobF{F}{\ovp_s}\right) \cdot m \quad\text{in }\, N[\omega]_{G_v}.
\end{align*}
But this follows from $\spin{\ovp} = \spin{\ovp'}$, as can be seen by applying Proposition \ref{prop:same_spins} to the difference
\[\mfB_{\ovp_s, N[\omega] \oplus \FFF}((\pi_s(n), \pi_s(h)))\left(\FrobF{F}{\ovp_s}\right) \,-\,  \mfB_{\ovp'_s, N[\omega] \oplus \FFF}((\pi_s(n), \pi_s(h)))\left(\FrobF{F}{\ovp'_s}\right)\] 
of cocycle classes in the module $N[\omega] \oplus \FFF$.
\end{proof}
This proposition may also be proved more explicitly using the double coset formula for corestriction and restriction. See \cite[Section 5.2]{Smi22b} for details.

\subsection{Dual Selmer groups and the Cassels--Tate pairing}
\label{ssec:CT}
\begin{defn}
Given a twistable module $N$, we define an action of $\Z_{\ell}[\xi]$ on $\Hom_{\Z_{\ell}}\left(N, \mu_{{\ell}^{\infty}}\right)$ so
\[\xi \phi(n) = \phi(\xi^{-1}n) \quad\text{for all } n \in N \,\text{ and }\, \phi \in \Hom_{\Z_{\ell}}\left(N, \mu_{{\ell}^{\infty}}\right). \]
On each such $\phi$, we define an action of $\sigma \in G_F$ by
\[(\sigma \phi)(n) = \sigma \big(\phi(\sigma^{-1} n)\big)\quad\text{for all } n \in N.\]
Using this, we can give the module
\[N^{\vee} = \Hom_{\Z_{\ell}}\left(N, \mu_{{\ell}^{\infty}}\right) \otimes_{\Z_{\ell}[\xi]} \QQ_{\ell}[\xi] \Big/ \Hom_{\Z_{\ell}}\left(N, \mu_{{\ell}^{\infty}}\right)\]
the structure of a twistable module. We call it the \emph{dual twistable module} to $N$.

Given a homomorphism $\chi$ in  $\Hom(G_v, \FFF)$ for some place $v$ or in $\Hom(G_F, \FFF)$, we see that the twistable modules $(N^{\chi})^{\vee}$ and $(N^{\vee})^{\chi}$ are canonically isomorhpic. We also have a non-degenerate pairing
\begin{equation}
\label{eq:dual_torsion}
N[\omega^k] \times N^{\vee}[\omega^k] \to \mu_{\ell^{\infty}}
\end{equation}
defined by
\[\left(x, \,\,\phi \otimes \omega^{-k}\right) \mapsto \phi\left(x\right).\]
This defines an isomorphism between $N^{\vee}[\omega^k]$ and $(N[\omega^k])^{\vee}$.

We write $V(N^{\chi})$ for the  inverse limit of $N^{\chi}$ indexed by positive integers with transitions given by multiplication by $\ell$. That is, we take elements in $V( N^{\chi})$ to be sequences $(n_1, n_2, \dots)$ in $N^{\chi}$ such that $\ell n_{i+1} = n_i$ for $i \ge 1$. This is a $\QQ_{\ell}$ vector space. 

The map taking $(n_1, \dots)$ to $n_1$ is a surjection from $V(N^{\chi})$ to $N^{\chi}$, and we write its kernel as $T(N^{\chi})$.

We note that the map taking  $n$ to $(n, \ell^{-1}n, \ell^{-2}n, \dots)$ defines an isomorphism
\[\Hom_{\Z_{\ell}}\left(N^{\chi}, \mu_{{\ell}^{\infty}}\right) \otimes_{\Z_{\ell}[\xi]} \QQ_{\ell}[\xi] \cong V\left((N^{\vee})^{\chi}\right).\]
We also have an isomorphism
\[\iota: \Hom_{\Z_{\ell}}\left(N^{\chi}, \mu_{{\ell}^{\infty}}\right) \otimes_{\Z_{\ell}[\xi]} \QQ_{\ell}[\xi] \isoarrow \Hom(V(N^{\chi}), \,\mu_{\ell^{\infty}})\]
given by
\[\iota(\phi \otimes \ell^{-k})\left(n_1, n_2, \dots\right) = \phi(n_{k+1}).\]
From these isomorphisms, we have a perfect pairing
\begin{equation}
\label{eq:V_pairing}
V(N^{\chi}) \otimes V\left((N^{\vee})^{\chi}\right) \to \mu_{\ell^{\infty}},
\end{equation}
and we find that $T(N^{\chi})$ and $T((N^{\vee})^{\chi})$ are orthogonal complements under this pairing.
\end{defn}

There is a canonical method for going from a set of local conditions for $N$ to a set of local conditions for $N^{\vee}$. This proceeds as follows:

\begin{defn}
Choose any place $v$ of $F$, and choose $\chi$ in $\Hom(G_v, \FFF)$.
 
By \cite[Lemma 4.9]{MS21}, we have an isomorphism
\[H^1(G_v, V(N^{\chi})) \cong \varprojlim H^1(G_v, N^{\chi}).\]
As a consequence, for each place $v$, there is a unique $\QQ_{\ell}$-vector space of $H^1(G_v, V(N^{\chi}))$ whose image in $H^1(G_v, N^{\chi})$ equals the local conditions for $N^{\chi}$ at $v$. We endow $V(N^{\chi})$ with these local conditions.

By \cite[Proposition 4.8]{MS21}, local Tate duality gives a nondegenerate pairing
\[H^1(G_v, \, V(N^{\chi})) \otimes H^1(G_v, V\left((N^{\vee})^{\chi}\right) \to \QQ/\Z\]
for each place $v$ of $F$. We endow $V((N^{\vee})^{\chi})$ with the orthogonal complement of the local conditions for $N^{\chi}$, and we endow $(N^{\vee})^{\chi}$ with the image of these local conditions under the projection from $V\left((N^{\vee})^{\chi}\right)$ to $(N^{\vee})^{\chi}$.

In the case that $v$ is outside $\Vplac_0$, we see that the local conditions this associates to $V(N^{\chi})$ is the group
\[ H^1_{\text{ur}}(G_v, V(N^{\chi})).\]
 By \cite[Proposition 1.4.2]{Rubin00}, we also have
\[H^1_{\text{ur}}(G_v, V(N^{\chi}))^{\perp} = H^1_{\text{ur}}(G_v, V((N^{\vee})^{\chi}))\]
under the local Tate duality pairing. As a consequence, we find that the local conditions this process associates to $(N^{\vee})^{\chi}$ are the automatic local conditions for $v$ outside $\Vplac_0$.

We call the set of local conditions this associates to $N^{\vee}$ the \emph{dual set of local conditions} to $N$.
\end{defn}

\begin{defn}
Choose $\chi \in \Hom(G_F, \FFF)$. With local conditions set as above, we find that the constructions in \cite[Section 4]{MS21} let us define a pairing
\[\langle\,\,,\,\,\rangle :\Sel\, N^{\chi} \times \Sel\, (N^{\vee})^{\chi} \to \QQ_{\ell}/\Z_{\ell}\]
whose kernel on each side is the subgroup of divisible Selmer elements. This is specifically the pairing associated to the exact sequence
\[0 \to T(N^{\chi}) \to V(N^{\chi}) \to N^{\chi} \to 0.\]
Given $\alpha \in \Z_{\ell}[\xi]$, naturality of the Cassels--Tate pairing  gives that this pairing satisfies
\[\langle \alpha \phi, \, \psi \rangle = \langle \phi,\, \overline{\alpha} \psi \rangle\,\,\text{ for all }\, \phi \in \Sel\,N^{\chi},\,\, \psi \in \Sel (N^{\vee})^{\chi},\]
where $\overline{\alpha}$ is the image of $\alpha$ under the continuous ring automorphism taking $\xi$ to $\xi^{-1}$.

Given nonnnegative integers $k$ and $j$, we can consider the restriction of this pairing
\begin{equation}
\label{eq:CTP_kj}
\Sel^{\omega^k} N^{\chi} \times \Sel^{\omega^j} (N^{\vee})^{\chi} \to \QQ_{\ell}/\Z_{\ell}.
\end{equation}
This is the pairing constructed in \cite{MS21} corresponding to the exact sequence
\[0 \to N^{\chi}\left[\omega^j\right] \hookrightarrow N^{\chi}\left[\omega^{k+j}\right] \xrightarrow{\,\,\omega^{j}\,\,}  N^{\chi}\left[\omega^{k}\right] \to 0,\]
where the local conditions for $N^{\chi}[\omega^{k + j}]$ at any place $v$ equals the preimage of $W_v(\chi)$ under the natural map $H^1(G_v, N^{\chi}[\omega^{k+j}]) \to H^1(G_v, N^{\chi})$.  By \cite[Theorem 1.3]{MS21}, the left kernel of this pairing is $\omega^j \Sel^{\omega^{k+j}} N^{\chi}$, and its right kernel is $\omega^k \Sel^{\omega^{k+j}} (N^{\vee})^{\chi}$.

We can then check that the pairing
\[\langle\,\,,\,\,\rangle_k: \omega^{k-1} \Sel^{\omega^k} N^{\chi} \times \omega^{k-1} \Sel^{\omega^k} (N^{\vee})^{\chi} \to \tfrac{1}{\ell}\Z/\Z\]
defined for $k \ge 1$ by
\[\langle\omega^{k-1}\phi,\, \omega^{k-1}\psi \rangle_k = \langle \phi,\, \omega^{k-1}\psi\rangle\]
is well defined and has left and right kernels $\omega^k \Sel^{\omega^{k+1}}N^{\chi}$ and $\omega^k \Sel^{\omega^{k+1}} (N^{\vee})^{\chi}$, respectively.

\end{defn}

\begin{defn}
We say that $N$ with an associated set of local conditions has \emph{alternating structure} if there is some equivariant isomorphism $\nu: N \to N^{\vee}$ that identifies the set of local conditions for $N$ with the dual set of local conditions to $N$ and for which the composition of the pairing \eqref{eq:dual_torsion} with $\nu$ gives an alternating pairing on $N[\omega^k]$ for $k \ge 1$.
\end{defn}
\begin{ex}
Taking $A/F$ to be an abelian variety and $N = A[2^{\infty}]$, we have a canonical identification between $A^{\vee}[2^{\infty}]$ and $N^{\vee}$, where $A^{\vee}$ denotes the dual abelian variety to $A$. Any polarization $\lambda \colon A \to A^{\vee}$ over $F$ of odd degree then can be used to define an alternating structure on $N$ with the local conditions of Example \ref{ex:abelian_Selmerable}. See \cite[Section 1.6.D]{Rubin00}.

We note that the twistable modules defined for class groups in Example \ref{ex:class_Selmerable} cannot have alternating structure, as $\QQ_{\ell}/\Z_{\ell}$ and its dual are not isomorphic over any number field.
\end{ex}
\begin{rmk}
\label{rmk:no2_noalt}
Suppose $N$ has alternating structure, and suppose either that $\ell \ne 2$ or that the Poonen--Stoll class defined as in \cite{MS21} for $N$ with $\nu$ is trivial. Then the Cassels--Tate pairing is an alternating pairing \cite[Theorem 5.21]{MS21}. Under these circumstances, the pairing $\langle\,\,,\,\,\rangle_k$ is an alternating pairing for odd $k$. However, if $k$ is even, this pairing is symmetric, so
\[\langle\phi, \,\psi\rangle_k = \langle\psi, \,\phi\rangle_k \,\text{ for } \phi, \psi \in \omega^{k-1} \Sel^{\omega^k} N^{\chi}.\]
If $\ell = 2$, this is equivalent to saying that $\langle\,\,,\,\,\rangle_k$ is antisymmetric for even $k$. If $|\FFF| = 2$, we have the stronger result that these pairings are alternating for even $k$ \cite[Theorem 1.7]{MS21}. But if $|\FFF|$ is a higher power of $2$, it is unclear what to expect from the diagonal entries of these pairings.

To avoid needing to control diagonal terms of the pairings $\langle\,\,,\,\,\rangle_k$, we will generally assume that $|\FFF| = 2$ when considering modules with alternating structure.
\end{rmk}

\subsection{The distribution of higher Selmer groups}
\label{ssec:heur}
The Cassels--Tate pairing gives us a method to understand the \emph{higher Selmer groups} of the twists of $N$, which are the Selmer groups of the form $\Sel^{\omega^k} N^{\chi}$ with $k$ larger than $1$. The Cohen--Lenstra--Gerth  heuristic \cite{Gert84} and Bhargava--Kane--Lenstra--Poonen-Rains heuristic \cite{BKLPR15} both predict how higher Selmer groups behave in twist families. The author's preferred way of interpreting and strengthening these heuristics is as equidistribution statements for the Cassels--Tate pairings $\langle\,\,,\,\,\rangle_k$ over the twist family. A priori, the pairings defined for distinct twists have distinct domains, complicating the form an equidistribution statement would take. Fortunately, we have a method for considering collections of twists whose Cassels--Tate pairings share a domain.
\begin{defn}
\label{defn:higher_grid}
Fix a twistable module $N$ defined over $F$ with respect to $\FFF$ and a tuple $(K/F, \Vplac_0)$ unpacking $N$. Take $X = \prod_{s \in S} X_s$ to be a grid of ideals with respect to this information. Fixing $\chi_0$ and $h \in \prod_{s \in S} \FFF(-1)^{\times}$, we can consider the twists of $N$ corresponding to ideals in $X$. 

Given $x$ in $X$ and a positive integer $j$, define
\[V_{\omega^j}(x) = \Psi_{x, N}^{-1}\left(\omega^{j-1} \Sel^{\omega^j} N^{\chi(x)}\right)\quad\text{and}\quad V^{\vee}_{\omega}(x) = \Psi_{x, N^{\vee}}^{-1}\left(\omega^{j-1}\Sel^{\omega^j} (N^{\vee})^{\chi(x)}\right),\]
where $\mathscr{M}$ and $\Psi$ are defined as in Definition \ref{defn:grid}. Fixing $x_0$ in $X$, Proposition \ref{prop:class_spin_symbol_exact} gives that
\[V_{\omega}(x)  = V_{\omega}(x_0) \quad\text{and}\quad V^{\vee}_{\omega^j}(x)  = V^{\vee}_{\omega}(x_0)\]
for $x$ in the grid class of $x_0$. We also note that
\[V_{\omega}(x_0) \supseteq V_{\omega^2}(x_0)\supseteq \dots \supseteq V_{\textup{tor}, N} \quad\text{and}\quad V^{\vee}_{\omega}(x_0) \supseteq V^{\vee}_{\omega^2}(x_0)\supseteq \dots \supseteq V_{\textup{tor}, N^{\vee}}.\]

Given an integer $k \ge 1$, we take $\class{x_0}_k$ to be the subset of $x$ in $\class{x_0}$ so
\begin{itemize}
\item For all positive integers $j \le k$, we have
\[V_{\omega^j}(x) = V_{\omega^j}(x_0) \quad\text{and}\quad V_{\omega^j}^{\vee}(x) = V_{\omega^j}^{\vee}(x_0) ; \,\text{ and }\]
\item For all positive integers $j$ strictly less than $k$, we have an identity
\[\left\langle \Psi_{x, N}(v), \,\Psi_{x, N^{\vee}}(v')\right\rangle_j = \left\langle \Psi_{x_0, N}(v), \,\Psi_{x_0, N^{\vee}}(v')\right\rangle_j \]
 of Cassels--Tate pairings for all $v$ in $V_{\omega^j}(x_0)$ and $v'$ in $V_{\omega^j}^{\vee}(x_0)$.
\end{itemize}
We call the set $\class{x_0}_k$ the \emph{higher grid class} of $x_0$ of level $k$. We note that the second condition on $x$ in $\class{x_0}_k$ implies the first, but we include both conditions to avoid concerns on the well-foundedness of this definition.
\end{defn}

With higher grid classes defined, we can give our heuristic.
\begin{heur}
\label{heur:random_matrix}
Take $\class{x_0}_k$ to be a higher grid class of $X$. Subject to assumptions on $N$, $X$ and $\class{x_0}_k$, the pairing
\[(v, v') \mapsto \left\langle \Psi_{x, N}(v), \,\Psi^{\vee}_{x, N^{\vee}}(v')\right\rangle_k\]
is approximately uniformly distributed among the reasonable possibilities for pairings
\[V_{\omega^k}(x_0) \times  V^{\vee}_{\omega^k}(x_0)  \to \tfrac{1}{\ell}\Z/\Z\]
as $x$ varies through $\class{x_0}_k$.
\end{heur}
Here, a pairing can only be reasonable if its left and right kernels contain the groups $V_{\textup{tor}, N}$ and $V_{\textup{tor}, N^{\vee}}$, respectively. Depending on the structure of $N$, there may be more restrictions on the form of these pairings. For example, if $N$ has alternating structure and either $k$ is odd or $|\FFF| = 2$, we know this pairing must be alternating. If $N$ has alternating structure but $k$ is even and $|\FFF| > 2$, we instead know that the pairing must be symmetric. In all the examples for which we have found the complete distribution of Selmer ranks, these are the only restrictions on the reasonable possibilities for these pairings.

The  main theorem of this paper, Theorem \ref{thm:main_higher}, gives conditions under which this heuristic holds. The following notation will be necessary.
\begin{defn}
\label{defn:grid_height}
Fix an unpacked starting tuple $(K/F, \Vplac_0, \# \FFF)$ and a twistable module $N$ as in Definition \ref{defn:higher_grid}. Consider a grid of twists associated to the info $(\chi_0, h)$ and the grid of ideals $X = \prod_{s \in S} X_s$. Given $s \in S$, we define the height of $X_s$ as
\[H_s = \max_{\ovp \in X_s} N_{F/\QQ}(\ovp \cap F),\]
and we define the height of $X$ to be the product $H = \prod_{s \in S} H_s$.

We next define the probabilities we will use to give the distribution of higher Selmer ranks. First, suppose $N$ has alternating structure and $\#\FFF = 2$. Then, for $n \ge j \ge 0$, we define $P(j\, |\,n)$ to be the probability that a uniformly selected alternating $n \times n$ matrix with coefficients in $\FFF_{2}$ has kernel of rank exactly $j$. 

If $N$ does not have alternating structure, we start by taking
\begin{equation}
\label{eq:dual_discrepancy}
u = r_{\omega}(N^{\chi(x_0)}) - r_{\omega}\left((N^{\vee})^{\chi(x_0)}\right).
\end{equation}
This integer does not depend on the choice of $x_0$ in $X$ by \cite[Theorem 8.7.9]{Neuk08}; see \cite[Remark 2.11]{Smi22b} for a fuller discussion.

Then, given $n \ge j \ge 0$, we take $P(j \,|\, n)$ to be the probability that a uniformly selected $(n - u) \times n$ matrix with coefficients in $\FFF_{\ell}$ has kernel of rank exactly $j$ if $n \ge u$, and we take $P(j\,|\,n) = 0$ otherwise. 
\end{defn}

The final notion we need to make Theorem \ref{thm:main_higher} precise is the notion of when a given grid class is \emph{ready for higher work}, a term we will define in Definition \ref{defn:ready}. If a grid class satisfies these conditions, we can find the distribution of Selmer groups over its twists.

\begin{thm}
\label{thm:main_higher}
Take $(K/F, \Vplac_0, \#\FFF)$ and $N$ as in Definition \ref{defn:grid_height}.
Then there are postive numbers $c, C > 0$ so, given any grid of twists of height $H > C$, and given a grid class $\class{x_0}$ that is ready for higher work in the sense of Definition \ref{defn:ready}, if we take
\[r_{\omega} = r_{\omega}(N^{\chi(x_0)}),\]
we have
\[\sum_{r_{\omega} \ge r_{\omega^2} \ge \dots} \left| \frac{\#\left \{x \in \class{x_0}\,:\,\, r_{\omega^k}\left(N^{\chi(x)}\right) = r_{\omega^k} \,\text{ for all } k \ge 1\right\}}{\#\class{x_0}} \,- \, \prod_{k \ge 1} P\left(r_{\omega^{k+1}} \,|\, r_{\omega^k}\right)\right|\]
\[ \le\, \exp\left(-c \cdot \left(\log \log\log H\right)^{1/2}\right),\]
where $P(j\,|\,n)$ is defined from $N$ and $X$ as in Definition \ref{defn:grid_height} and where the sum is over all nonincreasing sequences $r_{\omega^2} \ge r_{\omega^3} \ge \dots$ such that $r_{\omega} \ge r_{\omega^2}$.
\end{thm}
The proof of this theorem spans the entirety of Sections 6, 7, and 8.

\subsection{Grid classes that are ready for higher work}
Theorem \ref{thm:main_higher} gives the distribution of higher Selmer groups in specialized collections of twists, but we may apply it to prove distributional results about the Selmer groups in more natural families. Specifically, we may carve a natural family of twists into grids and split these grids into grid classes. After this process, some twists in the natural family will either not be in a grid or will not be in a grid class that is ready for higher work. However, so long as a negligible fraction of the natural family is left behind in this process, we can use the above theorem to find the distribution of Selmer groups in the natural family. This is the basic strategy for the proof of \cite[Theorem 2.14]{Smi22b}.

Our definition of a grid class that is ready for higher work needs to be restrictive enough that our methods for controlling higher Selmer groups work on such grid classes, but flexible enough that the twists outside any such class form a negligible portion of the natural family. In particular, our definition of these grid classes needs to complement the specific method we will use to carve natural families into grids, which appears in \cite[Section 8]{Smi22b}.

As a result, our definition of grid classes that are ready for higher work is complicated. We turn to it now.
\begin{defn}
\label{defn:Spotpre}
\label{defn:ready}
Consider a grid of twists associated to the pair $(\chi_0, h)$ and the grid of ideals $X = \prod_{s \in S} X_s$. Take $H$ to be the height of this grid of ideals, and take $H_s$ to be the height of $X_s$ for each $s \in S$. We assume $H > 20$.

We then define the set of potential active indices to be the set of elements $s \in S$ for which $\FrobF{F}{X_s}$ is the identity in $\Gal(K(\Vplac_0)/F)/\sim$ and for which $X_s$ is not a singleton. We define the set of \emph{potential prefix indices} $S_{\text{pot-pre}}$ to be the set of potential active indices $s$ satisfying
\begin{equation}
\label{eq:pre_disc_small}
H_s \le \exp^{(3)}\left(\tfrac{2}{3} \log^{(3)} H^2\right),
\end{equation}
and we define the set of \emph{potential a/b indices} $S_{\text{pot-a/b}}$ to be the set of potential active indices $s$ satisfying
\[H_s \ge \exp^{(3)}\left(\tfrac{3}{4} \log^{(3)} H\right).\]
Given $x_0 \in X$, a nonempty subset $S_{\text{pre}}$ of $S_{\text{pot-pre}}$ is called a set of \emph{prefix indices} if the kernels $\text{k}V_{\omega}$ and $\text{k}V^{\vee}_{\omega}$ of the standard projection maps
\[V_{\omega}(x_0) \to \bigoplus_{s \in S_{\text{pre}}} N[\omega](-1)\quad \text{and}\quad V^{\vee}_{\omega}(x_0) \to \bigoplus_{s \in S_{\text{pre}}} N^{\vee}[\omega](-1)\]
satisfy
\[\dim \text{k}V_{\omega} = r_{\omega}\left(N^{\chi(x_0)}\right) \quad\text{and}\quad\dim \text{k}V_{\omega}^{\vee} = r_{\omega}\left((N^{\vee})^{\chi(x_0)}\right).\]
Note that these identities are equivalent to the conditions
\[V_{\omega}(x_0)  = \text{k}V_{\omega} + V_{\textup{tor}, N}\quad\text{and}\quad V^{\vee}_{\omega}(x_0)  = \text{k}V_{\omega} + V_{\textup{tor}, N^{\vee}}.\]

We call the grid class $\class{x_0}$ \emph{ready for higher work} if, for every positive integer $k \le \log^{(3)} H$,
there is a choice of set of prefix indices $S_{\text{pre}}$ of cardinality $k$ such that we have the following:

\begin{enumerate}
\item We have $|S| \le\left(\log^{(2)} H^2 \right)^2$.
\item For each $s \in S$, either $X_s$ is a singleton, or
\[\#X_s\big/H_s \ge \exp^{(3)}\left(\tfrac{2}{7} \log^{(3)} H\right)^{-1} \quad\text{and}\quad H_s \ge \exp^{(3)}\left(\tfrac{1}{3} \log^{(3)} H\right).\]
\item We have
\begin{equation}
\label{eq:gridclass_big_enough}
\# \class{x_0}\big/ \# X \ge \exp^{(3)}\left(\tfrac{1}{4} \log^{(3)} H\right)^{-1}\,\,\text{ and}\quad r_{\omega}(N^{\chi(x_0)}) \le \left(\log^{(3)} H^2\right)^{1/4}.
\end{equation}
\item Define  $\text{k}V_{\omega}$ and $\text{k}V_{\omega}^{\vee}$ from the prefix indices $S_{\text{pre}}$ as above. Suppose that $N$ does not have alternating structure. Then, given any nonzero vectors $v_a \in \text{k}V_{\omega}$ and $v_b \in \text{k}V_{\omega}^{\vee}$ and codimension $1$ subspaces 
\[\text{k}_0V_{\omega} \subset \text{k}V_{\omega} \quad\text{and}\quad \text{k}_0V^{\vee}_{\omega} \subset \text{k}V^{\vee}_{\omega}\]
satisfying
\begin{equation}
\label{eq:ram_vector}
\text{k}V_{\omega} = \text{k}_0V_{\omega} + \langle v_a\rangle \quad\text{and}\quad \text{k}V^{\vee}_{\omega} = \text{k}_0V^{\vee}_{\omega} + \langle v_b\rangle,
\end{equation}
there are potential a/b indices $\saaa, \sbbb$ such that
\[\pi_{\saaa}(\text{k}_0V_{\omega}) = \pi_{\sbbb}(\text{k}_0V_{\omega}) = \pi_{\saaa}(\text{k}_0V^{\vee}_{\omega}) = \pi_{\sbbb}(\text{k}_0V^{\vee}_{\omega}) = \pi_{\saaa}(v_b) = \pi_{\sbbb}(v_a) = 0\]
and so, with respect to the natural pairing 
\[N[\omega](-1) \times N^{\vee}[\omega](-1) \xrightarrow{\quad} \tfrac{1}{\ell}\Z/\Z(-1),\]
we have
\[\pi_{\saaa}(v_a) \cdot \pi_{\sbbb}(v_b) \ne 0.\]
\item Taking $\text{k}V_{\omega}$ as before, suppose $N$ has alternating structure. Then, given any subspace $\langle v_a, v_b \rangle$ of $\text{k}V_{\omega}$ of dimension exactly $2$, and given a complement $V_0$ of $\langle v_a, v_b \rangle$ in $\text{k}V_{\omega}$, if we take
\[\text{k}_0V_{\omega} = V_0 + \langle v_b \rangle \quad\text{and}\quad \text{k}_0V^{\vee}_{\omega} = V_0 + \langle v_a \rangle ,\]
there are potential a/b indices $\saaa, \sbbb$ so
\[\pi_{\saaa}(\text{k}_0V_{\omega}) = \pi_{\sbbb}(\text{k}_0V^{\vee}_{\omega})  =0\]
and so
\[\pi_{\saaa}(v_a) \cdot \pi_{\sbbb}(v_b) \ne 0.\]
\item If $N$ has alternating structure, we have $\#\FFF = 2$, and the Poonen--Stoll class associated to $\nu$ \cite[Definition 1.6]{MS21} is $0$.
\end{enumerate}

\end{defn}

\section{Bilinear character sums}
\label{sec:bilinear}
In this section, we give a general bilinear character sum estimate. Our starting point is the following result of Jutila \cite[Lemma 3]{Juti75}: there is some absolute constant $C > 0$ so that, for any $N_1, N_2 \ge 3$, and for any sequence of complex coefficients $a_d$ indexed by integers of magnitude at most one, we have
\begin{equation}
\label{eq:Juti}
 \sum_{\substack{0 < e < N_1\\e \text{ odd, squarefree}}} \left|  \sum_{\substack{|d| < N_2\\ d \text{ squarefree}}} a_d \left(\frac{d}{e}\right)\right|\,\, \le \,\,C \cdot \left(N_1 \cdot N_2^{\frac{1}{2}} \,+\, N_1^{\frac{3}{4}} \cdot N_2 \cdot \left(\log N_2\right)^3\right),
\end{equation}
where $\left(\frac{d}{e}\right)$ denotes the Jacobi symbol, the bimultiplicative generalization of the standard Legendre symbol. This bilinear estimate fits into the theory of large sieve inequalities, with the standard reference being \cite[Chapter 7]{Iwan04}. 

Jutila's result has been sharpened and generalized substantially since it first appeared. Thanks to work of Heath-Brown \cite{HB95}, the exponents on $N_1$ and $N_2$ on the right side of \eqref{eq:Juti} can be lowered to within $\epsilon$ of the optimal values. Bilinear estimates of more general characters have also been found. Goldmakher and Louvel \cite{Gold13} directly extended Heath-Brown's work to quadratic Hecke families, which are certain collections of order-two Hecke characters defined over a general number field $F$. Follow-up work generalized this to higher-order characters \cite{Blom14}.

In parallel, Friedlander et al. found that estimates of bilinear sums of characters over number fields could be combined with estimates of short character sums to find the spin distribution of prime ideals in that number field  \cite[Proposition 5.1]{FIMR13}. As part of this program, several bilinear character sum estimates have been derived over the last few years, with a particularly streamlined form of the argument given in \cite[Proposition 3.6]{KoMi18}.

The result we give in this section is an adaptation of the argument in \cite{FIMR13} to our more general framework. Usually, such estimates do not make the dependence on the underlying fields (in our case, the field extension $K/F$) explicit. If we only needed this result for the tuple fixed in Definition \ref{defn:Sel}, we could have followed this trend, avoiding the work in Section \ref{ssec:Weiss}. But a key step in our proof of Theorem \ref{thm:higher_smallgrid} is Corollary \ref{cor:an}, a corollary of our bilinear character sum results that requires us to change the starting tuple.

This result may have applications outside its uses between this paper and its sequel \cite{Smi22b}, so we have included a more self-contained version of this result in Section \ref{ssec:Jutila}.

\begin{notat}
\label{notat:deg}
Given a number field $L$, we will denote the degree of $L$ by $n_L$ and the magnitude of its discriminant by $\Delta_L$. Given an ideal $\mfa$ of $L$, we will denote the rational norm of $\mfa$ by $N_L(\mfa)$.

We will use $n_{K/F}$ as shorthand for $n_K/n_F$.
\end{notat}

\begin{thm}
\label{thm:bilin_symb}
There is an absolute $C > 0$ so we have the following

Choose an unpacked starting tuple $(K/F, \Vplac_0, e_0)$ and some associated ramification section, and choose classes $\class{\ovp_0}$ and $\class{\ovq_0}$. We take $X_1$ to be a finite subset of $\class{\ovp_0}$ and $X_2$ to be a finite subset of $\class{\ovq_0}$. We assume that no two primes $\ovp, \ovp'$ in $X_{1} \cup X_2$ are over the same prime of $K(\Vplac_0)$.

For $i = 1, 2$, take 
\[N_i = \max_{\ovp \in X_i} N_F(\ovp \cap F).\]
Then, for any $s$ in $\symb{\class{\ovp_0}}{\class{\ovq_0}}$,
\begin{equation}
\label{eq:bilin_symb}
\sum_{\ovp \in X_1} \left| \#\{\ovq \in X_2\,:\,\, \symb{\ovp}{\ovq} = s\} \,-\, \frac{\#X_2}{\#\symb{\class{\ovp_0}}{\class{\ovq_0}}}\right| \le  A \cdot N_1 \cdot N_2 \cdot \left(N_1^{-\alpha} + N_2 ^{-\alpha}\right),
\end{equation}
where we have taken
\[A =  \left( e_0^{|\Vplac_0|} \cdot \log 2 \Delta_K\right)^{C \cdot n_K} \cdot \Delta_K^C\]
and
\[\alpha = \begin{cases}  1/(3n_{K/F} + 2) &\text{ if } \,n_{K/F} \text{ is even, and} \\ 1/(3n_{K/F} + 3) &\text{ if } \,n_{K/F} \text{ is odd.}\end{cases}\]
\end{thm}
Sections \ref{ssec:Weiss} and \ref{ssec:bilin_proof} constitute the proof of this result.

The exponent $\alpha$ in this result is  unoptimized, with Jutila's work showing the $\alpha = 1/6$ in the case $K = F = \QQ$ is improvable to any real number less than $1/4$. We would be interested in further optimizations to this theorem, but we have no need for them in our work.

\subsection{An application of Weiss's work}
\label{ssec:Weiss}
The proof of Theorem \ref{thm:bilin_symb} begins with the smoothed character sum estimates proved by Weiss in \cite{Weiss83}. We would like to thank  Jesse Thorner for pointing out the relevance of this work.

 We begin by recounting some of Weiss's notation.
\begin{defn}[\cite{Weiss83}]
Given a number field $L$ and a nonzero integral ideal $\mfb$ of $L$, take $I(\mfb)$ to be the group of fractional ideals of $L$ coprime to $\mfb$ and take $P(\mfb)$ to be the set of principal ideals generated by an element in $L^{\times}$ that is positive at all real places and equal to $1$ mod $\mfb$. A Dirichlet character mod $\mfb$ will then be a homomorphism from $I(\mfb)/P(\mfb)$ to $\mathbb{C}^{\times}$. 

Given a Dirichlet character $\chi$ mod $\mfb$, we extend $\chi$ to a homomorphism from $I(1)$ by setting $\chi(\mfp) = 1$ for each $\mfp$ dividing $\mfb$. This contrasts with the normal convention, so we call these maps modified Dirichlet characters.
\end{defn}

\begin{ex}
\label{ex:dirich}
Choose an assignment of approximate generators $(\alpha_{\ovp})_{\ovp}$ for $(K/F, \Vplac_0, \FFF)$. Choose a prime $\ovq$ of $\ovQQ$ not over a place in $\Vplac_0$, and choose $\tau  \in G_F$. Choose some class $\class{\ovp_0}$ of primes in $\ovQQ$.

Take $E = E(\ovp_0)$ and $E(\ovq)$ as in Notation \ref{notat:approx_gen_assign}, and take $L_{\tau} = E(\ovp_0) \cdot \tau E(\ovq)$. Then there is a Dirichlet character $\chi_{0, \ovq}$ mod $\mfb \cdot (\ovq \cap L_{\tau})$  over $L_{\tau}$ so
\[\symb{\ovp}{\ovq}(\tau) = \chi_{0, \ovq}(\ovp \cap L_{\tau})\quad\text{for all }\, \ovp \in \class{\ovp_0}\text{ satisfying }\ovp \cap K \ne \tau \ovq \cap K,\]
where $\mfb$ is a product of primes over $\Vplac_0$. Applying the transfer map from $G_E$ to $G_{L_{\tau}}^{\text{ab}}$, we find that there is a modified Dirichlet character $\chi_{\ovq}$ mod $\mfb' \cdot (\tau\ovq \cap E)$ over $E$ so
\[\symb{\ovp}{\ovq}(\tau) = \chi_{0, \ovq}(\ovp \cap L_{\tau})\quad\text{for all }\, \ovp \in \class{\ovp_0}\text{ satisfying }\ovp \cap K \ne \tau \ovq \cap K,\]
where $\mfb'$ is a product of primes over $\Vplac_0$.

Furthermore, given $\ovq'$ in $\class{\ovq}$, if we construct $\chi_{\ovq'}$ the same way, we find that $\chi_{\ovq}\cdot \overline{\chi_{\ovq'}}$ is a modified Dirichlet character mod $\tau(\ovq \ovq') \cap E$. This amounts to the observation that the character $\chi_{\ovq}\cdot\overline{\chi_{\ovq'}}$ corresponds to a Galois extension of $E$ that is unramified over the primes in $\Vplac_0$.
\end{ex}

Theorem \ref{thm:bilin} will follow as a consequence of the following proposition.
\begin{prop}
\label{prop:Thorn}
There is some absolute $C > 0$ so we have the following:

Take $L$ to be a number field of degree $n_L$ and of discriminant with magnitude $\Delta_L$. Take $F$ to be a subfield of $L$ of degree $n_F$. For each prime $\mfp$ of $F$ , choose a modified Dirichlet character $\chi_{\mfp}$ on $L$ and a complex constant $c_{\mfp}$ of magnitude at most $1$. Given any pair of distinct primes $\mfp, \mfp'$ of $F$ for which $c_{\mfp}$ and $c_{\mfp'}$ are both nonzero, we assume $\chi_{\mfp} \cdot \overline{\chi_{\mfp'}}$ is a nontrivial modified Dirichlet character mod $\mfp \mfp' \cdot \mathcal{O}_L$ over $L$.

 Then, for $N_1, N_2 > 1$, we have
\[\sum_{N_1 \le N_L(\mfa)  \le 2N_1} \left| \sum_{N_2 \le N_F(\mfp) \le 2N_2} c_{\mfp} \cdot \chi_{\mfp}(\mfa)\right|^2 \le  N_1 \cdot N_2 \cdot (C \cdot \log 2\Delta_L)^{n_L}  \,\,+\,\, C^{\,n_L} \cdot \Delta_L^{3/4} \cdot N_2^{2 + \frac{3n_L}{2n_F}}.\]
Here, the inner sum is over primes of $F$, and the outer sum is over integral ideals of $L$.
\end{prop}
\begin{proof}
Using the central limit theorem, we find that the smoothing function $H_{2n}$ appearing in \cite[Lemma 3.4]{Weiss83} is at least $c \cdot n^{-1/2}$ on the interval $[1/2, 1]$, where $c > 0$ is some constant. Applying this lemma, we find that there is some absolute $C_0 > 0$ and some function
\[h_L: \mathbb{R}^{\ge 0} \rightarrow \mathbb{R}^{\ge 0}\]
that is at least one on $[1/2, 1]$ so that, for any $y \ge 1$, any squarefree integral ideal $\mfb$ of $L$, and any nontrivial Dirichlet character $\chi$ defined mod $\mfb$, we have
\begin{equation}
\label{eq:Weiss_nontrivial}
\sum_{\mfa}  h_L\left(\frac{y}{N_L(\mfa)}\right) \chi(\mfa)\, \le\, C_0^{\,n_L} \cdot 2^{\omega(\mfb)}\cdot \Delta_L^{3/4} \cdot N_L(\mfb)^{3/4},
\end{equation}
where the sum is taken over all integral ideals of $L$, and where $\omega(\mfb)$ is defined as the number of distinct primes of $L$ dividing $\mfb$. Here, the Dirichlet characters are unmodified, so they equal $0$ on integral ideals that are not coprime to $\mfb$.

For the trivial character, \cite[Lemma 3.4]{Weiss83} gives the estimate
\begin{equation}
\label{eq:Weiss_trivial}
\sum_{\mfa} h_L\left(\frac{y}{N_L(\mfa)}\right) \,\le\, C_0^{\,n_L} \cdot \kappa(L) \cdot y \,+\, C_0^{\,n_L} \cdot \Delta_L^{3/4},
\end{equation}
where $\kappa(L)$ is defined to be the residue of the Dedekind zeta function for $L$ at $s = 1$. From \cite{Loub00}, we have the bound 
\[\kappa(L) < (C_1 \cdot \log 2\Delta_L)^{n_L}\]
for some absolute $C_1 > 0$. Now, we can bound the expression of the proposition by
\begin{align}
&\sum_{\mfa} h_L\left(\frac{N_1}{N_L(\mfa)}\right) \left| \sum_{N_2 \le N_F(\mfp) \le 2N_2} c_{\mfp} \cdot \chi_{\mfp}(\mfa)\right|^2  \label{eq:Weiss_smoothed_quad}\\
&\qquad\le \sum_{\substack{N_2 \le N_F(\mfp_1), N_F(\mfp_2) \le 2N_2 \\ c_{\mfp} \cdot c_{\mfp'} \ne 0}} \left|\sum_{\mfa} h_L\left(\frac{N_1}{N_L(\mfa)}\right)  \chi_{\mfp_1}(\mfa) \cdot \overline{\chi_{\mfp_2}(\mfa)}\right| \nonumber
\end{align}
We handle the diagonal terms of this sum using \eqref{eq:Weiss_trivial}. Given an off diagonal term with corresponding primes $\mfp_1 \ne \mfp_2$, we split the sum over $\mfa$ based on whether $\mfa$ is coprime to $\mfp_1\mfp_2 \cdot \mathcal{O}_L$. For the sum over $\mfa$ coprime to this ideal, we apply \eqref{eq:Weiss_nontrivial}. For the remainder, we use the identity
\[\sum_{\mfa \text{ not coprime to } \mfp_1\mfp_2}  h_L\left(\frac{N_1}{N_L(\mfa)}\right)  \, \le\, 2n_L/n_F \cdot \sum_{\mfa} h_L\left(\frac{N_1/N_2}{N_L(\mfa)}\right)\]
and \eqref{eq:Weiss_trivial}. Summing these terms, we find that \eqref{eq:Weiss_smoothed_quad} is at most
\[ N_2 \cdot N_1 \cdot (C \cdot \log 2\Delta_L)^{n_L}  \,\,+\,\, C^{\,n_L} \cdot \Delta_L^{3/4} \cdot N_2^{2 + \frac{3n_L}{2n_F}},\]
where $C > 0$ is some absolute constant. This gives the proposition.
\end{proof}

\subsection{The proof of Theorem \ref{thm:bilin_symb}}
\label{ssec:bilin_proof}
We will need the following simple degree estimate for extensions of local fields.

\begin{prop}
\label{prop:ram_exp_m}
Take $p$ to be a prime number and take $K_p$ to be a finite extension of $\QQ_p$. Take $q$ to be the cardinality of the residue field of $K_p$. Take $e_1$ to be a positive integer indivisible by $p$, take $s$ to be a nonnegative integer, and define $e_0$ to be $e_1 p^s$. Then, if $L_p/K_p$ is an abelian extension of exponent dividing $e_0$, the inertia subgroup of $\textup{Gal}(L_p/K_p)$ has order dividing
\[ e_1 \cdot q^{3s}.\]
\end{prop}
\begin{proof}
From local class field theory, we can  write the inertia subgroup of $L_p/K_p$ as a quotient of 
\[\mathcal{O}_{K_p}^{\times}\big/ (\mathcal{O}_{K_p}^{\times})^{e_0},\]
so we need only prove that $e_1q^{3s}$ is an upper bound on the size of this group. Taking $k_p$ to be the residue field of $K_p$, we know that $k_p^{\times}/(k_p^{\times})^{e_0}$has order dividing $e_1$  since $k_p^{\times}$ is cyclic and of order coprime to $p$. Now suppose $a$ is an element of $\mathcal{O}_{K_p}^{\times}$ that maps to one in $k_p$. Hensel's lemma implies that
\[x^{e_0} - a = 0\]
has a solution for $x$ in $\mathcal{O}_{K_p}$ if
\[a \equiv 1 \,\text{ mod } p^{3s};\]
see \cite[Theorem 7.3]{Eisen95}, for example. This gives the proposition.
\end{proof}
From this result, we find that
\[[K(\Vplac_0): K(\emptyset)] \,\le\, e_0^{n_{K/F} (3n_F +  |\Vplac_0|)}.\]
The degree of $K(\emptyset)$ over $K$ is bounded by the size of the class group of $K$, which is bounded by
\[C \cdot \Delta_K^{1/2} \cdot (\log 2 \Delta_K)^{n_K}\]
for some absolute $C$. This can be found by combining the upper bound of \cite{Loub00} with the lower bounds on regulators in \cite{Fried89}. We then get that
\begin{equation}
\label{eq:KV0_deg}
[K(\Vplac_0): F] \,\le\, e_0^{C n_K}\cdot  \Delta_K^{1/2} \cdot (\log 2 \Delta_K)^{n_K} \cdot e_0^{|\Vplac_0|\cdot n_{K/F}}
\end{equation}
for some absolute $C > 0$.

At this point, take  $(K/F, \Vplac_0, e_0)$ and $[\ovp_0], [\ovq_0]$ as in the theorem statement, and take $B$ as in Proposition \ref{prop:chinese}. We take $S(N_1, N_2)$ to be the maximal value of the left hand side of \eqref{eq:bilin_symb} attained for some pair of subsets $X_1 \subset \class{\ovp_0}$, $X_2 \subset \class{\ovq_0}$ satisfying  
\[ \max_{\ovp \in X_i} N_F(\ovp \cap F) \le N_i\quad\text{and}\quad \min_{\ovp \in X_i} N_F(\ovp \cap F) \ge \tfrac{1}{2}N_i\quad\text{for } i = 1, 2.\]

Given a function $\kappa: B \to \Z$ and $\zeta_0 \in \mu_{e_0}$, we can define a function $a: X_1 \times X_2 \to \mu_{e_0}$ by
\[a(\ovp, \ovq) =  \prod_{\tau \in B} \symb{\ovp}{\ovq}(\tau)^{\kappa(\tau)}.\]
We call this a non-trivial symbol product if $m_{\tau}$ does not divide $\kappa(\tau)$ for some $\tau$ in $B$.

For a positive integer $t$, take  $S_{ t}(N_1, N_2)$ to be the maximal value of an expression
\begin{equation}
\label{eq:S1t_exp}
\sum_{\ovp \in X_1} \left| \sum_{\ovq \in X_2}d_{\ovq} \cdot a(\ovp, \ovq)\right|^t
\end{equation}
attained for some choice of $X_1, X_2$ restricted as before, some choice of coefficients $d_{\ovq}$ indexed by primes in $X_2$ of magnitude $1$, and some choice of a non-trivial symbol product $a$. From basic Fourier analysis, we see that
\[S(N_1, N_2) \le  S_{1}(N_1, N_2).\]
Taking $n$ to be the degree of $K(\Vplac_0)$ over $\QQ$, we see that H{\"o}lder's inequality gives
\begin{equation}
\label{eq:S2S2t}
S_{1}(N_1, N_2)\, \le \,(n N_1)^{\frac{t-1}{t}} \cdot S_{ t}(N_1, N_2)^{\frac{1}{t}}.
\end{equation}
Now, given a positive integer $t$, and given $a, X_1, X_2$ maximizing \eqref{eq:S1t_exp}, we can find coefficients $c_{\ovp}$ indexed by $\ovp$ in $X_1$ of magnitude $1$ so
\[\sum_{\ovp \in X_1} \left| \sum_{\ovq \in X_2}  a(\ovp, \ovq)\right|^t = \sum_{\ovq_1, \dots, \ovq_t \in X_2}\, \sum_{\ovp \in X_1} c_{\ovp} \cdot a(\ovp, \ovq_1) \cdot \dots\cdot a(\ovp, \ovq_t).\]

Take $L = E(\ovq_0)$ as in Notation \ref{notat:approx_gen_assign}. Given $\ovp$ in $X_1$, Example \ref{ex:dirich} shows for each $\ovp \in X_1$ that there is some modified Dirichlet character $\chi_{\ovp}$ over $L$ and some root of unity $c$ so that
\[\chi_{\ovp}(\ovq \cap L) = c \cdot a(\ovp, \ovq)\]
for all $\ovq \in X_2$ so long as $\ovq$ does not divide $\ovp \cap F$. These characters are defined so that, for any $\ovp'$ in $X_1$, the modified Dirichlet character $\chi_{\ovp}\overline{\chi_{\ovp'}}$ is defined mod $(\ovp \ovp' \cap F) \cdot \mathcal{O}_L$.

 With these characters set, we find there are coefficients $c'_{\ovp}$ of magnitude $1$ so that
\[S_{t}(N_1, N_2) - \sum_{\ovq_1, \dots, \ovq_t \in X_2}\, \sum_{\ovp \in X_1} c'_{\ovp} \cdot \chi_{\ovp}(\ovq_1 \cdot \dots \cdot \ovq_t \cap L)\]
has magnitude bounded by $2nt (nN_2)^{t}$, as this expression bounds the number of choices of $(\ovp, \ovq_1, \dots, \ovq_t)$ in $X_1 \times X_2^t$ for which $\ovp \cap F$ equals some $\ovq_i \cap F$.

 Applying H{\"o}lder's inequality again gives
\[S_{t}(N_1, N_2) \le 2nt(nN_2)^{t}  + (nN_2)^{t/2} \cdot \left(\sum_{\ovq_1, \dots, \ovq_t \in X_2} \left|\sum_{\ovp \in X_1} c'_{\ovp}\cdot  \chi_{\ovp}\left(\ovq_1 \cdot \dots \cdot \ovq_t \cap L\right)\right|^2\right)^{1/2}.\]
Note that the number of primes in $X_2$ dividing a given ideal $\ovq_1 \dots \ovq_t \cap L$ is at most $nt$. Using this and Proposition \ref{prop:Thorn} then gives
\[S_{ t}(N_1, N_2) \le (tn)^{C_0t} \cdot \left( (C_0 \log \Delta_K)^{n_K/2}  N_1^{1/2} N_2^t + C_0^{\,n_K} \Delta_K^{3/8} N_1^{1 + \frac{3}{4}n_{K/F}} N_2^{t/2}\right)\]
for some absolute $C_0 > 0$.
With \eqref{eq:S2S2t}, we now see that $S_{1}(N_1, N_2)$ has an upper bound on the order of 
\[N_1N_2\left(N_1^{-\frac{1}{2t}} + N_1^{\frac{3}{4t}n_{K/F}}N_2^{-\frac{1}{2}}\right),\]
for each positive integer $t$. In the case that $N_2 \ge \tfrac{1}{2}N_1$, this inequality is optimized for
\[t = \left\lceil \tfrac{3}{2}n_{K/F} + 1\right\rceil.\]
To handle other cases, we note that our definition of the function $S_{t}$ uses a choice of $(\ovp_0, \ovq_0)$. If we take $S^{\circ}_t$ to be the analogous function defined for $(\ovq_0, \ovp_0)$, we see that reciprocity gives
\[S_1(N_1, N_2) = S^{\circ}_1(N_2, N_1)\quad\text{if } \left[\tfrac{1}{2}N_1,\, N_1\right]\cap \left[\tfrac{1}{2}N_2,\, N_2\right] = \emptyset.\]
In particular, if $N_2 < \tfrac{1}{2}N_1$, the upper bound we have already found for $S_1(N_2, N_1)$ is also an  upper bound for $S_1(N_1, N_2)$. We may finish the proof of the theorem by summing dyadically.
\qed

\subsection{A self-contained result for bilinear character sums}
\label{ssec:Jutila}

With a view to possible applications outside our work, we have decided to include a variant of Theorem \ref{thm:bilin_symb} that avoids the jargon of Section \ref{sec:symbols}. In this version, symbols are replaced by tempered functions, which we define now.

\begin{defn}
\label{defn:bilin}
Choose a Galois extension $K/F$ of number fields, a positive integer $e_0$ so $\mu_{e_0}$ lies in $K$, and a set of places $\Vplac_0$ of $F$ containing all places where $K/F$ ramifies, all primes dividing $e_0$ or the degree of $K$ over $F$, and all archimedean places. For any set of places $\Vplac$ of $F$, take $K(\Vplac)$ to be maximal abelian extension of $K$ of exponent dividing $e_0$ ramified only at places over $\Vplac$. 

Given a prime $\mfp$ of $F$ outside $\Vplac_0$, a \emph{tempered function} at $\mfp$ will be a real-valued class function
\[\phi: \Gal\left(K(\Vplac_0 \cup \{\mfp\})\big/F\right) \rightarrow [-1, 1]\]
so that $\phi$ has zero mean value on every coset of
\[\Gal\left(K(\Vplac_0 \cup \{\mfp\})\big/K(\Vplac_0)\right).\]

Given a tempered function $\phi_{\mfq}$ at $\mfq$ and another prime $\mfp$ of $F$, we define
\[\phi_{\mfq}(\mfp) = \begin{cases}0 &\text{if } \mfq \in \Vplac_0 \cup \{\mfp\} \\ \phi_{\mfq}(\Frob \,\mfp) &\text{otherwise.}\end{cases}\]
\end{defn}

We will keep all notation as in Notation \ref{notat:deg}.

\begin{thm}
\label{thm:bilin}
There is some absolute constant $C > 0$ so we have the following:

Choose $(K/F,\Vplac_0, e_0)$ as in Definition \ref{defn:bilin}. Choose a tempered function $\phi_{\mfp}$ at each prime of $F$ outside $\Vplac_0$.

Then, given $N_1, N_2 > 1$, we have
\begin{equation}
\label{eq:bilin}
\sum_{\substack{ N_F(\mfp) \le N_1}} \left| \sum_{\substack{\mfq \not \in \Vplac_0 \\  N_F(\mfq) \le N_2}} \phi_{\mfq}(\mfp)\right| \le  A \cdot N_1 \cdot N_2 \cdot \left(N_1^{-\alpha} + N_2 ^{-\alpha}\right)
\end{equation}
with
\[A =  \left( e^{|\Vplac_0|} \cdot \Delta_K \right)^{C \cdot n_K \cdot \log e_0}.\]
and
\[\alpha = \begin{cases}  1/(3n_{K/F} + 2) &\text{ if } \,n_{K/F} \text{ is even, and} \\ 1/(3n_{K/F} + 3) &\text{ if } \,n_{K/F} \text{ is odd.}\end{cases}\]
\end{thm}
\begin{proof}
Using Proposition \ref{prop:unpack_it}, we find that it suffices to consider the case where $(K/F, \Vplac_0, e_0)$ is unpacked. Under this assumption, the result follows from the  equivalence of (1) and (4) in Proposition \ref{prop:same_symbs_old}. The restriction of \eqref{eq:bilin_symb} to specific classes of primes is not a problem by \eqref{eq:KV0_deg}, and the restriction of the summand to the function identifying a certain symbol is not a problem by the simple estimate
\[ \#\symb{\class{\ovp_0}}{\class{\ovq_0}} \le e_0^{n_{K/F}}.\]
\end{proof}

\section{Regridding}
\label{sec:regrid}
Theorem \ref{thm:bilin_symb} is the final result from this paper that will be needed for our work on fixed point Selmer groups in \cite{Smi22b}. The remainder of this paper will be devoted to higher Selmer groups, with our main goal being to prove Theorem \ref{thm:main_higher}. We will return to the fixed point Selmer group and the streamlined consequences of Theorem \ref{thm:main_higher} in \cite{Smi22b}.

Before we can prove Theorem \ref{thm:main_higher}, we need to reduce its scope substantially. Our first reduction takes the theorem from a statement about the probability of seeing a sequence of Selmer ranks $r_{\omega^2}, r_{\omega^3}, \dots$ to an equidistribution statement for the Cassels--Tate pairing more in line with Heuristic \ref{heur:random_matrix}.

Throughout this section, we fix a twistable module $N$ defined over $F$ with respect to $\FFF$, and we fix $(K/F, \Vplac_0)$ unpacking $N$.

\begin{defn}
Choose a grid of twists corresponding to a grid of ideals $X$, and choose a grid class $\class{x_0}$ that is ready for higher work. 
Choose $k \le \log^{(3)} H$, and choose a set of prefix indices $S_{\text{pre}}$ of cardinality $k$ satisfying the conditions of Definition \ref{defn:ready}. Take $\text{k}V_{\omega}$ and $\text{k}V_{\omega}^{\vee}$ as in Definition \ref{defn:ready}. For $j \ge 1$ and $x \in X$, we write $r_{\omega^j}(x)$ for $r_{\omega^j}(N^{\chi(x)})$.

If $N$ does not have alternating structure, a \emph{Cassels--Tate test element} at level $\omega^k$ will be any nonzero element $w$ in
\[\big(\text{k}V_{\omega} \cap V_{\omega^k}(x_0)\big) \,\otimes\, \big(\text{k}V^{\vee}_{\omega} \cap V^{\vee}_{\omega^k}(x_0)\big).\]
If $N$ has alternating structure, a Cassels--Tate test element at level $\omega^k$ will be any nonzero element $w$ in
\[\bigwedge^2 \big(\text{k}V_{\omega} \cap V_{\omega^k}(x_0)\big).\]

Given $x$ in the higher grid class $\class{x_0}_k$, we can evaluate the Cassels--Tate pairing $\langle\,\,,\,\,\rangle_k$ at the image of $w$ under $\Psi_{x, N} \otimes \Psi_{x, N^{\vee}}$ or, in the case of alternating structure, $\wedge^2 \Psi_{x, N}$ (see Remark \ref{rmk:no2_noalt}). We take $\text{ct}_x(w)$ to be the result of this evaluation.

Given a subset $Y$ of $X$, we define the \emph{test mean} for $w$ on $Y$ to be the expression
\begin{equation}
\label{eq:amplified}
\frac{1}{\# Y} \cdot \left|\sum_{x \in \class{x_0}_k \cap Y} \exp\big(2\pi i \cdot  \textup{ct}_x(w)\big) \right|.
\end{equation}
\end{defn}

The following theorem implies Theorem \ref{thm:main_higher}, as we will prove in Section \ref{ssec:higher_fullgrid}.
\begin{thm}
\label{thm:higher_fullgrid}
There is $C> 0$ depending just on $N$, $(K/F, \Vplac_0)$ so we have the following:

Choose a grid of twists corresponding to a grid of ideals $X$ of height $H > C$, choose a grid class $\class{x_0}$ that is ready for higher work, and choose $k \ge 1$. Suppose we have the inequality
\begin{equation}
\label{eq:not_bad}
(r_{\omega}(x_0) + k) \cdot \left(r_{\omega}(x_0) + r_{\omega^2}(x_0) + \dots  + r_{\omega^k}(x_0)\right) \,\le\, \frac{\log^{(3)} H}{30\cdot (\dim N[\omega])^2\cdot \log \ell}.
\end{equation} 
Then, for any Cassels--Tate test vector $w$ of level $\omega^k$, the test mean for $w$ on $\class{x_0}$ is at most $(\log \log H)^{-1/4}$.
\end{thm}

We prove this theorem by giving estimates for the test mean on certain small subgrids of $\class{x_0}_k$ before patching to give the result for the whole grid. 

\begin{defn}
\label{defn:to_smallgrid}
Choose a grid of twists corresponding to a grid of ideals $X$ and choose a grid class $\class{x_0}$ that is ready for higher work. Choose a positive integer $k \le \log^{(3)} H$, and choose a Cassels--Tate test element $w$ of level $\omega^k$.

Take $S_{\text{pre}}$ to be a cardinality $k$ set of prefix indices. If $N$ does not have alternating structure, we choose decompositions \eqref{eq:ram_vector} of $\text{k}V_{\omega}$ and $\text{k}V^{\vee}_{\omega}$ so $w$ has nonzero image in 
\[\big(\text{k}V_{\omega}/\text{k}_0V_{\omega}\big) \otimes \big(\text{k}V^{\vee}_{\omega}/\text{k}_0V^{\vee}_{\omega}\big).\]
If $N$ has alternating structure, we choose $V_0$ of codimension $2$ in $\text{k}V_{\omega}$ and $v_a$ and $v_b$ as in Definition \ref{defn:ready} so $w$ has nonzero image in
\[\bigwedge^2 \big(\text{k}V_{\omega}/ V_0\big).\]
In both cases, we choose the vectors $v_a, v_b$ so $v_a$ is in $V_{\omega^k}(x_0)$ and $v_b$ is in $V_{\omega^k}^{\vee}(x_0)$, and we choose a/b indices $\saaa$ and $\sbbb$ obeying the conditions given in Definition \ref{defn:ready} with respect to these decompositions of $\text{k}V_{\omega}$ and $\text{k}V_{\omega}^{\vee}$. Such a choice is always possible for any given Cassels--Tate test vector.

With $S_{\text{pre}}$, $\saaa$, and $\sbbb$ selected, we take the \emph{active indices} to be the set
\[S_{\text{act}} = S_{\text{pre}} \cup \{\saaa,\sbbb\}.\]
We take the set of \emph{constant indices} $S_{\text{con}}$ to be $S \backslash S_{\text{act}}$.

A \emph{prefix} is a product of the form 
\[Z_{\text{pre}} = \left\{\left(\ovp_{0s}\right)_{s \in S_{\text{con}}}\right\}\times \prod_{s \in S_{\text{pre}}} X_{\text{pre}, s}\]
so that
\begin{enumerate}
\item The prime $\ovp_{0s}$ is in $X_s$ for $s \in S_{\textup{con}}$,
\item For $s \in S_{\text{pre}}$, the set $X_{\text{pre}, s}$ is a subset of $X_s$ of cardinality
\[E_{\text{pre-size}} := \left\lfloor \left(\log \log H\right)^{2/3}\right\rfloor\]
so that no two distinct elements of $X_{\text{pre}, s}$ are over the same prime in $F$, and
\item For all  $s, t$ in $S \backslash \{\saaa, \sbbb\}$ and $x$ in $Z_{\text{pre}}$, we have
\[\symb{\pi_s(x)}{\pi_t(x)} = \symb{\pi_s(x_0)}{\pi_t(x_0)}.\]
\end{enumerate}
We take $X_{\text{pre}}$ to be the product $\prod_{s \in S_{\text{pre}}} X_{\text{pre}, s}$. We refer to this as the \emph{prefix grid}.

We take $X_{\saaa}(Z_{\text{pre}})$ to be the subset of $ \ovp \in X_{\saaa}$ satisfying
\[\symb{\ovp}{\ovp} = \symb{\pi_{\saaa}(x_0)}{\pi_{\saaa}(x_0)}\]
and
\[\symb{\ovp}{\pi_s(x)} = \symb{\pi_{\saaa}(x_0)}{\pi_s(x_0)} \]
for all $x \in Z_{\text{pre}}$ and $s \in S \backslash \{\saaa, \sbbb\}$. We similarly define $X_{\sbbb}(Z_{\text{pre}})$.

We call this prefix \emph{suitable} if 
\[\frac{\# X_{\saaa}(Z_{\text{pre}})}{\# X_{\saaa}},\, \frac{\# X_{\sbbb}(Z_{\text{pre}})}{\# X_{\sbbb}} >  \exp^{(3)}\left(\tfrac{2}{7} \log^{(3)} H\right)^{-2}.\]
\end{defn}

The following result implies Theorem \ref{thm:higher_fullgrid}, as we will prove in Section \ref{ssec:higher_smallgrid}. We will then prove this result in Section \ref{sec:gov}.
\begin{thm}
\label{thm:higher_smallgrid}
There is $C> 0$ depending just on $N$ and $(K/F, \Vplac_0)$ so we have the following:

Choose a grid of twists corresponding to a grid of ideals $X$ of height $H > C$, and choose a grid class $\class{x_0}$ that is ready for higher work. Choose $k \ge 1$ so \eqref{eq:not_bad} is satisfied.

Choose any Cassels--Tate test vector $w$ of level $\omega^k$ and associated active indices $S_{\textup{act}} = S_{\textup{pre}} \cup \{\saaa, \sbbb\}$, and take $Z_{\textup{pre}}$ to be a suitable prefix. Take
\[Y = \class{x_0} \cap \big(X_{\saaa}(Z_{\textup{pre}}) \times X_{\sbbb}(Z_{\textup{pre}}) \times Z_{\textup{pre}}\big).\]
We assume that $Y$ is nonempty. 

Then the test mean for $w$ on $Y$ is at most $\tfrac{1}{2} (\log \log H)^{-1/4}$.
\end{thm}

\subsection{Theorem \ref{thm:higher_fullgrid} implies Theorem \ref{thm:main_higher}}
\label{ssec:higher_fullgrid}
The reduction of Theorem \ref{thm:main_higher} to Theorem \ref{thm:higher_fullgrid} proceeds through two intermediate results. The final stepping stone is the following.
\begin{intres}
\label{intres:full_sequence}
There is $C> 0$ depending just on $N$, $(K/F, \Vplac_0)$ so we have the following: 

Take $r_{\omega}$, $\class{x_0}$, $P$, and $X$ as in Theorem \ref{thm:main_higher}, with $X$ of height $H > C$ . Choose integers $k \ge 1$ and an integer sequence $r_{\omega^2} \ge \dots \ge r_{\omega^k}$ with $r_{\omega} \ge r_{\omega^2}$ and $r_{\omega^k} \ge 0$. Suppose that
\begin{equation}
\label{eq:not_bad_var}
(r_{\omega} + k ) \cdot (r_{\omega} + \dots + r_{\omega^k}) \,\le\, \frac{\log^{(3)} H}{30 \cdot (\dim N[\omega])^2\cdot \log \ell}.
\end{equation}
Then
\[\left|\frac{\#\left \{x \in \class{x_0}\,:\,\, r_{\omega^j}\left(N^{\chi(x)}\right) = r_{\omega^j} \text{ for } j \le k\right\}}{\#\class{x_0}} - \prod_{j = 1}^{k-1} P\left( r_{\omega^{j+1}} \,|\,r_{\omega^j}\right)\right|\]
is at most  $(\log \log H)^{-1/8}$.
\end{intres}

\begin{proof}[Proof that Intermediate Result \ref{intres:full_sequence} implies Theorem \ref{thm:main_higher}]
Given a nonnegative integer $r$, call $r$ \emph{terminal} if it is at most $1$ and $N$ has alternating structure, or if it equals $\max(0, u)$ with $u$ defined as in \eqref{eq:dual_discrepancy} if $N$ does not have alternating structure. Given $x \in \class{x_0}$, we see that $r_{\omega^j}(x)$ being terminal implies that
\[r_{\omega^j}(x) = r_{\omega^{j+1}}(x) = r_{\omega^{j+2}}(x)=\dots.\]

Call a sequence  $r_{\omega} \ge r_{\omega^2} \ge \dots \ge 0$  \emph{slow} if there is some choice of $k \ge 1$ for which $r_{\omega^{k-1}}$ is not terminal and \eqref{eq:not_bad_var} is not satisfied . If $k$ is chosen to be minimal so \eqref{eq:not_bad_var} is not satisfied, we call the sequence $r_{\omega} \ge \dots \ge r_{\omega^{k-1}}$ the \emph{valid fragment}.

Take $h = \log^{(3)} H$. Since $r_{\omega}$ is at most $h^{1/4}$, we find that any valid fragment will satisfy
\[c_0 h^{1/2}\le r_{\omega} + \dots + r_{\omega^{k-1}} \le h^{5/8}.\]
for $H$ sufficiently large, where $c_0 > 0$ does not depend on $H$. The upper bound here is proved by separately considering the cases where $k -1 \ge h^{3/8}$ and $k-1 < h^{3/8}$.

There is then $c_1 > 0$ not depending on $H$ so
\[\prod_{j \le k-2} P\left( r_{\omega^{j+1}} \,|\,r_{\omega^j}\right) \le e^{-c_1 h ^{1/2}}.\]
if $H$ is sufficiently large. From Intermediate Result \ref{intres:full_sequence}, we then have
\[\#\left\{x \in \class{x_0}\,:\,\, r_{\omega^j}\left(N^{\chi(x)}\right) = r_{\omega^j} \text{ for } j \le k - 1\right\}  \]
\[\le \left(e^{-c_1 h^{1/2}}+ (\log \log H)^{-1/8}\right) \cdot \#\class{x_0}.\]

Next, from the theory of partitions, we can bound the number of valid fragments by $e^{C_1 h^{5/16}}$ for sufficiently large $H$, where $C_1 > 0$ does not depend on $H$. In particular, the sum over valid fragments
\[\sum_{r_{\omega} \ge \dots \ge r_{\omega^{k-1}}} \left(\frac{\#\left\{x \in \class{x_0}\,:\,\, r_{\omega^j}\left(N^{\chi(x)}\right) = r_{\omega^j} \text{ for } j \le k - 1\right\}}{\# \class{x_0}} +  \prod_{j \le k-2} P\left( r_{\omega^{j+1}} \,|\,r_{\omega^j}\right)\right)\]
is bounded by  $\exp\left(- c_2 h^{1/2}\right)$ for sufficiently large $H$, where $c_2 > 0$ does not depend on $H$. This is within the error of Theorem \ref{thm:main_higher}

We may similarly bound the number of non-slow sequences $r_{\omega} \ge r_{\omega^2} \ge \dots$ by $e^{C_2 h^{1/2}}$ for some $C_2 >0$ not depending on $H$. The theorem follows by using Intermediate Result \ref{intres:full_sequence} for each of the non-slow sequences and taking a sum.
\end{proof}

The next intermediate result brings the focus to a single transition.
\begin{intres}
\label{intres:one_transition}
There is $C> 0$ depending just on $N$, $(K/F, \Vplac_0)$ so we have the following: 

Choose a grid of twists corresponding to a grid of ideals $X$ of height $H > C$, and choose a grid class $\class{x_0}$ that is ready for higher work. Choose $k \ge 1$ so \eqref{eq:not_bad} is satisfied.

Then, for any $r_{\omega^{k+1}} \le r_{\omega^k}(x_0)$, we have
\begin{align}
&\Big| \# \big \{ x \in \class{x_0}_k\,:\,\, r_{\omega^{k+1}}(x) = r_{\omega^{k+1}}\big\} \,-\, P\left(r_{\omega^{k+1}}\,|\, r_{\omega^k}\right) \cdot \# \class{x_0}_k\Big| \label{eq:one_transition}\\
&\qquad\qquad \qquad\le (\log \log H)^{-3/16} \cdot \#\class{x_0} \nonumber
\end{align}
\end{intres}

\begin{proof}[Proof that Intermediate Result \ref{intres:one_transition} implies Intermediate Result \ref{intres:full_sequence}]
Take $r_{\omega} \ge \dots \ge r_{\omega^k}$ as in Intermediate Result \ref{intres:full_sequence}. Given $j \le k$, take $n_j$ to be the number of higher grid classes contained in $\class{x_0}$ of level $j$. There is $C_0, C_1 > 0$ not depending on $H$ so we can bound $n_j$ by
\[\exp\left(C_0 \sum_{i \le j} r_{\omega^i}^2\right) \le \exp\left(C_0 r_{\omega}(r_{\omega} + \dots + r_{\omega^j})\right) \le \exp\left(C_1 \left(\log^{(3)} H\right)^{7/8}\right). \]
This function grows more slowly than $(\log \log H)^{1/32}$, so Intermediate Result \ref{intres:one_transition} lets us bound
\begin{align*}&\Big|\#\big\{ x \in \class{x_0} \,:\,\, r_{\omega^i}(x) = r_{\omega^i} \text{ for } i \le j + 1\big\} \\
&\qquad -\, P\left(r_{\omega^{j+1}}\,|r_{\omega^j}\,\right)\cdot \#\big\{ x \in \class{x_0} \,:\,\, r_{\omega^i}(x) = r_{\omega^i} \text{ for } i \le j \big\}\Big|
\end{align*}
by $(\log \log H)^{-5/32} \cdot \# \class{x_0}$. Summing over $j < k$ then gives  Intermediate Result \ref{intres:full_sequence}.
\end{proof}

\begin{proof}[Proof that Theorem \ref{thm:higher_fullgrid} implies Intermediate Result \ref{intres:one_transition}]
 With all notation fixed as in Intermediate Result \ref{intres:one_transition}, we take
\[W =\begin{cases}\bigwedge^2 \big(\text{k}V_{\omega} \cap V_{\omega^k}(x_0)\big) &\text{ if } N \text{ has alternating structure}\\  \big(\text{k}V_{\omega} \cap V_{\omega^k}(x_0)\big) \,\otimes\, \big(\text{k}V^{\vee}_{\omega} \cap V^{\vee}_{\omega^k}(x_0)\big)& \text{ otherwise}\end{cases}.\]
Given a homomorphism $\kappa: W \to \tfrac{1}{\ell}\Z/\Z$, the number of twists in $\class{x_0}_k$ with level $k$ Cassels--Tate pairing equal to $\kappa$ is
\[\frac{1}{\#W} \sum_{w \in W} \sum_{x \in \class{x_0}_k}  \exp\big(2\pi i \cdot  \textup{ct}_x(w) \cdot \kappa(w)^{-1}\big).\]
The summand corresponding to a given nonzero $w$ has magnitude at most $\#W^{-1} \cdot (\log \log H)^{-1/4} \cdot\#\class{x_0}$ by Theorem \ref{thm:higher_fullgrid}. The summand at $w = 0$ equals $\#\class{x_0}_k/\#W$.

The set of all possible pairings $\kappa$ has size $\#W$. Summing over this set, we find that \eqref{eq:one_transition} is at most
\[\# W\cdot  (\log \log H)^{-1/4} \cdot\#\class{x_0},\]
which is within the required bound for sufficiently large $H$.
\end{proof}

\subsection{Theorem \ref{thm:higher_smallgrid} implies Theorem \ref{thm:higher_fullgrid}}
\label{ssec:higher_smallgrid}
Theorem \ref{thm:higher_smallgrid} and Theorem \ref{thm:higher_fullgrid} are distinguished by the size of the grid to which they apply, with the former applying to smaller grids. To show that Theorem \ref{thm:higher_smallgrid} implies Theorem \ref{thm:higher_fullgrid}, we need to show that these smaller grids evenly cover the larger grids. To do this, we need a good estimate for the number of tuples in a large grid that satisfy a network of symbol conditions. The following setup will be useful.

\begin{notat}
Fix an unpacked starting tuple $(K/F, \Vplac_0, e_0)$ and a finite directed acyclic graph $G = (V, E)$ with vertex set $V$ and edge set $E$. For $v \in V$, take $X_v$ to be a nonempty set of primes of $\ovQQ$ lying in the same class, which we denote by $\class{X_v}$. For $(v, w) \in E$, choose $s(v, w)$ in $\symb{\class{X_v}}{\class{X_w}}$. For $v \in V$, take
\[H_v = \max_{\ovp \in X_v} N_{F/\QQ}(\ovp \cap F).\]
For each $v \in V$, we assume that no two primes in $X_v$ have the same intersection in $K(\Vplac_0)$. Take $X = \prod_{v \in V} X_v$.
\end{notat}

We then define the \emph{set of realizations} of $s$ to be the set of $x \in X$ satisfying
\[\symb{\pi_v(x)}{\pi_w(x)} = s(v, w)\quad\text{for all }\, (v, w) \in E.\]
We write this set as $\text{Rz}(X, G, s)$.
\begin{prop}
\label{prop:rz_est}
With all information fixed as above, there are constants $C, c > 0$ depending just on the tuple $(K/F, \Vplac_0, e_0)$ so 
\begin{equation}
\label{eq:rz_est}
\left|\# \textup{Rz}(X, G, s) \,-\, \# X\cdot \prod_{(v, w) \in E} \left(\# \symb{\class{X_v}}{\class{X_w}}\right)^{-1}\right| \,\le\, C \cdot H^{-c} \cdot \Delta^{-2} \cdot \# E \cdot  \#X,
\end{equation}
where $H$ is taken to be $\min_{v \in V} H_v$  and $\Delta$ is taken to be $\min_{v \in V} \#X_v/H_v$.
\end{prop}
\begin{proof}
Given $e = (v, w)$ in $E$ and $x$ in $X$, take $\delta_{e}(x)$ to be $1$ if $\symb{\pi_v(x)}{\pi_w(x)}$ equals $s(v, w)$ and $0$ otherwise. Then
\[\# \textup{Rz}(X, G, s)  = \sum_{x \in X} \prod_{e \in E}  \delta_e(x).\]
Choose any edge $e_0 = (v, w)$ in $E$. Take 
\[X' = \prod_{v \in V \backslash \{v, w\}} X_v.\]
We will use the identification of $X$ with $X' \times X_v \times X_w$. 

Take $E_1$ to be a subset of $E$ of edges not incident to $v$, and take $E_2$ to be a subset of $E\backslash E_1$ of edges not incident to $w$. For $e$ in $E_1$ and $(x', x_v, x_w)$ in $X' \times X_v \times X_w$, we see that $\delta_e((x', x_v, x_w))$ does not depend on $x_v$, and we rewrite it as $\delta_e(x', x_w)$. Similarly, if $e$ is in $E_2$, we write $\delta_e(x', x_v)$ for $\delta_e((x', x_v, x_w))$.

Take $E_0$ to be the complement of $\{e_0\} \cup E_1 \cup E_2$ in $E$.  Given $e = (v', w')$  in $E$, we take $n(e)$ to be $\# \symb{\class{X_{v'}}}{\class{X_{w'}}}$. We take the notation
\[a_1(x', x_w) = \prod_{e \in E_1} \delta_e(x', x_w)\quad\text{and}\quad a_2(x', x_v) = \prod_{e \in E_2} \delta_e(x', x_v). \]
Then there are constants $c_0, C_0 > 0$ depending just on the starting tuple so that
\begin{align*}
 &\left|\sum_{x \in X} \left( \delta_{e_0}(x) - n(e_0)^{-1}\right) \cdot \prod_{e \in E_0} n(e)^{-1}  \cdot \prod_{e \in E \backslash E_0\cup \{e_0\}} \delta_e(x)\right|\\
&\quad\le \sum_{x' \in X'} \prod_{e \in E_0} n(e)^{-1} \left| \sum_{x_v \in X_v} \sum_{x_w \in X_w} a_1(x', x_w) a_2(x', x_v) \left( \delta_{e_0}(x', x_v, x_w) - n(e_0)^{-1}\right)\right| \\
&\quad\le C_0 \cdot |X'| \cdot \prod_{e \in E_0} n(e)^{-1}\cdot H_vH_w(H_v^{-c_0} + H_w^{-c_0}),
\end{align*}
with the last inequality following from Theorem \ref{thm:bilin_symb}. Iterating this inequality gives \eqref{eq:rz_est}.
\end{proof}

\begin{proof}[Proof of Theorem \ref{thm:higher_fullgrid} given Theorem \ref{thm:higher_smallgrid}]
Take all notation as in Theorem \ref{thm:higher_fullgrid}, and fix $S_{\text{pre}}$ and $\saaa, \sbbb$. Take $S_{\text{sm}}$ to be the set of $s \in S$ for which $|X_s| = 1$, and take $S_{\text{lg}} = S \backslash S_{\text{sm}}$.

For $s \in S_{\text{lg}}$, take $X'_s$ to be the set of $\ovp \in X_s$ satisfying
\[\symb{\ovp}{\pi_t(x_0)} = \symb{\pi_s(x_0)}{\pi_t(x_0)}\quad\text{for all }\, t \in S_{\text{sm}}\]
and also satisfying $\spin{\ovp} = \spin{\pi_s(x_0)}$. Note that we can bound the ratio $\# X'_s/\# X_s$ using \eqref{eq:gridclass_big_enough}. 

For a given $x$ in $\class{x_0}$, take $n(x)$ to be the number of prefixes $Z_{\textup{pre}}$ so $x$ is in $X_{\saaa}(Z_{\textup{pre}}) \times X_{\sbbb}(Z_{\textup{pre}}) \times Z_{\textup{pre}}$. For $s, t\in S$, take $m_{st}$ to be $\# \symb{\class{X_s}}{\class{X_t}}$

Choose some total ordering on $S$, and take 
\[A = \prod_{s \in S_{\text{lg}}} \# X'_s \bigg/ \prod_{\substack{s, t \in S_{\text{lg}} \\ s > t}} m_{st}\]
and
\[B = \prod_{s \in S_{\text{pre}}} \frac{(\# X'_s)^{E_{\text{pre-size} - 1}}}{(E_{\text{pre-size}} - 1)!} \bigg/ \prod_{\substack{s, t \in S_{\text{pre}} \\  t <s}} m_{st}^{E_{\text{pre-size}}^2 - 1} \cdot \prod_{\substack{s \in S_{\text{pre}}\\ t \in S_{\text{lg}}\backslash S_{\text{pre}}}} m_{st} ^{E_{\text{pre-size}}-1}.\]
Given a nonnegative integer $j$, take $G_j$ to be the graph with vertex set
\[S_{\text{lg}}\, \cup \,\big(S_{\textup{pre}} \times \{2, \dots, E_{\textup{pre-size}}\} \times \{1, \dots, j\}\big),\]
and with an edge between any two vertices $v$ and $w$ unless $v = (s_1, i_1, k_1)$ and $w = (s_2, i_2, k_2)$, where either $s_1 = s_2$ or $k_1 \ne k_2$. By applying Proposition \ref{prop:rz_est} to $G_j$, we then see that there is some $c > 0$ not depending on $H$ so, for $j \in \{0, 1, 2\}$ and $H$ sufficiently large,
\[\left|\sum_{x \in \class{x_0}} n(x)^j - AB^j\right| \le AB^j \cdot \exp^{(3)}\left(\tfrac{1}{3} \log^{(3)} H\right)^{-c}\]
Take $n_0(x)$ to be the number of non-suitable prefixes containing $x$. By the definition of suitability, we find that
\[\sum_{x \in \class{x_0}} n_0(x)^j \le AB^j \cdot \exp^{(3)}\left(\tfrac{2}{7} \log^{(3)} H\right)^{-1/2}\]
for $j = 1, 2$. So, for sufficiently large $H$, we always have
\[\sum_{x \in \class{x_0}} \left(\frac{n(x) - n_0(x)}{B} - 1\right)^2 \le A\exp^{(3)}\left(\tfrac{2}{7} \log^{(3)} H\right)^{-1/4}.\]
Define $f: \class{x_0} \to \C$ by
\[f(x) = \begin{cases} \exp\big(2\pi i \cdot  \textup{ct}_x(w)\big)  &\text{ if } x \in \class{x_0}_k \\ 0 &\text{ otherwise.}\end{cases}.\]
Then, taking $Y(Z_{\text{pre}}) = \class{x_0} \cap Z_{\text{pre}} \times X_{\saaa} \times X_{\sbbb}$, we have
\[\left|\sum_{x \in \class{x_0}} (n(x) - n_0(x)) f(x)\right| = \left|\sum_{Z_{\text{pre}}} \sum_{x \in Y(Z_{\text{pre}})}  f(x)\right| \le \tfrac{3}{4}AB (\log \log H)^{-1/4},\]
for sufficiently large $H$  by Theorem \ref{thm:higher_smallgrid}, where the sum indexed by $Z_{\text{pre}}$ is over all suitable prefixes. The result then follows by Cauchy's theorem.
\end{proof}

\section{Algebra and combinatorics for grids}
\label{sec:algcomb}
The goal of the next two sections is to prove Theorem \ref{thm:higher_smallgrid}. This theorem is complicated, but the algebra and combinatorics that make its proof possible are comparatively simple. The fundamental combinatorial idea reduces to Hoeffding's inequality, which is a basic tail estimate for a sequence of Bernoulli trials. As already stated, the fundamental algebraic idea is to exploit the isomorphism $N[\omega] \cong N^{\chi}[\omega]$ as much as possible.

Because these underlying ideas are simple, we will begin by presenting them outside the context of Theorem \ref{thm:higher_smallgrid}. The $X$ and $S$ appearing in this section are different from the $X$ and $S$ appearing above, but their basic form is the same; $X$ represents a product space of twists, with $S$ giving the indices for the product.

\begin{notat}
\label{notat:gen_tar}
Take $\Z_{\ell}[\xi]$ to be the ring defined in Definition \ref{defn:twistable}, and take $\omega = \xi - 1$ as in that definition. Choose any topological group $G$. We consider the category $\textbf{Mod}_{G, \xi}$ of discrete $\Z_{\ell}[\xi]$-modules $M$ endowed with a continuous $G$ action commuting with the action of $\xi$.

Given a continuous homomorphism $\chi: G \to \langle \xi \rangle$ and $M$ in $\textbf{Mod}_{G, \xi}$, we can define a twist $M^{\chi}$ in $\textbf{Mod}_{G, \xi}$ and an isomorphism $\beta_{\chi}: M^{\chi} \to M$ of discrete $\Z_{\ell}[\xi]$ modules so
\[\beta_{\chi}(\sigma m) = \chi(\sigma) \sigma \beta_{\chi}(m)\quad\text{for all } m \in M,\,\, \sigma \in G.\]
This defines a twisting functor $\chi$ from $\textbf{Mod}_{G, \xi}$ to itself.

Given $M$ in $\textbf{Mod}_{G, \xi}$ and any finite set $X$, we take $M^X$ to be the set of maps from $X$ to $M$ viewed as an object in $\textbf{Mod}_{G, \xi}$. We define a multiplication between the function sets $\Z_{\ell}[\xi]^X$ and $M^X$ by
\[a \cdot m = \sum_{x \in X} a(x) m(x) \quad\text{for } a \in \Z_{\ell}[\xi]^X,\,\, m \in M^X.\]

We focus on the case of a set $X = \prod_{s \in S} X_s$, where $S$ is a nonempty finite set and the $X_s$ are disjoint nonempty finite sets. Given $x$ in $X$, we write the $s$-component of $x$ as $\pi_s(x)$. Fix $\chi_0$ in $\Hom_{\text{cont}}(G, \langle \xi\rangle)$ and
\[\chi_s:X_s \to \Hom_{\text{cont}}(G, \langle \xi\rangle)\]
for each $s \in S$. We then define $\chi: X \to \Hom_{\text{cont}}(G, \langle \xi\rangle)$ by
\begin{equation}
\label{eq:chi_families}
\chi(x) = \chi_0 \cdot \prod_{s \in S} \chi_s(\pi_s(x)) \quad\text{for } x\in X.
\end{equation}
We will use $\beta_x$ as shorthand for $\beta_{\chi(x)}$.

Given $Y \subseteq X$, we define the twist family for $M$ over $Y$ by
\[M(Y) = \bigoplus_{x \in Y} M^{\chi(x)},\]
and we define $\beta: M(Y) \to M^Y$ to be the map defined by
\[\beta\big((m_x)_{x \in Y}\big) = (\beta_x(m_x))_{x \in Y}\quad\text{for all}\quad (m_x)_{x \in Y} \in M(Y).\]
\end{notat}
\subsection{The zero-sums-in-lines condition}
\label{ssec:zsil}
\begin{defn}
\label{defn:closed}
Given $s \in S$, an $s$-line of $X$ is any subset $L$ of $X$ of the form
\[L = X_s \times \prod_{t \ne s} \{x_{0t}\}\quad\text{for some }\, (x_{0t})_t \in \prod_{t \ne s} X_t.\]
For $U$ a subset  of $S$ and $Y$ a subset of $X$, take $\zs(U, Y)$ to be the submodule of $a \in \Z_{\ell}[\xi]^Y$ satisfying
\[\sum_{x \in L \cap Y} a(x) = 0\quad\text{for every }s\text{-line }L\text{ with } s \in U.\]
We take $\zs(Y)$ to be shorthand for $\zs(S, Y)$.

Given $Y \subseteq Y' \subseteq X$, we have a natural injection $\Zlz^Y \to \Zlz^{Y'}$ given by assigning a map the value zero on $Y' \backslash Y$. This restricts to a map $\zs(U, Y) \to \zs(U, Y')$ for all subsets $U$ of $S$.

Given $Y \subseteq X$, we define the \emph{closure} of $Y$ to be the maximal set $\overline{Y}$ satisfying $Y \subseteq \overline{Y} \subseteq X$ for which the injection
\begin{equation}
\label{eq:closure_def}
\Zlz^Y\big/\zs(Y) \xhookrightarrow{\quad} \Zlz^{\overline{Y}}\big/\zs\left(\overline{Y}\right)
\end{equation}
is surjective. We call $Y$ closed if $\overline{Y} = Y$. If $Y$ is closed, a \emph{basis} for $Y$ is any minimal subset $Y_0$ of $Y$ with closure equal to $Y$. Note that the $x$ in $\overline{Y}\backslash Y$ are characterized by the fact that there is $a \in \zs(Y \cup \{x\})$ so $a(x) = 1$.
\end{defn}
\begin{rmk}
Adam Morgan has constructed explicit examples showing that the definition of closure depends on $\ell$. However, the definition does not otherwise depend on the choice of the ring $\Z_{\ell}[\xi]$.
\end{rmk}

\begin{prop}
Given a closed subset $Y$ of $X$ and a subset $Y_0$ of $Y$, the set $Y_0$ is a basis of $Y$ if and only if $\overline{Y_0} = Y$ and $\zs(Y_0) = 0$.
\end{prop}
\begin{proof}
If $Y_1$ is a strict subset of $Y_0$, the map $\Zlz^{Y_1} \to \Zlz^{Y_0}$ cannot be surjective. The conditions $\overline{Y_0} = Y$ and $\zs(Y_0) = 0$ then imply that $Y_0$ is a basis of $Y$.

Conversely, suppose $Y_0$ is a basis. We need to check that $\zs(Y_0) = 0$. Suppose otherwise, so there is nonzero $a$ in $\zs(Y_0)$. Note that $\Zlz$ is a local ring with maximal ideal $(\omega)$. We can find $x_0 \in Y_0$ for which the valuation of $a(x_0)$ at $(\omega)$ is minimized. Then 
\[(a(x_0))^{-1} a \in \zs(Y_0),\]
 so $\overline{Y_0 \backslash \{x_0\}} = Y$, contradicting the minimality of $Y_0$.
\end{proof}
\begin{ex}
\label{ex:2_squared}
Suppose $S = \{0, 1\}$, $X_0 = \{x_{00}, x_{10}\}$, and $X_1 = \{x_{01}, x_{11}\}$. The $\Zlz$ module $\zs(X)$ is then generated by the function $\Delta$ given by
\begin{alignat*}{2}
&\Delta(x_{00}, x_{01}) = 1 \qquad && \Delta(x_{00}, x_{11}) = -1\\
&\Delta(x_{10}, x_{01}) = -1 \qquad && \Delta(x_{10}, x_{11}) = 1.
\end{alignat*}
A subset $Y$ of $X$ is then closed if and only if $|Y| \ne 3$. If $|Y| = 3$, then $Y$ is a basis for $X$.
\end{ex}

For more general $S$ and $X$, we can generalize this example as follows.
\begin{defn}
Choose $x_0 = (x_{0s})_s$ and $x_1 = (x_{1s})_s$ in $X$. We then define $\Delta(x_0, x_1)$ in $\Zlz^X$ by
\[\Delta(x_0, x_1)\big((x_s)_s\big) = \prod_{s \in S} \Delta_s(x_s)\]
for $(x_s)_s \in X$, where $\Delta_s \colon X_s \to \{-1, 0, 1\}$ is defined by
\[\Delta_s(x_s) = \begin{cases} 1 &\text{ if } x_s = x_{0s} \\ -1 & \text{ if } x_s = x_{1s} \ne x_{0s} \\ 0 &\text{ otherwise.}\end{cases}\] 
If $x_{0s}\ne x_{1s}$ holds for all $s$ in $S$, this element lies in $\zs(X)$. 
\end{defn}

\begin{prop}
\label{prop:basis_X}
Choose $x_0 = (x_{0s})_s$ in $X$. Take $Y$ to be the set of $(x_s)_s$ in $X$ satisfying $x_s = x_{0s}$ for some $s$. Then $Y$ is a basis for $X$.
\end{prop}
\begin{proof}
Given $x_1$ outside $Y$, the element $\Delta(x_0, x_1)$ lies in $\zs(Y \cup \{x_1\})$, so $\overline{Y}$ contains $x_1$. So the closure of $Y$ is $X$.

Given $a$ in $\zs(Y)$, we need to show that $a$ is $0$. So take $Z$ to be the set of $x \in Y$ for which $a(x)$ is nonzero. Suppose this set is nonempty, and choose $x$ in $Z$ so that
\[\# \{s \in S\,:\,\, \pi_s(x) = x_{0s}\}\]
is minimized. Choose $s \in S$ so $\pi_s(x) = x_{0s}$.

From the zero-sum-in-lines condition, there is some other $x'$ on the $s$-line through $x$ so $a(x')$ is nonzero. But then
\[\# \{t \in S\,:\,\, \pi_t(x') = x_{0t}\}\, =\, \# \{t \in S\,:\,\, \pi_t(x) = x_{0t}\} - 1,\]
contradicting our choice of $x$. So $Z$ is empty and $\zs(Y)$ is $0$.
\end{proof}

\begin{prop}
\label{prop:basis_bound}
Any basis $Y$ of a closed subset  of $X$ satisfies
\[|Y|\, \le\, |X| - \prod_{s \in S} \big(|X_s| - 1 \big)\, \le\, |X| \cdot \sum_{s \in S} 1/|X_s|.\]
\end{prop}
\begin{proof}
Take $B$ to be the basis for $X$ constructed in Proposition \ref{prop:basis_X}. We then have an embedding
\[\Zlz^{Y} \xhookrightarrow{\quad} \Zlz^{X}/\zs(X) \cong \Zlz^B.\]
Tensoring with $\QQ_{\ell}$ and dimension counting gives $|Y| \le |B|$, giving the proposition.
\end{proof}

We can now state the main combinatorial ingredient in the proof of Theorem \ref{thm:higher_smallgrid}. The result concerns the $\Zlz$ module $\tfrac{1}{\ell}\Z/\Z$, where $\xi$ acts as the identity.

\begin{prop}
\label{prop:bye_Ramsey}
Given $\Zlz$ and any $X = \prod_{s \in S} X_s$ as in Notation \ref{notat:gen_tar}, and given any positive integer $M$, there is a choice of $g: X \to \tfrac{1}{\ell}\Z/\Z$ so the following holds:

For any sequence $Y_1, Y_2, \dots, Y_M$ of disjoint closed sets in $X$, and given any function 
\[f: Y\to \tfrac{1}{\ell}\Z/\Z\quad\text{with}\quad Y = \bigcup_{i=1}^m Y_i\]
satisfying
\begin{equation}
\label{eq:fg_zs}
a \cdot f = a \cdot g \quad\text{for all }\, i \le M \,\text{ and }\, a\in \zs(Y_i),
\end{equation}
we have
\begin{equation}
\label{eq:bye_Ramsey}
\Big|\left|f^{-1}(c)\right| \,-\, \tfrac{1}{\ell}|Y|\Big| \,\le\, \left(M \cdot \log(\ell|X|) \cdot \sum_{s \in S} 1/|X_s|\right)^{1/2} \cdot |X|.
\end{equation}
\end{prop}
\begin{proof}
The case $|X| = 1$ is easily checked. So suppose $|X| > 1$.

Write the right hand side of \eqref{eq:bye_Ramsey} as $\delta|X|$. Given $Y$, the number of $f: Y \to \tfrac{1}{\ell}\Z/\Z$ satisfying 
\[\Big|\left|f^{-1}(0)\right| \,-\, \tfrac{1}{\ell}|Y| \Big| \ge \delta |X|\]
is bounded by 
\[2 \exp\left(-\frac{2 \delta^2 |X|^2}{|Y|}\right)\cdot \ell^{|Y|}\]
by Hoeffding's inequality \cite[Theorem 1]{Okam59}. Take $B$ to be the basis for $X$ constructed in Proposition \ref{prop:basis_X}. Given $Y_1, \dots, Y_M$, the number of $g$ satisfying \eqref{eq:fg_zs} for a given $f$ is bounded by
\[\ell^{|X| - |Y| + M |B|},\]
since the restriction of $g$ to $Y_i$ can be determined from $f$ and the values $g$ takes on some basis of $Y_i$.

The number of possible $Y$ is bounded by the number of $M$-tuples of bases, which is bounded by $|X|^{M |B|}$. Combining these estimates, the number of $g$ for which \eqref{eq:bye_Ramsey} does not follow at $c = 0$ for some choice of $Y_1, \dots, Y_M$ and some $f$ satisfying \eqref{eq:fg_zs} is bounded by
\[2 \exp\left(-2 \frac{\delta^2 |X|^2}{|Y|}\right)\cdot (\ell |X|)^{M |B|} \cdot \ell^{|X|} < \ell^{|X|}.\]
So some $g$ is outside this set. Since \eqref{eq:bye_Ramsey} holds for all $f$ satisfying \eqref{eq:fg_zs} in the case $c = 0$, we find that it also holds for general $c$ by considering functions $f$ shifted by a constant.
\end{proof}

\subsection{Interesting submodules of $N(X)$}
\label{ssec:interesting}
Take $I$ to be an ideal of the power set $\mathscr{P}(S)$ of $S$ viewed as a lattice; that is, take $I$ to be a nonempty set of subsets of $S$ so every subset of every $U$ in $I$ is also in $I$. Take $b$ to be a nonnegative integer so every subset in $I$ has cardinality at most $b$. Given a subset $Y$ of $X$, we define
\[\zs_{I, b}(Y) = \left(\sum_{U \in I } \omega^{b - |U|} \cdot \zs(U, X) \right)\, \cap\, \Zlz^Y,\]
and we then define
\[N_{I, b}(Y) = \left\{n \in N(Y) \,:\,\, a \cdot \beta(n) = 0 \,\text{ for }\, a \in \zs_{I, b}(Y)\right\}\]
for any $N$ in $\textbf{Mod}_{G, \xi}$. In the case where $I$ is $\{\emptyset\}$, $N_{I, b}(Y)$ equals the $\omega^b$-torsion of $N(Y)$. As more sets are added to $I$, this module decreases in size. We take $N_I(Y)$ as shorthand for $N_{I, |S|}(Y)$.

We call $N$ divisible if $N = \ell N$.

\begin{ex}
Choose $(x_{0s})_{s \in S}$, and suppose that the image of the set of homomorphisms
\[\{\chi_s(x_s)\chi_s(x_{0s})^{-1}\,:\,\, s \in S,\, x_s \in X_s \backslash \{x_{0s}\}\}\]
in $\Hom\left(G, \,\langle \xi\rangle/ \langle \xi^{\ell}\rangle\right)$ spans a $\sum_{s \in S} |X_s| - 1$ dimensional subspace of this space of homomorphisms. Suppose further that $N$ is a divisible module.

Then $N_{\mathscr{P}(S)}(Y)$ is the minimal $G$-submodule of $N(Y)$ containing $(\beta_x^{-1}(n))_{x \in Y}$ for every $n$ in $N\left[\omega^{|S|}\right]$. That is, this is the minimal $G$-submodule containing the diagonal image of $N\left[\omega^{|S|}\right]$ in $N(Y)$.

This result will follow as a consequence of Proposition \ref{prop:closed_under_G} and the relation \eqref{eq:sig_n}.
\end{ex}

\begin{prop}
\label{prop:Nib_surj}
Given any divisible $N$, and given $Y \subseteq Y' \subseteq X$ and $I$ and $b$ as above, the map
\[N_{I, b}(Y') \to N_{I, b}(Y)\]
given by restricting the natural projection is surjective.
\end{prop}
\begin{proof}
It suffices to prove this result in the case where $Y' = Y \cup \{x\}$ for some $x$ outside $Y$. Choose $a \in \zs_{I, b}(Y')$ for which the valuation of $a(x)$ at $\omega$ is minimized. From the divisibility hypothesis, given $n$ in $N_{I, b}(Y)$, we can choose $n'$ in $N(Y \cup \{x\})$ projecting to $n$ so $a \cdot\beta(n') = 0$. 

Given any other $a_1 \in \zs_{I, b}(Y')$, there must be $c \in \Zlz$ so 
\[(a_1 - ca)(x) = 0.\]
Then $a_1 - ca$ is in $\zs_{I, b}(Y)$, so $a_1 \cdot \beta(n') = 0$. This confirms that $n'$ is in $N_{I,b}(Y')$, establishing surjectivity.
\end{proof} 

\begin{prop}
\label{prop:Nib_bij}
Given any divisible $N$ and subset $Y$ of $X$, the projection from $N\left(\overline{Y}\right)$ to $N(Y)$ defines an isomorphism
\[N_{\mathscr{P}(S), |S|}\left(\overline{Y}\right) \isoarrow N_{\mathscr{P}(S), |S|}(Y).\]
\end{prop}
\begin{proof}
We need to show that this map is injective, so take $n$ to be in its kernel. Given $x$ in $\overline{Y} \backslash Y$, we can find $a$ in $\zs(Y \cup \{x\})$ such that $a(x) = 1$. Since $a \cdot \beta(n) = 0$, we find that $n(x) = 0$. Since this holds for all $x$ in $\overline{Y}\backslash Y$, we find $n  = 0$, so the kernel of this map is trivial.
\end{proof}

\begin{defn}
\label{defn:eta}
Given $m \in N$, $U \subseteq S$, and a tuple $(x_{0s})_{s \in U}$ in $\prod_{s \in U} X_s$, we define
\begin{equation}
\label{eq:eta_gen}
\eta\big(U,\, (x_{0s})_{s \in U},\, m\big)
\end{equation}
to be the unique element $n$ in $N(X)$ satisfying
\[n(x) = \begin{cases} \beta_x^{-1}(m) &\text{ if } \pi_s(x) = x_{0s} \text{ for all } s\in U \\ 0 &\text{ otherwise}\end{cases}\]
for all $x \in X$. The element $\beta(n)$ is then nonzero on a fixed $|S| - |U|$ dimensional subgrid of $X$, where it takes the constant value $m$.

Given an ideal $I$ in $\mathscr{P}(S)$ and nonnegative integer $b$ no smaller than the maximal cardinality of a set in $I$, we call an element of the form \eqref{eq:eta_gen} an \emph{eta element for} $(I, b)$ if $m$ lies in $N[\omega^{b-t}]$, where $t$ is the maximal size of a subset of $U$ contained in $I$.
\end{defn}
\begin{prop}
\label{prop:eta_basis}
Given any $N$, $X$, $I$, and $b$ as in Definition \ref{defn:eta}, the eta elements for $(I, b)$ generate the module $N_{I, b}(X)$ over $\Z$.
\end{prop}
\begin{proof}
Take $\mathscr{N}$ to be the eta elements for $(I, b)$. It is easily checked that $\mathscr{N}$ is a subset of $N_{I, b}(X)$. Now suppose we have $n \in N_{I, b}(X)$. We wish to express $n$ as an integer combination of elements of $\mathscr{N}$.

To this end, choose any $(x_{0s})_s$ in $X$. Given $x$ and $x'$ in $X$, we say $x > x'$ if 
\[\#\{s \,:\,\, \pi_s(x) = x_{0s}\} > \#\{s \,:\,\, \pi_s(x') = x_{0s}\}.\]
Extend this partial order to a linear order on $X$ in any fashion. Now suppose we have $n_0 \in N_{I, b}(X)$ so $n_0 - n$ is an integer combination of elements in $\mathscr{N}$. We claim we can find $n_1$ so $n_1 - n$ is an integer combination of elements in $\mathscr{N}$ and so that either $n_1 = 0$ or the greatest $x$ for which $n_1(x) \ne 0$ is less than the greatest $x$ for which $n_0(x) \ne 0$.

To do this, take $x_1 = (x_{1s})_s$ to be the greatest point in $X$ for which $n_0(x_1)$ is nonzero, and take $U$ to be the set of $s$ for which $x_{1s} \ne x_{0s}$. Take $t$ to be the maximal size of a subset of $U$ contained in $I$. Since $x_1$ is maximal, we find that 
\[\Delta(x_1, x_0) \cdot \beta(n_0) = \beta_{x_1}(n_0(x_1)).\]
 Since $\Delta(x_1, x_0)$ lies in $\zs(U, X)$, this implies that $\omega^{b-t} n_{0}(x_1) = 0$.
We can then subtract $\eta\big(U,\, (x_{1s})_{s \in U}, \,\beta(n_0)(x_1)\big)$ from $n_0$ to define $n_1$, and it has the properties we described. 

Repeating this process will eventually show that $n$ is an integer combination of elements in $\mathscr{N}$.
\end{proof}

In the introduction, we claimed that the equivariant isomorphism \eqref{eq:2_guy} had higher analogues. These higher analogues come out of the following result.
\begin{prop}
\label{prop:closed_under_G}
Suppose $N$ is divisible. Given $I$ and $b$ as above, the submodule $N_{I, b}(Y)$ of $N(Y)$ is closed under the action of $G$.
\end{prop}
\begin{proof}
From Proposition \ref{prop:Nib_surj}, we see that it suffices to prove this result with the added assumption $Y = X$. Take
\[n = \eta\big(U,\, (x_{0s})_{s \in U},\, m\big)\]
to be an eta element for $(I, b)$. Given $\sigma$ in $G$, we need to check that $\sigma n$ still is an integer combination of eta elements for $(I, b)$.

Given $x = (x_s)_s$, we have
\[\beta(\sigma n(x)) = \chi(x)(\sigma)\sigma \beta(n(x)),\]
where $\chi$ is defined as in \eqref{eq:chi_families}. Choose $x_1$ in $X$ projecting to $(x_{0s})_{s \in U}$, and write $\chi'_s(x_s)$ for $\chi_s(x_s) \cdot \chi_s(\pi_s(x_1))^{-1}$.  For all $x$ in $X$ projecting to $(x_{0s})_{s \in U}$, we have
\begin{align*}
\chi(x)(\sigma) & \,=\, \chi(x_1)(\sigma) \cdot \prod_{s \in S \backslash U}\left(\chi'_s(x_s)(\sigma) -1 + 1\right)\\
&\,=\, \chi(x_1)(\sigma) \cdot \sum_{V \subseteq S \backslash U}  \cdot \prod_{s \in V} \left(\chi'_s(x_s)(\sigma)  -1\right)\\
&\,=\,   \chi(x_1)(\sigma) \cdot \sum_{V \subseteq S \backslash U} \omega^{|V|} \cdot \prod_{s \in V} \frac{\chi'_s(x_s)(\sigma)  -1}{\omega}
\end{align*}
We then have
\begin{equation}
\label{eq:sig_n}
\sigma n = \sum_{V \subseteq S \backslash U}\sum_{(x_{0s})_{s \in V}} \omega^{|V|} \cdot n\big( (x_{0s})_{s \in V}\big),
\end{equation}
where for any $V \subseteq S \backslash U$ and  $(x_{0s})_{s \in V}$ in $\prod_{s \in V} X_s$ we have taken the notation
\[n\big( (x_{0s})_{s \in V}\big) = \chi(x_1)(\sigma)  \cdot \left(\prod_{s \in V} \frac{\chi'_s(x_{0s})(\sigma) -1}{\omega}\right) \cdot \eta\big(V \cup U,\, (x_{0s})_{s \in V \cup U}, \,\sigma m\big).\]
The expression \eqref{eq:sig_n} is a sum of eta elements for $(I, b)$, and the result follows.
\end{proof}
\begin{ex}
\label{ex:four_subquo}
Take $X$ to be the grid considered in Example \ref{ex:2_squared}.

 In ths case, $N_{\{\{\}, \{0\}, \{1\}\}, 2}(X)$ consists of the elements in $N(X)[\omega^2]$ for which the value of $\omega \beta(n)(x)$ does not depend on the choice of $x\in X$. Taking $I$ to be the power set of $\{0, 1\}$, we see that $N_{I, 2}(X)$ is the subset of elements $n$ in this module that additionally satisfy
\[\Delta((x_{00}, x_{01}),\, (x_{10}, x_{11})) \cdot \beta(n) = 0.\] 
From this last condition and Proposition \ref{prop:closed_under_G}, we see that $N_{I, 2}(X)$ is isomorphic in $\textbf{Mod}_{G, \xi}$ to a $G$-submodule of
\[N^{\chi(x_{00},\, x_{01})}[\omega^2] \oplus N^{\chi(x_{10},\, x_{01})}[\omega^2] \oplus N^{\chi(x_{00}, \,x_{11})}[\omega^2]\]
that projects equivariantly and surjectively onto $N^{\chi(x_{10}, \,x_{11})}[\omega^2]$.

Given any continuous homomorphisms $\chi_1, \chi_2: G \to \langle \xi \rangle$, this shows that the group module $N^{\chi_1 \cdot \chi_2}[\omega^2]$ is a subquotient of $N^{\chi_1}[\omega^2] \oplus N^{\chi_2}[\omega^2] \oplus N[\omega^2]$.
\end{ex}

\begin{prop}
\label{prop:equivar_NIb}
Given $N_{I, b}(Y)$ as in Proposition \ref{prop:closed_under_G}, given a subset $T \subseteq S$ outside $I$ of cardinality at most $b$ whose proper subsets are all in $I$, and given $a$ in $\zs_{I \cup \{T\}, b}(Y)$, the map $n \mapsto a \cdot\beta(n)$ defines an equivariant homomorphism from $N_{I, b}(Y)$ to $N[\omega]$.

\end{prop}
\begin{proof}
From Proposition \ref{prop:Nib_surj}, we see we may assume $Y = X$. We can also assume that $a$ takes the form $\omega^{b - |T|}a_0$ with $a_0 \in \text{zs}(T, X)$.

From Proposition \ref{prop:eta_basis}, we find that $N_{I, b}(X)/N_{I \cup \{T\}, b}(X)$ is generated by elements of the form
\[n = \eta\left(T, \,(x_{0s})_{s \in U}, m\right)\]
with $m$ lying in $N\left[\omega^{b - |T| + 1}\right]$. From \eqref{eq:sig_n}, we find that
\[\omega^{b- |T|} a_0 \cdot \beta(\sigma n) = \sigma \omega^{b- |T|} a_0 \cdot \beta(n),\]
 and the result follows.
\end{proof}

We finish this section with a construction that we will use to prove Proposition \ref{prop:variety} in the next section.
\begin{ctn}
\label{ctn:variety}
 With $X$, $S$, and $N$ taken as above, take
\[M = N_{\mathscr{P}(S), |S| + 1}(X)/ N_{\mathscr{P}(S),|S|}(X).\]
Suppose we are given a cocycle $\phi \in Z^1(G, M)$. 
Choose $x_0 = (x_{0s})_s$ in $X$ and $\sigma$ in $G$, and assume that $\sigma$ acts trivially on $M$. Set
\[m = \beta_{\chi(x_0)} \big( \omega^{|S|} \phi(\sigma)(x_0)\big).\]
This is equivalent to the statement that there is $m' \in N$ so $\omega^{|S|} m' = m$ and so
\[ \phi(\sigma)\,\equiv\, \eta(\emptyset, m') \,\,\text{ mod }  N(X)\left[\omega^{|S|}\right] + N_{\mathscr{P}(S),|S|}(X).\]

Choose $x_1 = (x_{1s})_s$ in $X$ so $x_{1s} \ne x_{0s}$ for each $s \in S$. Suppose we can find $\tau_s \in G$ so
\[\chi_{t}(x_t)(\tau_s) = \chi_t(x_{0t})(\tau_s)\]
for all $t \in S$ and $x_t \in X_t$ unless $t =s$ and $x_s = x_{1s}$, where we have
\[\chi_s(x_{1s})(\tau_s) \cdot \chi_s(x_{0s})(\tau_s)^{-1} = \xi.\]
We also suppose that the $\tau_s$ act trivially on $N[\omega]$ and that $\phi(\tau_s) = 0$ for all $s \in S$.

Then, for $s_0 \in S$, we may use \eqref{eq:sig_n} to show
\begin{align*}
&(\tau_{s_0} - 1)\phi(\sigma) \,\equiv\, \eta\left(\{s_0\},  \,(x_{1s})_{s\in \{s_0\}}, \, \omega m'\right)\,\,\text{ mod }\,\,  N(X)\left[\omega^{|S| - 1}\right] + N_{\mathscr{P}(S),|S|}(X) 
\end{align*}
Iterating, we find that
\[\left(\prod_{s \in S} (\tau_s - 1) \right)\phi(\sigma) \,=\, \eta(S, \,(x_{1s})_{s \in S}, m)\]
inside $M$, with the order of the product over $S$ not affecting its value. Writing $[\sigma_1, \sigma_2]$ for the commutator $\sigma_1^{-1}\sigma_2^{-1}\sigma_1\sigma_2$, we also have
\[\phi\left(\big[\tau_s^{-1}, \sigma_0^{-1}\big]\right) =  (\tau_s - 1) \phi(\sigma_0)\]
for any $\sigma_0 \in G$ acting trivially on $M$ and any $s \in S$.  Choosing some enumeration $\tau_1, \tau_2, \dots, \tau_{|S|}$ for the elements $(\tau_s)_{s \in S}$, we find
\begin{equation}
\label{eq:iterated_commutator}
\phi\left( \big[\tau_1^{-1}, \, \big[ \tau_2^{-1},\, \dots \big[ \tau_{|S|}^{-1},\,\, \sigma^{-1} \big]\dots \big] \big]\right) = \eta(S, \,(x_{1s})_{s \in S}, m).
\end{equation}

\end{ctn}

\section{Manipulating higher Selmer groups in grids}
\label{sec:gov}

\subsection{Algebraic results}
\label{ssec:TAR}
The goal of this section is to prove Theorem \ref{thm:higher_smallgrid}, which implies Theorem \ref{thm:main_higher} by the work in Section \ref{sec:regrid}. We will start by using the tools of the previous section to find relationships between Selmer elements of the twists in the grids considered in this theorem. 

We will start by recalling some of the notation of Theorem \ref{thm:higher_smallgrid}. This notation will be fixed throughout this section, and we will assume that the objects obey all the hypotheses going into Theorem \ref{thm:higher_smallgrid}.

\begin{notat}
\label{notat:smallgrid_review}
In Theorem \ref{thm:higher_smallgrid}, we started with a grid of ideals $X = \prod_{s \in S} X_s$ and an associated map $\chi: X \to \Hom(G_F, \FFF)$. We partitioned $S$ as
\[S = S_{\textup{con}} \cup S_{\textup{act}} = S_{\textup{con}} \cup  S_{\textup{pre}} \cup \{\saaa, \sbbb\},\]
and we chose a prefix
\[Z_{\text{pre}} = X_{\textup{pre}} \times \prod_{s \in S_{\text{con}}}\left\{\ovp_{0s}\right\}  \subseteq \prod_{s \in S \backslash \{\saaa, \sbbb\}} X_s,\]
where $X_{\textup{pre}}$ is defined as
\[ X_{\textup{pre}} = \prod_{s \in S_{\text{pre}}} X_{\text{pre}, s}.\]
We took $k = \# S_{\textup{pre}} \ge 1$, and we chose an element $x_0$ in $X$. This point was used to fix a grid class $\class{x_0}$ and higher grid  class $\class{x_0}_k$.

In this section, we will take the additional notation
\[\Vplac_{1} = \{\ovp_{0s} \cap F\,:\,\, s\in S_{\text{con}}\}\quad\text{and }\quad\Vplac_{\text{pre}} = \bigcup_{s \in S_{\text{pre}}} \{\ovp \cap F\,:\,\, \ovp \in X_{\text{pre}, s}\}.\]
\end{notat}

\begin{defn}
\label{defn:strict_cond}
Choose any $x \in \class{x_0} \cap Z_{\text{pre}} \times X_{\saaa} \times X_{\sbbb}$, and note that the kernel
\[\ker\left(Z^1(G_F, N[\omega]) \to \prod_{v \in \Vplac_0 \cup \Vplac_1} H^1(G_v, N^{\chi(x)})/W_v(\chi(x)) \times \prod_{v \not \in \Vplac_0 \cup \Vplac_1} H^1(I_v, N^{\chi(x)})\right),\]  
does not depend on the specific choice of $x$.

We call a finite subset $\Sigma_{\textup{strict}}$ of  $G_F$ a \emph{set of strict conditions} if every nonzero cocycle $\phi$ in this kernel satisfies $\phi(\sigma) \ne 0$ for some $\sigma \in \Sigstrict$ and if no proper subset of  $\Sigstrict$ has this property. Comparing the local conditions of the above kernel with the local conditions for the $\omega$-Selmer group of $N^{\chi(x)}$, we find that this minimal set of strict conditions satisfies
\begin{equation}
\label{eq:sigstrict}
\# \Sigstrict \,\le\, r_{\omega}\left(N^{\chi(x_0)}\right) + (k + 3)\dim N[\omega].
\end{equation}

\end{defn}

To state our algebraic results, we will need to move to smaller grids than the ones considered in Theorem \ref{thm:higher_smallgrid}. This requires the following notation.
\begin{notat}
\label{notat:Xcirc}
Recall the definition of $X_{\sbbb}(Z_{\textup{pre}})$ and $X_{\saaa}(Z_{\textup{pre}})$ from the passage before Theorem \ref{thm:higher_smallgrid}. Choose some $\ovp_{0b}$ in $X_{\sbbb}(Z_{\textup{pre}})$, and choose a nonempty subset
\[X'_{\saaa} \subseteq X_{\saaa}(Z_{\textup{pre}})\]
such that $\{\ovp_{0b}\} \times X'_{\saaa} \times Z_{\textup{pre}}$ is carried into the grid class $[x_0]$ under the standard inclusion.

We will take $S^{\circ} = S_{\textup{pre}} \cup \{\saaa\}$ and
\[X^{\circ} = X'_{\saaa} \times X. \]
We embed $X^{\circ}$ into $X$ via the composition
\[X^{\circ} \cong X^{\circ} \times \{\ovp_{0b}\} \times  \left(\left\{\ovp_{0s}\right\}\right)_{s \in S_{\text{con}}}\hookrightarrow X.\]

Given $x \in X^{\circ}$, we define the twist $\chi(x)$ as $\chi(x')$, where $x'$ is the image of $x$ under this embedding, and we take $N(X^{\circ}) = \bigoplus_{x \in X^{\circ}} N^{\chi(x)}$.

We will need to define certain submodules of $N(X^{\circ})$ using the construction from the beginning of Section \ref{ssec:interesting}. Take $I_1$ to be the power set $\mathscr{P}(S^{\circ})$, and take $I_0$ to be the power set $\mathscr{P}(S_{\text{pre}})$. We think of $I_0$ as a subideal of $I_1$. Given a subset $Y$ of $X^{\circ}$, we then take $N_0(Y)$ as notation for $N_{I_0}(Y)$ and $N_1(Y)$ as notation for $N_{I_1}(Y)$, where the underlying set of indices for both modules is $S^{\circ}$.
\end{notat}

\begin{defn}
\label{defn:strict_class}
Fix a set of strict conditions $\Sigstrict$, and take $X^{\circ}$ as in Notation \ref{notat:Xcirc}. Choose $x_1$ in $X^{\circ}$ and a cocycle  $\phi_{x_1}$ representing a class in $\Sel^{\omega^k} N^{\chi(x)}$. Given $x \in X^{\circ}$, we say that $x$ is in the \emph{strict class} of $x_1$ with respect to $\phi_{x_1}$ if there is a cocycle $\phi_x$ representing a class in $\Sel^{\omega^k} N^{\chi(x')}$ such that
\begin{alignat*}{2}
&\beta_x \circ \phi_x(\TineF{F}{\pi_s(x)}) =  \beta_{x_1} \circ \phi_{x_1}(\TineF{F}{\pi_s(x_1)})\quad&&\text{for all } s \in S^{\circ} = S_{\text{pre}} \cup \{\saaa\},\\
&\beta_x \circ \phi_x(\sigma) = \beta_{x_1} \circ \phi_{x_1}(\sigma) \quad&&\text{for } \sigma \in \Sigstrict,\, \text{ and }\\
&\beta_x \circ \phi_x(\TineF{F}{\ovp_{0b}}) = \beta_{x_1} \circ \phi_{x_1}(\TineF{F}{\ovp_{0b}}).
\end{alignat*}
If these conditions are satisfied, we call $\phi_x$ a \emph{strict class certificate} for $x$.  We take $\class{x_1}_{\text{strict}, \,\phi_{x_1}}$ to be the set of $x$ in the strict class of $x_1$ with respect to $\phi_{x_1}$. 
\end{defn}

With strict classes defined, we can finally say something about how the higher Selmer groups of the twists indexed by $X^{\circ}$ are related.
\begin{lem}
\label{lem:closure_noram}
Choose $x_1$ and $\phi_{x_1}$ as in Definition \ref{defn:strict_class}, and suppose that $\phi_{x_1}(\TineF{F}{\pi_s(x_1)})$ lies in $N[\omega^{k-1}]$ for all $s \in S^{\circ}$. Take $Y = \class{x_1}_{\textup{strict}, \,\phi_{x_1}}$, and take $\phi_x$ to be a strict class certificate for each $x \ne x_1$ in $Y$.  Then
\begin{enumerate}
\item The tuple $(\phi_x)_{x \in Y}$ is in the image of $Z^1(G_F, \omega N_{1}(Y))$, and
\item The set $Y$ is closed with respect to $X^{\circ}$ and $S^{\circ}$, in the sense of Definition \ref{defn:closed}.
\end{enumerate}
\end{lem}

For the proof of this lemma, we introduce the following piece of notation; given $a \in \Z[\xi]^Y$ for a given subset $Y$ of $X$, and given $\ovp$ in $X_s$ for some $s \in S_{\textup{act}}$, define $a\delta_{\ovp} \in \Z[\xi]^Y$ by
\[a\delta_{\ovp}(x) = \begin{cases} a(x) &\text{ if } \pi_s(x) = \ovp \\ 0 &\text{ otherwise.} \end{cases}\]
We also will write $\sum_{x \in Y} a(x)$ simply as $\sum a$.

\begin{proof}
Take $\phi \in Z^1(G_F, N(Y))$ to be the cocycle with $x$-component $\phi_x$ for $x \in Y$. We havve $\omega N_1(Y) = N_{\mathscr{P}(S^{\circ}) \backslash \{S^{\circ}\},\, k}(Y)$, so, to prove the first part, it suffices to prove that $a \cdot \beta(\phi)$ is zero for all $a \in \zs_{I, k}(Y)$ with $I = \mathscr{P}(S^{\circ}) \backslash \{S^{\circ}\}$. We do this by induction on the ideal $I$. For $I = \emptyset$, the result is clear.

Now suppose we have proved the result for $I$ and wish to show it for an ideal of the form $I \cup \{T\}$. Choosing $a \in \zs_{I \cup \{T\}, k}(Y)$, we know that $a \cdot \beta(\phi)$ lies in $Z^1(G_F, N[\omega])$ by Proposition \ref{prop:equivar_NIb}.  By the definition of $\zs_{I \cup \{T\}, k}(Y)$, we can find elements $f_U$ in $\zs(U, X)$ for each $U \in I \cup \{T\}$ so that
\begin{equation}
\label{eq:rewrite_a}
\sum_{U \in I \cup \{T\}} \omega^{k - |U|} f_U(x) = \begin{cases}  a(x) &\text{ if  }x \in Y \\ 0 &\text{ if } x \in X^{\circ} \backslash Y.\end{cases}
\end{equation}
In particular, we find that $\sum a = \omega^k \sum f_{\emptyset}$, and is hence divisible by $\omega^k$. As a result, we have that $a \cdot \beta(\phi)(\tau) = 0$ for
\begin{equation}
\label{eq:sigstrict_plus}
 \tau \in \Sigstrict \cup\left \{\TineF{F}{\ovp_{0b}}\right\}.
\end{equation}
For $v \in \Vplac_0 \cup \Vplac_1$, the restriction of $\chi(x)$ to $G_v$ does not depend on the choice of $x \in X^{\circ}$. So we find that $a \cdot \beta(\phi)$ maps into $W_v(\chi(x))$ for $v \in \Vplac_0 \cup \Vplac_1$ and any/every $x \in X^{\circ}$.

Now choose any $s$ in $S^{\circ}$ and $\ovp$ in $X_s$. From \eqref{eq:rewrite_a}, we see that $\sum a\delta_{\ovp}$ is divisible by $\omega^{k-1}$. From our restriction on the ramification of the $\phi_{x_1}$ at $\pi_s(x_1)$, we can conclude that $a \cdot \beta(\phi)$ is $0$ at $\TineF{F}{\ovp}$.

At all other primes, we find that $a \cdot \beta(\phi)$ is unramified since the cocycles $\phi_x$ were unramified for each $x \in Y$. The definition of a set of strict conditions allows us to conclude that $a\cdot\beta(\phi) = 0$. By induction, we can conclude the first part.

For the second part, suppose that $y$ is in the closure of $Y$ but is outside $Y$. We have an isomorphism
\[ N_1(Y \cup \{y\}) \cong N_1(Y)\]
by Proposition \ref{prop:Nib_bij}. We can thus view $\phi$ as a cocycle with coefficients in $\omega N_1(Y \cup \{y\})$, and we can consider its $y$-component $\phi_y \in Z^1(G_F, N^{\chi(y)}[\omega^k])$. We need to show that this cocycle represents a element in the $\omega^k$-Selmer group of $N^{\chi (y)}$, and that it is a strict class certificate for $y$.

So choose $a \in \zs(Y \cup \{y\})$ satisfying $a(y) = 1$. We then have $a \cdot \beta(\phi) = 0$, and we can conclude that $\phi_y$ obeys the local conditions at $\Vplac_0 \cup \Vplac_1$ since the $\phi_x$ obey the local conditions at these places for $x \in Y$. Since $\sum a = 0$, we may also conclude that $\beta_x \circ \phi_x(\tau) = \beta_y \circ \phi_y(\tau)$ for $\tau$ satisfying \eqref{eq:sigstrict_plus} and $x \in Y$.

For $s \in S^{\circ}$, we see that $a\delta_{\pi_s(y)} \cdot \beta(\phi)$ is $0$ since $a \delta_{\pi_s(y)}$ is in $\zs_{\mathscr{P}(S^{\circ}) \backslash \{S^{\circ}\},\, k}(Y \cup \{y\})$. Considering this, we find that the only places of $F$ where $\phi_y$ can be ramified are those in $\Vplac_0 \cup \Vplac_1$ and the primes of the form $\pi_s(y) \cap F$ for some $s \in S_{\textup{act}}$.  We also note that, if a given $x \in Y$ satisfies $\pi_s(y) = \pi_s(x)$, then $\phi_x$ obeys the local conditions at $v =\pi_s(y)\cap F$ and $\chi(x)$ and $\chi(y)$ have equal restriction at $G_v$. So the condition $a\delta_{\pi_s(y)} \cdot \beta(\phi) = 0$ implies that $\phi_y$ satisfies the local conditions at $\pi_s(y) \cap F$. Since $\sum a \delta_{\ovp_{\pi_s(y)}} = 0$, we also find that
\[\beta_y \circ \phi_y\left(\TineF{F}{\pi_s(y)}\right) = \beta_{x} \circ \phi_{x}\left(\TineF{F}{\pi_s(x)}\right)\]
for $x \in Y$.

So $\phi_y$ is a strict class certificate for $y$. This contradicts the definition of $Y$, so we can conclude that $Y$ is closed.
\end{proof}

We will need a version of this result that can handle cocycles with extra ramification at the prime indexed by $\saaa$. This will require some extra notation.
\begin{defn}
\label{defn:Mgov}
Given $x \in X_{\text{pre}}$, define
\[\chi(x) = \sum_{s \in S_{\textup{pre}}} \mfB_{\pi_s(x), \FFF}(h_s)\quad\text{and}\quad N(X_{\text{pre}}) = \bigoplus_{x \in X_{\text{pre}}} N^{\chi(x)}.\]
Following the construction of Section \ref{ssec:interesting} with respect  to the set of indices $S_{\text{pre}}$, we may consider the submodule $N_{\mathscr{P}(S_{\text{pre}}), k+1}(X_{\text{pre}})$ of this direct sum. We then define the \emph{governing subquotient} by
\[M_{\text{gov}} = \frac{N_{\mathscr{P}(S_{\text{pre}}), k+1}(X_{\text{pre}})}{\omega N_{\mathscr{P}(S_{\text{pre}}), k+1}(X_{\text{pre}})}.\]
This module is $\omega$-torsion. Using \eqref{eq:sig_n}, we find that $M_{\text{gov}}$ is a $\Gal(K_{\text{pre}}/F)$-module.

For each $\ovp \in X_{\saaa}$, choose a cocycle
\[\Phi_{\ovp} \in Z^1(G_F, M_{\textup{gov}}).\]
We call the collection $(\Phi_{\ovp})_{\ovp \in X_{\saaa}}$ a \emph{governing expansion} if there is some $n_a \in M_{\textup{gov}}$ for which
\[\Phi_{\ovp}(\Tine_F\, \ovp) = n_a \quad\text{for all }\, \ovp \in X_{\saaa}\]
and if there is some $m_{a} \in N[\omega^{k+1}]$  so that $n_a$ equals the image of  $(\beta^{-1}_{\chi(x)} m_{a})_{x \in X_{\text{pre}}}$ under the natural projection from $N_{\mathscr{P}(S_{\text{pre}}), k+1}(X_{\text{pre}})$ to $M_{\textup{gov}}$.

Given a governing expansion $(\Phi_{\ovp})_{\ovp \in X_{\saaa}}$, we define the \emph{governing class} $\class{\ovp_0}_{\text{gov}}$ to be the set of primes $\ovp$ in $X_{\saaa}(Z_{\text{pre}})$ such that
\begin{enumerate}
\item $\Phi_{\ovp}  - \Phi_{\ovp_0}$ is unramified at all places of $F$ besides $\ovp \cap F$ and $\ovp_0 \cap F$,
\item For every place $v \in \Vplac_0 \cup \Vplac_1 \cup \Vplac_{\text{pre}}$, the restriction of $\Phi_{\ovp} - \Phi_{\ovp_0}$ to $G_v$ is $0$, and
\item We have
\[\Phi_{\ovp}(\sigma) = \Phi_{\ovp_0}(\sigma) \quad\text{for }\, \sigma \in \Sigstrict. \]
\end{enumerate}
\end{defn}
\begin{rmk}
The use of the term `governing' in this definition has its roots in the paper \cite{CoLa83}. The authors of this paper conjectured that, given a nonzero integer $d$ and a positive integer $k$, there is a number field $L_d$ such that the $2^k$-class rank of $\QQ(\sqrt{dp})$ can be determined by Frobenius element of $p$ in $\Gal(L_d/\QQ)$ for any prime $p$ not dividing $2d$. The hypothetical field $L_d$ was called a \emph{governing field}.

For  $k > 3$, this conjecture is likely incorrect \cite{KoMi21}. However, a result of this section is that the value of $\Phi_{\ovp}$ at $\Frob_F\, \ovp_{0b}$ has some impact on the structure of the $\omega^{k+1}$-Selmer groups in $X^{\circ}$, which partially vindicates this conjecture.
\end{rmk}

\begin{notat}
\label{notat:that_map}
Suppose we have chosen a governing expansion $(\Phi_{\ovp})_{\ovp \in X_{\saaa}}$. Choose a subset $X'_{\saaa}$ of $X_{\saaa}(Z_{\textup{pre}})$ lying in a single governing class, and take $X^{\circ}$ as in Notation \ref{notat:Xcirc}.

By identifying $N^{X^{\circ}}$ with $\bigoplus_{\ovp \in X_{\saaa}'} N^{X_{\textup{pre}}}$, we may define a non-equivariant isomorphism
\[N(X^{\circ}) \cong \bigoplus_{\ovp \in X_{\saaa}'} N(X_{\textup{pre}})\]
Taking subquotients on both sides, this isomorphism descends to an equivariant isomorphism
\[N_0(X^{\circ})/N_1(X^{\circ}) \cong \bigoplus_{\ovp \in X_{\saaa}'} M_{\text{gov}}\Big/\Delta(M_{\text{gov}}),\]
where $\Delta(M_{\text{gov}})$ is the image of $M_{\text{gov}}$ under the diagonal map. We take $\Phi$ to be the the image of $(\Phi_{\ovp})_{\ovp \in X'_{\saaa}}$ in $Z^1(G_F, \, N_0(X^{\circ})/N_1(X^{\circ}))$. Multiplication by $\omega$ defines a surjection
\[N_0(X^{\circ})/N_1(X^{\circ}) \to \omega N_0(X^{\circ})/\omega N_1(X^{\circ}),\]
and we take $\omega \Phi$ to be the image of $\Phi$ under this map.
\end{notat}

In the case that the $\Phi_{\ovp}$ are all $0$, the following lemma recovers Lemma \ref{lem:closure_noram}.
\begin{lem}
\label{lem:closure_ram}
Take all notation as in Notation \ref{notat:that_map}, and define $m_a \in N[\omega^{k+1}]$ from the governing expansion $(\Phi_{\ovp})_{\ovp \in X'_{\saaa}}$ as in Definition \ref{defn:Mgov}. Choose $x_1$ in $X^{\circ}$ and a cocycle $\phi_{x_1}$ representing a class in $\Sel^{\omega^k} N^{\chi(x_1)}$. We suppose that
\begin{equation}
\label{eq:happy_ram}
\omega^{k - 1}\beta_{x_1} \left(\phi_{x_1}(\TineF{F}{\pi_s(x_1)})\right) = \begin{cases} \omega^{k} m_{a} &\text{ for } s = \saaa, \text{ and } \\ 0 &\text{ for } s \in S_{\textup{pre}}. \end{cases}
\end{equation}
Take $Y$ to be the strict class of $x_1$ with respect to $\phi_{x_1}$, and take $\phi_x$ to be a strict class certificate for each $x$ in $Y$. Then
\begin{enumerate}
\item The tuple  $(\phi_x)_{x \in Y}$ lies in $Z^1(G_F, \omega N_{0}(Y))$, and its image  in 
\[Z^1(G_F, \omega N_0(Y)/\omega N_1(Y))\]
equals the image of $\omega \Phi$, and
\item The set $Y$ is closed.
\end{enumerate}
\end{lem}

\begin{proof}
The proof of this lemma largely follows that of Lemma \ref{lem:closure_noram}. First, note that $a \cdot \beta(\Phi)$ gives a well-defined map to $N$ for $a$ in $\zs_{I_1, k + 1}(Y)$.

Take $ \Phi'$ to be the projection of $\Phi$ to $ N_0(Y)/ N_1(Y)$. Define $\phi$ from $(\phi_x)_{x \in Y}$ as in Lemma \ref{lem:closure_noram}.  For the first part, we need to check that $a \cdot \beta(\phi - \omega \Phi')$ is zero for all $a \in \zs_{I, k}(Y)$ with $I= \mathscr{P}(S^{\circ}) \backslash \{S^{\circ}\}$. We do this by induction on the ideal $I$. For $I = I_0$, the result follows as in Lemma \ref{lem:closure_noram}, with $a \cdot \beta(\omega\Phi') = 0$ following from the fact that $\Phi'$ takes values in a quotient of $N_0(Y)$. Now suppose we know it the result for $I \supseteq I_0$, and wish to show it for the ideal $I \cup \{T\}$. 

For $a \in \zs_{I \cup \{T\}, k}(X^{\circ})$, we claim that
\begin{equation}
\label{eq:omegaPhiTriv}
a \cdot \beta(\omega \Phi)(\tau) = 0 \quad \text{for}\quad\tau \in \Sigstrict  \,\cup\, \bigcup_{v \,\not\in \,\{\ovp \cap F\,:\,\, \ovp \in X'_{\saaa}\}} I_v \,\cup  \bigcup_{v \in \Vplac_0 \cup \Vplac_1 \cup \Vplac_{\text{pre}}} G_v
\end{equation}
To prove this claim, we first suppose that $a$ is supported within a set of the form $X_{\text{pre}} \times \{\ovp_{1a}, \ovp_{2a}\}$ with $\ovp_{1a}, \ovp_{2a}$ distinct primes, and that $a$ satisfies
\[a((x_0, \ovp_{1a})) = -a((x_0, \ovp_{2a}))\]
for all $x_0 \in X_{\text{pre}}$. Given such an $a$, note that the element $\omega a\delta_{\ovp_{1a}}$ projects to an element $a'$ in $\text{zs}_{I_0, k}(X_{\text{pre}})$ and that $\omega a \delta_{\ovp_{2a}}$ projects to $-a'$. So $a' \cdot \beta(\Phi_{\ovp_{ia}})$ is well defined for $i = 1, 2$, and \eqref{eq:omegaPhiTriv} follows for this $a$ from the definition of a governing class.

Now, any $a$ in $\zs_{I \cup \{T\}, k}(X^{\circ})$ may be written as a sum of elements of this form and some element $a_0$ supported on a set of the form $X_{\text{pre}} \times \{\ovp_{1a}\}$. This element $a_0$ must lie in $\zs_{I_0, k}(X^{\circ})$, so $a_0 \cdot \beta(\omega \Phi) = 0$. We may conclude that \eqref{eq:omegaPhiTriv} holds for $a \in \zs_{I \cup \{T\}, k}(X^{\circ})$. 

We now consider $a \cdot \beta(\phi - \omega \Phi')$ for some $a \in \zs_{I \cup \{T\}, k}(X^{\circ})$. From the induction step, this lies in $Z^1(G_F, N[\omega])$. Following the method of Lemma \ref{lem:closure_noram}, we see the first part of the lemma will follow if we can show that $a \cdot \beta(\phi - \omega \Phi')$ is unramified at every place divisible by a prime in $X_{\saaa}'$. So choose $\ovp$ in $X_{\saaa}'$, and choose some $x \in Y$ with $\pi_{\saaa}(x) = \ovp$. We find that
\[a\delta_{\ovp} \cdot \beta\left(\omega \Phi'\left(\Tine_F \ovp\right)\right) = \left(\sum a\delta_{\ovp}\right)\omega m_a = a\delta_{\ovp} \cdot \beta\left(\phi\left(\Tine_F \ovp\right)\right),\]
with the equalities following by the definition of a strict class, the hypothesis \eqref{eq:happy_ram}, and the fact that $\sum a\delta_{\ovp}$ is divisible by $\omega^{k-1}$. We may conclude that $a \cdot (\phi - \omega \Phi')$ is $0$, giving the first part.

For the second part, suppose $y$ lies in the closure of $Y$, and choose $a \in \zs(Y \cup \{y\})$ with $a(y) = 1$. Complete $\phi - \omega \Phi'$ to $y$ as in the proof of Lemma \ref{lem:closure_noram}. We see $a\delta_{\ovp}$ is in $\zs_{I_0, k}(Y \cup \{y\})$ for $\ovp$ in $X_{\saaa}'$, so $a\delta_{\ovp} \cdot \beta(\phi)$ is $0$. Because of this, $\phi_y$ satisfies the local conditions at all places except potentially at the primes indexed by $S_{\text{pre}}$. But for $s$ in $S_{\text{pre}}$ and $\ovp$ in $X_{\text{pre}, s}$, we have the identity
\[a\delta_{\ovp} \cdot \beta(\phi - \omega \Phi') = 0.\]
We see that $a\delta_{\ovp}$ is in $\zs_{\mathscr{P}(S^{\circ}) \backslash S^{\circ}, k}(X)$, so \eqref{eq:omegaPhiTriv} shows that $a\delta_{\ovp} \cdot \beta(\omega \Phi')$ is $0$ on $G_{F, \ovp}$. So $\phi_y$ obeys the local conditions at $\ovp \cap F$. These were the last local conditions to check.
\end{proof}

Our final lemma gives some control for the Cassels--Tate pairings in the closed sets introduced in Lemmas \ref{lem:closure_noram} and \ref{lem:closure_ram}.
\begin{lem}
\label{lem:CTP_control}
Suppose we are in the situaiton of Lemma \ref{lem:closure_ram}, and take all notation as in that Lemma. We assume that $\Phi(\Frob_F\, \ovp_{0b})$ lies in the image of $N_0(X^{\circ})[\omega] \subseteq N_0(X^{\circ})$ in $N_0(X^{\circ})/N_1(X^{\circ})$.

Choose $\psi$ in $\beta_{x_1}\left(\Sel^{\omega} (N^{\vee})^{\chi(x_1)}\right)$. We assume that $\psi$ is unramified at the primes indexed by $S^{\circ}$. Taking $\psi_x = \beta_x^{-1}(\psi)$ for $x \in Y$, we see that $\psi_x$ lies in $\Sel^{\omega}(N^{\vee})^{\chi(x)}$. 

For $x \in Y$, take $\langle \phi_x, \,\psi_x \rangle$ to be the Cassels--Tate pairing of these cocycles as in \eqref{eq:CTP_kj}, and define $\textup{ct} \in \left(\tfrac{1}{\ell}\Z/\Z\right)^Y$ by taking $\textup{ct}(x) = \langle \phi_x, \,\psi_x \rangle$ for $x \in Y$.

Then, for any $a$ in $\zs(Y)$, we have
\[a \cdot \textup{ct} = \iota\big(\left(a \cdot \beta\left(\Phi(\FrobF{F}{\ovp_{0b}})\right)\right) \cdot \psi(\TineF{F}{\ovp_{0b}})\big),\]
where the outer product on the right is the pairing \eqref{eq:dual_torsion} and $\iota: \mu_{\ell} \to \tfrac{1}{\ell}\Z/\Z$ is the isomorphism taking the image of $\overline{\zeta}$ in $\mu_{\ell}$ to $\tfrac{1}{\ell}\Z/\Z$.
\end{lem}
In the case that $\phi_{x_1}(\Tine_F\, \pi_{\saaa}(x_1))$ lies in $N^{\chi(x_1)}[\omega^{k-1}]$, applying this lemma with the zero governing expansion gives $a \cdot \textup{ct} = 0$

\begin{proof}
This uses tools and terminology developed in \cite{MS21}. We consider the commutative diagram
\begin{equation}
\label{eq:BerkBrek}
\begin{tikzcd}[column sep =1.0em]
 0 \arrow{r} & N(Y)[\omega] \arrow{r} \arrow[d, equals] & N(Y)[\omega^{k+1}] \arrow{r}{\omega}& N(Y) [\omega^k] \arrow{r}  &0 \\
0 \arrow{r} & N(Y)[\omega]  \arrow{r} \arrow[d] &  N_0(Y) \arrow{r}{\omega} \arrow[u, hook]  \arrow[d, two heads] &\omega N_0(Y) \arrow{r} \arrow[u, hook]  \arrow[d, two heads]& 0\\
0 \arrow{r} & N(Y)[\omega]\big/N_1(Y)[\omega]  \arrow{r} & N_0(Y)/N_1(Y) \arrow{r} & \omega N_0(Y)/\omega N_1(Y) \arrow[r] & 0
 \end{tikzcd}
\end{equation}
with exact rows and endow the top row with local conditions coming from our choices in Definition \ref{defn:Sel}. The pairing associated to this exact sequence decomposes as the sum over $x$ in $Y$  of the pairing between $\Sel^{\omega^k} N^{\chi(x)}$ and $\Sel^{\omega} (N^{\vee})^{\chi(x)}$.

We then give the second row in \eqref{eq:BerkBrek} the local conditions so it is a pullback of the first row in the category of Galois modules decorated with local conditions, so both of the inclusions are strictly monic in the sense of \cite[Section 4.1]{MS21}. For the final row, we assign local conditions so the vertical morphism to $N_0(Y)/N_1(Y)$ is strictly epic, so the first horizontal morphism is strictly monic, and the final horizontal map is strictly epic. With these local conditions, we find that $(\phi_x)_{x \in Y}$ lifts to a Selmer element of $\omega N_0(Y)$, and that this Selmer elment maps to the image $\omega \Phi'$ of $\omega \Phi$ in $\omega N_0(Y)/\omega N_1(Y)$, another Selmer element.

The object in this diagram with the least straightforward local conditions is the first term of the last row. Fortunately, it is unnecessary to determine these local conditions precisely. We have a diagram
\[\begin{tikzcd}
N(Y)[\omega] \arrow[rr, bend left = 10, "a"]  \arrow{r} & N(Y)[\omega]\big/N_1(Y)[\omega]  \arrow[r, "a" '] & N[\omega]
\end{tikzcd}\]
of $G_F$ modules, where $a$ denotes the map taking $n$ to $a \cdot \beta(n)$. For $v \in \Vplac_0 \cup \Vplac_1 \cup \{\ovp_{0b}\cap F\}$, this triangle fits into a diagram
\[\begin{tikzcd}[row sep = small]
N(Y)[\omega] \arrow[rr, bend left = 10, "a"] \arrow[r]  \arrow[d] & N(Y)[\omega]\big/N_1(Y)[\omega] \arrow[r, "a" '] \arrow[d] & N[\omega] \arrow[d, equals] \\
N_0(Y)  \arrow[d, hook] \arrow[r, two heads] & N_0(Y)/N_1(Y) \arrow[r, "a" ] & N[\omega]\arrow{d}\\
N(Y)  \arrow[r, "\sim"] & \bigoplus_{y \in Y} N^{\chi(x_1)}   \arrow[r, "a" ']  & N^{\chi(x_1)}
\end{tikzcd}\]
of $G_v$ modules, so the image of the local conditions at $v$ under 
\[N(Y)[\omega]\big/N_1(Y)[\omega]\xrightarrow{\,\,a\,\,} N[\omega]\]
is a subgroup of the local conditions for the $\omega$-Selmer group of $N^{\chi(x_1)}$ at $v$. At other places not indexed by $S_{\text{pre}} \cup \{\saaa\}$, we see that the image of the local conditions under $a$ sit inside the set of unramified classes.

In particular, since $\psi$ is unramified at  $\ovp \cap F$ for all $\ovp$ in $\bigcup_{s \in S^{\circ}} X_s$, and is hence trivial at all of these primes, we find that $a^{\vee} \psi$ lies in $\Sel (N(Y)[\omega]\big/N_1(Y)[\omega])^{\vee}$. Applying naturality of the Cassels--Tate pairing to \eqref{eq:BerkBrek}, we find that
\[a\cdot \textup{ct} = \big\langle \omega \Phi',\, a^{\vee} \psi \big \rangle,\]
where the Cassels--Tate pairing on the right is taken with respect to the final exact sequence of \eqref{eq:BerkBrek}. The cocycle $\Phi'$ is trivial at places in $\Vplac_0 \cup \Vplac_1$ and ramified only at places dividing a prime in $X_{\saaa}'$, where $\psi$ is necessarily trivial. We also find that $\omega\Phi'(\FrobF{F}{\ovp_{0b}})$ is zero. Taking $v = \ovp_{0b} \cap F$, we are left with
\[a \cdot \text{ct} = \text{inv}_{v}\left( \res_{G_v}a \cdot \beta(\Phi)\,\cup\, \res_{G_v}\, \psi\right)\]
by the definition of the Cassels--Tate pairing, where the cup product is the one induced by the pairing \eqref{eq:dual_torsion} between $N[\omega]$ and $N^{\vee}[\omega]$. The right hand side of this identity can be reduced to a norm residue symbol by naturality of the cup product and \cite[Section XIV.1.3]{Serre79}, leaving the claimed identity.

\end{proof}

Lemmas \ref{lem:closure_ram} and \ref{lem:CTP_control} give us control over individual Selmer elements and individual matrix coefficients in the Cassels--Tate pairing. The final result of this subsection gives a more holistic view of the $\omega^k$-Selmer groups over $X^{\circ}$ and their Cassels--Tate pairing. This result uses more of the hypotheses of Theorem \ref{thm:higher_smallgrid} on the Selmer groups of twists in the higher grid class $[x_0]_k$. The following notation will be helpful.
\begin{notat}
\label{notat:restraint}
Choose a governing expansion $(\Phi_{\ovp})_{\ovp \in X_{\saaa}}$ and a choice of $X^{\circ}$ and $\ovp_{0b}$ as in Notation \ref{notat:that_map}, define $m_a$ as in Definition \ref{defn:Mgov}, and define $\Phi$ from the governing expansion as in Notation \ref{notat:that_map}. We say that the governing expansion \emph{shows restraint} at $\ovp_{0b}$ if there is some function $g: X^{\circ} \to \FFF_{\ell}$ such that $\beta\left(\Phi(\Frob_F\, \ovp_{0b})\right)$ equals the image of
\[ \left(g(x) \cdot \omega^{k}m_a\right)_{x \in X^{\circ}} \in N_0[\omega]^{X^{\circ}}\]
in $N_0(X^{\circ})/N_1(X^{\circ})$. If this condition is satisfied, we call $g$ an \emph{associated scalar function} to $\Phi$.
\end{notat}

\begin{thm}
\label{thm:TAR}
Suppose we are in the situation of Theorem \ref{thm:higher_smallgrid}, and take all notation as in that statement. Choose $X^{\circ}$ and $\ovp_{0b}$ as in Notation \ref{notat:Xcirc}, and choose a governing expansion $(\Phi_{\ovp})_{\ovp \in X_{\saaa}}$. We assume that the element $m_a$ measuring the ramification of the governing expansion satisfies
\begin{equation}
\label{eq:gov_exp_TAR}
 \omega^{k} m_a = \pi_{\saaa}(v_a) (\overline{\zeta}),
\end{equation}
where $\pi_{\saaa}(v_a) \in N[\omega](-1)$ is the element fixed in Definition \ref{defn:to_smallgrid}.  We also assume that the governing expansion shows restraint at $\ovp_{0b}$, and we take $g: X^{\circ} \to \Z/\ell Z$ to be an associated scalar function to the governing expansion. 

Then the portion of the set $X^{\circ}$ mapping into the higher grid class $\class{x_0}_k$ under the inclusion $X^{\circ} \hookrightarrow X$ may  be partitioned into closed sets $Y_1, \dots, Y_M$ with
\begin{equation}
\label{eq:num_strict_classes}
M \le \exp\Big( \log  \ell \cdot \dim N[\omega] \cdot (r_{\omega}(x_0) + \dots + r_{\omega^k}(x_0)) \cdot( k + \# \Sigstrict) \Big)
\end{equation}
so that, for any given $i \le M$ and any function $a$ in $\zs(Y_i)$, we have
\begin{equation}
\label{eq:CT_g}
a \cdot \textup{ct} = \nu a \cdot g,
\end{equation}
where $\textup{ct}: Y_i \to \tfrac{1}{\ell}\Z/\Z$ is defined by $\pi_x(\textup{ct}) = \textup{ct}_x(w)$. Here, $\nu$ is a nonzero element in $\tfrac{1}{\ell}\Z/\Z$ that does not depend on the choice of $Y_i$ or $a$.
\end{thm}

\begin{proof}
Fix some $s_1 \in S_{\text{pre}}$. Given $x \in X^{\circ}$, we define a \emph{strict signature map}
\[\text{sig}_x\,:\, Z^1(G_F, N^{\chi(x)}) \to   N^{\Sigstrict \cup S^{\circ} \backslash \{s_1\}}\]
by
\begin{alignat*}{2}
&\text{sig}_x(\phi)(\tau) = \beta_x\left(\phi(\tau)\right) &&\,\text{ for } \tau \in \Sigstrict \,\text{ and }\\
&\text{sig}_x(\phi)(s) = \beta_x\left(\phi\left(\Tine_{F}\, \pi_s(x)\right) \right)&&\,\text{ for } s\in S^{\circ}\backslash \{s_1\}.
\end{alignat*}
Take $V(x)$  to be the set of cocycles $\phi \in Z^1(G_F, N^{\chi(x)})$ mapping into $\Sel^{\omega^k} N^{\chi(x)}$ that are unramified at $\pi_{s_1}(x)$. We then define the strict signature at $x$ to be the image $\text{sig}_x (V(x))$ in $N^{\Sigstrict \cup S^{\circ}\backslash \{s_1\}}$. We note that $V(x)$ and the image of $H^0(G_F, N[\omega])$ under the connecting map corresponding to
\[0 \to  N^{\chi(x)}\left[\omega^{k}\right ]\to N^{\chi(x)}\left[\omega^{k+1}\right]\to N[\omega] \to 0\]
together generate $\Sel^{\omega^k} N^{\chi(x)}$. Taking $B(x)$ to be the kernel of the map from $V(x)$ to $\Sel^{\omega^k} N^{\chi(x)}$, we find that $B(x)$ consists of coboundaries valued in $N[\omega]$ and that $\text{sig}_x B(x)$ does not depend on the choice of $x$.

Given $x_1 \in Y$, take $Y(x_1)$ to be the set of points in $Y$ whose strict signature equals the strict signature of $x_1$. We claim this set is closed. This starts by noting that, for $\phi \in V(x_1)$, the set of points in $X^{\circ}$ whose strict signature contains $\text{sig}_{x_1}(\phi)$ is closed by Lemma \ref{lem:closure_ram}. Since intersections of closed sets are closed, we find that the set of points whose strict signature contains $\text{sig}_{x_1}(V(x_1))$ is closed.  Repeated applications of Lemma \ref{lem:CTP_control} and \cite[Theorem 1.3]{MS21} show that this set of points has closed intersection with $Y$ and that all the points in this intersection have the same strict signature.

In the context of Definition \ref{defn:to_smallgrid}, choose $\nu_0 \in \FFF_{\ell}$ such that the image of $w$ in 
\[\big(\text{k}V_{\omega}/\text{k}_0V_{\omega}\big) \otimes \big(\text{k}V^{\vee}_{\omega}/\text{k}_0V^{\vee}_{\omega}\big)\quad\text{or}\quad \bigwedge^2 \big(\text{k}V_{\omega}/ V_0\big)\]
is given by $\nu_0 (v_a \otimes v_b)$ or $\nu_0(v_a \wedge v_b)$, with the exterior product used in the case of alternating structure and the tensor product used otherwise. Then, for any $a \in \zs(Y(x_1))$, Lemma \ref{lem:CTP_control} gives
\[a \cdot \textup{ct} = \nu a \cdot g\]
for
\[\nu =  \nu_0 \cdot \omega^k m_a \cdot \pi_{\sbbb}(v_b),\]
where the right product is the natural pairing
\[N[\omega] \times N^{\vee}[\omega](-1) \to \tfrac{1}{\ell}\Z/\Z.\]
Note that $\nu$ is nonzero by part (4) or (5) of Definition \ref{defn:ready}.

So, to prove the theorem, it suffices to show that the number of distinct strict signatures that appear over $Y$ is bounded by the right hand side of \eqref{eq:num_strict_classes}. For $j > 0$, take $r_{\omega^j} = r_{\omega^j}(N^{\chi(x_0)})$. Then we have an isomorphism
\[V(x)/B(x) \cong (\Z/\ell^k \Z)^{r_{\omega^k}} \oplus \bigoplus_{j = 1}^{k-1} (\Z/\ell^k \Z)^{r_{\omega^j} - r_{\omega^{j+1}}}\]
of abelian groups. So it suffices to give an upper bound on how many images this abelian group can have in $N^{\Sigstrict \cup S^{\circ} \backslash \{s_1\}}$.
An application of the telescoping sum
\[r_{\omega} + \dots + r_{\omega^k} = k r_{\omega^k} + \sum_{j < k}  j \cdot \left(r_{\omega^j} - r_{\omega^{j+1}}\right) \]
then gives the theorem.
\end{proof}

This theorem and the combinatorial result Proposition \ref{prop:bye_Ramsey} fit together like a lock and a key. Combining them gives the following corollary.
\begin{cor}
\label{cor:TAR_comb}
Given a starting tuple $(K/F, \Vplac_0, \FFF)$, there is a $C > 0$ such that, for all $H > C$, we have the following.

Suppose we are in the situation of Theorem \ref{thm:higher_smallgrid}, and take all notation as in that statement. There is then a choice of a function
\[g: \{1, \dots, E_{\textup{pre-size}}\} \times   X_{\textup{pre}} \to \tfrac{1}{\ell}\Z/\Z,\]
for which the following holds:

Choose a governing expansion $(\Phi_{\ovp})_{\ovp \in X_{\saaa}}$ satisfying \eqref{eq:gov_exp_TAR}. Choose $X'_{\saaa}$ and $\ovp_{0b}$ as in Notation \ref{notat:that_map} such that $X'_{\saaa}$ is in one governing class and has cardinality $E_{\textup{pre-size}}$. We assume that the governing expansion shows restraint at $\ovp_{0b}$ with respect to $X'_{\saaa}$, and we presume that there is some bijection between $\{1, \dots, E_{\textup{pre-size}}\}$ and $X_{\saaa}'$ under which $g$ is identified with an associated scalar function to the governing expansion at $\ovp_{0b}$.

Then the test mean of $\textup{ct}_x(w)$ on $X'_{\saaa} \times \{\ovp_{0b}\} \times Z_{\textup{pre}}$ is at most $\frac{1}{3}\left(\log \log H\right)^{-1/4}$.
\end{cor}

\begin{proof}
From \eqref{eq:sigstrict}, we have
\[k + \# \Sigstrict \le 3 \cdot \left(r_{\omega}\left(N^{\chi(x_0)}\right) + k \right) \cdot \dim N[\omega].\]
So the right hand side of \eqref{eq:num_strict_classes} is bounded by $\left( \log \log H\right)^{1/10}$. Choose a function $g$ for which the hypothesis of Proposition \ref{prop:bye_Ramsey} holds with $M =\lceil \left( \log \log H\right)^{1/10}\rceil$. Applying Theorem \ref{thm:TAR} together with this proposition and the inequality of fractions
\[\tfrac{1}{2}\left(\tfrac{2}{3} - \tfrac{1}{10}\right) > \tfrac{1}{4}\]
gives the corollary.
\end{proof}

\subsection{The proof of Theorem \ref{thm:higher_smallgrid}}
Take all notation as in Theorem \ref{thm:higher_smallgrid} and Notation \ref{notat:smallgrid_review}.  Our goal is to apply the bilinear result Theorem \ref{thm:bilin_symb} to the product space $X_{\saaa}(Z_{\text{pre}}) \times X_{\sbbb}(Z_{\text{pre}})$ to produce many tuples $(X'_{\saaa}, \ovp_{0b})$ satisfying the condition of Corollary \ref{cor:TAR_comb}. Our application Theorem \ref{thm:bilin_symb} requires a new starting tuple, and we define this tuple first.

\begin{defn}
\label{defn:T}
Define $K(\Vplac_0)$ from the tuple $(K/F, \Vplac_0, \#\FFF)$ as in Definition \ref{defn:class}, and also define
\[K_{\text{pre}} = K(\Vplac_0 \cup \Vplac_{\text{pre}})\]
from this tuple.

There is some constant $C > 0$ determined from $(K/F, \Vplac_0, \#\FFF)$ for which that the degree $n_{K_{\textup{pre}}}$ of $K_{\textup{pre}}$ is bounded by
\[\exp\left(C k E_{\textup{pre-size}}\right),\]
and so the absolute value of its discriminant $\Delta_{K_{\text{pre}}}$ is bounded by
\[ C \cdot \exp^{(3)}\left(\tfrac{2}{3}\log^{(3)} H\right)^{\exp(C k E_{\textup{pre-size}})},\]
where we are using \eqref{eq:pre_disc_small} to bound the size of the primes where $K_{\text{pre}}$ is ramified. We have
\begin{align}
\label{eq:Kpre_disc}
&\Delta_{K_{\text{pre}}} \,\le\, \exp^{(3)}\left(\left(\tfrac{2}{3} + \epsilon\right) \log^{(3)} H\right)\quad\text{and}\\
\label{eq:Kpre_deg}
&n_{K_{\text{pre}}} \,\le\, \exp^{(2)} \left(\left(\tfrac{2}{3} + \epsilon\right) \log^{(3)} H\right)
\end{align}
for any $\epsilon > 0$ and $H$ sufficiently large relative to $(K/F, \Vplac_0, \FFF)$ and $\epsilon$.

Take $\Vplac_{\text{pre-unpack}}$ to be a minimal finite set of places of $F$ indivisible by any primes in $X_{\saaa}$ or $X_{\sbbb}$ such that  $(K_{\text{pre}}/F, \,\Vplac_0 \cup \Vplac_{\text{pre-unpack}},\, |\FFF|)$ is unpacked. We take the notation 
\[\mathbb{T} =  (K_{\text{pre}}/F, \,\Vplac_0 \cup \Vplac_{\text{pre-unpack}},\, |\FFF|).\] 
Given $\epsilon > 0$, we will have
\begin{equation}
\label{eq:Kpre_unpack}
\# \Vplac_{\text{pre-unpack}} \le \exp^{(2)} \left(\left(\tfrac{2}{3} + \epsilon\right) \log^{(3)} H\right).
\end{equation}
so long as $H$ is sufficiently large relative to $(K/F, \Vplac_0, \#\FFF)$ and $\epsilon$. Using \eqref{eq:KV0_deg}, we then find that the number of classes with respect to $\mathbb{T}$  has upper bound
\begin{equation}
\label{eq:T_class_bound}
\exp^{(3)} \left(\left(\tfrac{2}{3} + 2\epsilon\right) \log^{(3)} H\right).
\end{equation}
for sufficiently large $H$ relative to $(K/F, \Vplac_0, \#\FFF)$ and $\epsilon$.

Take $r_{0a}$ to be an element of $N[\omega^{k+1}](-1)$ satisfying $\omega^{k} r_{0a} = \pi_{\saaa}(v_a)$. Take $r_a$ to be the image of $(\beta_{\chi(x)}^{-1} r_{0a})_{x \in X_{\text{pre}}}$ in $M_{\text{gov}}(-1)$. Choose some ramification section for the unpacked starting tuple $\mathbb{T}$. For each $\ovp \in X_{\saaa}$, fix a cocycle
\[\Phi_{\ovp} \in Z^1(G_F, M_{\text{gov}})\]
in the cocycle class of $\mfB_{\ovp, M_{\text{gov}}}(r_a)$. This defines a governing expansion $(\Phi_{\ovp})_{\ovp \in X_{\saaa}}$ satisfying the conditions of Theorem \ref{thm:TAR}.
\end{defn}

\begin{notat}
\label{notat:PhiT_class}
Fix some set of strict conditions $\Sigstrict$ as in Definition \ref{defn:strict_cond}. We split $X_{\saaa}(Z_{\text{pre}})$ into equivalence classes so $\ovp_{0a}$ and $\ovp_{1a}$ lie in the same class if and only if they 
\begin{enumerate}
\item Satisfy $\class{\ovp_{0a}} = \class{\ovp_{1a}}$ with respect to the starting tuple $\mathbb{T}$ and
\item Lie in the same governing class with respect to $(\Phi_{\ovp})_{\ovp}$.
\end{enumerate}
We denote the equivalence class of $\ovp_{0a}$ by $\class{\ovp_{0a}}_{\Phi, \mathbb{T}}$
\end{notat}

\begin{prop}
\label{prop:PhiT_classes}
Given $\epsilon > 0$, $N$, and $(K/F, \Vplac_0)$, there is a $C > 0$ so, if $H > C$, the equivalence relation defined in Notation \ref{notat:PhiT_class} divides $X_{\saaa}(Z_{\textup{pre}})$ into at most
\[\exp^{(3)} \left(\left(\tfrac{2}{3} + \epsilon\right) \log^{(3)} H\right)\]
classes.
\end{prop}
\begin{proof}
Using \eqref{eq:T_class_bound}, we see that it suffices to consider the number of equivalence classes within a set of the form $\class{\ovp} \cap X_{\saaa}(Z_{\text{pre}})$, where the class is defined with respect to $\mathbb{T}$.

Note that the dimension of $M_{\textup{gov}}$ is bounded by $\#X_{\textup{pre}} = E_{\textup{pre-size}}^k$, which is bounded by $\left( \log \log H\right)^{\log^{(3)} H}$ for $H$. Also note that the number of places in $\Vplac_0 \cup \Vplac_1 \cup \Vplac_{\textup{pre}}$ can be bounded by $2 \left( \log \log H^2\right)^2$ by the first part of Definition \ref{defn:ready} for sufficiently large $H$. Together with \eqref{eq:sigstrict}, we find that the number of classes inside $\class{\ovp} \cap X_{\saaa}(Z_{\text{pre}})$ can be bounded by
\[\exp^{(2)}\left( 10 \cdot\log^{(3)}H \right)\]
for sufficiently large $H$, and the proposition follows.
\end{proof}

\begin{defn}
\label{defn:Cheb-like}
Take all notation as above, and choose $\ovp_{0a}$ in $X_{\saaa}(Z_{\textup{pre}})$. We say that $\class{\ovp_{0a}}_{\Phi, \mathbb{T}}$ is \emph{dense} if
\[\frac{\# \class{\ovp_{0a}}_{\Phi, \mathbb{T}}}{\# X_{\saaa}(Z_{\text{pre}})} \ge \exp^{(3)}\left(\tfrac{11}{16}\log^{(3)} H\right)^{-1}.\]
Given $\ovp_{0b}$ in $X_{\sbbb}(Z_{\textup{pre}})$, we say that $\class{\ovp_{0a}}_{\Phi, \mathbb{T}}$ is \emph{Chebotarev-like} at $\ovp_{0b}$ if
\begin{align}
\label{eq:good_pb}
&\sum_{s \in \symb{\class{\ovp_{0a}}}{\class{\ovp_{0b}}}} \left| \#\{\ovp_a \in \class{\ovp_{0a}}_{\Phi, \mathbb{T}} \,:\,\, \symb{\ovp_{a}}{\ovp_{0b}} = s\} \,-\, \frac{\#\class{\ovp_{0a}}_{_{\Phi, \mathbb{T}}}}{\#\symb{\class{\ovp_{0a}}}{\class{\ovp_b}}}\right| \\
&\nonumber\qquad\qquad\le \exp^{(3)}\left(\tfrac{11}{16}\log^{(3)} H\right)^{-1} \cdot \# \class{\ovp_{0a}}_{\Phi, \mathbb{T}},
\end{align}
where the set of symbols $\symb{\class{\ovp_{0a}}}{\class{\ovp_{0b}}}$ is defined with respect to $\mathbb{T}$.
\end{defn}

The following result is a corollary of our work in Section \ref{sec:bilinear}.
\begin{cor}
\label{cor:an}
If $H$ is sufficiently large relative to $N$ and $(K/F, \Vplac_0)$, we have the following.

Take $Y$ to be the subset of $(\ovp_{0a}, \ovp_{0b})$ in $X_{\saaa}(Z_{\textup{pre}}) \times X_{\sbbb}(Z_{\textup{pre}})$ for which $\class{\ovp_{0a}}_{\Phi, \mathbb{T}}$ is Chebotarev-like at $\ovp_{0b}$ and dense. Then
\[\frac{\# Y} {\# (X_{\saaa}(Z_{\textup{pre}}) \times X_{\sbbb}(Z_{\textup{pre}}))} \ge 1 - \exp^{(3)}\left(\tfrac{11}{16}\log^{(3)} H\right)^{-1/2}.\]
\end{cor}
\begin{proof}

We may use the suitability of the prefix (see Definition \ref{defn:to_smallgrid}) and the second part of Definition \ref{defn:ready} to give lower bounds on the sizes of $X_{\saaa}(Z_{\textup{pre}})$ and $X_{\sbbb}(Z_{\textup{pre}})$. Together with the bounds for the tuple $\mathbb{T}$ appearing in Definition \ref{defn:T}, we are able to use Theorem \ref{thm:bilin_symb} to prove
\[\sum_{\ovp_{0b} \in X_{\sbbb}(Z_{\text{pre}}) } \sum_{s \in \symb{\class{\ovp_{0a}}}{\class{\ovp_{0b}}}} \left| \#\{\ovp_a \in \class{\ovp_{0a}}_{\Phi, \mathbb{T}} \,:\,\, \symb{\ovp_{0a}}{\ovp_{0b}} = s\} \,-\, \frac{\#\class{\ovp_{0a}}_{\Phi, \mathbb{T}}}{\#\symb{\class{\ovp_{0a}}}{\class{\ovp_b}}}\right| \]
\[\le \exp^{(3)}\left(\tfrac{11}{16}\log^{(3)} H\right)^{-3} \cdot \# X_{\sbbb}(Z_{\text{pre}}) \cdot \# X_{\saaa}(Z_{\textup{pre}}),\]
for any $\ovp_{0a}$ so long as $H$ is sufficiently large relative to $(K/F, \Vplac_0)$ and $N$. 

In addition, from Proposition \ref{prop:PhiT_classes}, we know that the number of primes $\ovp_{0a}$ for which $\class{\ovp_{0a}}_{\Phi, \mathbb{T}}$ is not dense is at most
\[\exp^{(3)}\left(\tfrac{11}{16}\log^{(3)} H\right)^{-2/3} \cdot \#X_{\saaa}(Z_{\textup{pre}}) \]
for sufficiently large $H$. We thus can finish the proof by applying the above bilinear result to $\ovp_{0a}$ where $\class{\ovp_{0a}}_{\Phi, \mathbb{T}}$ is dense.
\end{proof}

Given a dense equivalence class $\class{\ovp_{0a}}$ and a prime $\ovp_{0b}$ where the equivalence class is Chebotarev-like, our next goal is to apply Corollary \ref{cor:TAR_comb} to subsets of $\class{\ovp_{0a}}_{\Phi, \mathbb{T}} \times \ovp_{0b}$. This requires the following two propositions.

\begin{prop}
\label{prop:variety}
Choose some $m$ in the image of 
\[\beta^{-1}\left(\FFF_{\ell}^{X_{\textup{pre}}} \otimes \langle  \pi_{\saaa}(v_a)(\overline{\zeta}) \rangle\right)\]
inside $M_{\textup{gov}}$. Also take 
\[L = K_{\textup{pre}}(\Vplac_0 \cup \Vplac_{\textup{pre-unpack}}) \cdot K(\Vplac_0 \cup \{\ovp_{0a} \cap F\}),\]
where $K_{\textup{pre}}(\Vplac_0 \cup \Vplac_{\textup{pre-unpack}})$ is defined with respect to $\mathbb{T}$. 

Then there is some $\sigma \in G_L$ such that
\[\Phi_{\ovp_{0a}}\left(\sigma\right) = m .\] 
\end{prop}

\begin{proof}
Choose an enumeration $s_1, \dots, s_k$ for $S_{\text{pre}}$. Using Construction \ref{ctn:variety}, we see that such a $\sigma$ can be constructed as a product of iterated commutators of the form
\[ \big[\TineF{F}{\ovp_1}^{-1}, \, \big[ \TineF{F}{\ovp_2}^{-1},\, \dots \big[ \TineF{F}{\ovp_k}^{-1},\,\, \TineF{F}{\ovp_{0a}}^{-1} \big]\dots \big] \big],\]
where $\ovp_i$ is some prime in $X_{\textup{pre}, s_i}$ for each $i \le k$.
\end{proof}

\begin{prop}
\label{prop:force_random}
For sufficiently large $H$ relative to $(K/F, \Vplac_0)$ and $N$, we have the following:

Choose any function
\[g: \{1, \dots, E_{\textup{pre-size}}\} \times X_{\textup{pre}} \to \tfrac{1}{\ell}\Z/\Z.\]
Also choose a $\ovp_{0a} \in X_{\saaa}(Z_{\textup{pre}})$ and $\ovp_{0b} \in X_{\sbbb}(Z_{\textup{pre}})$ so $\class{\ovp_{0a}}_{\Phi, \mathbb{T}}$ is Chebotarev-like at $\ovp_{0b}$ and dense.

Then we may write $\class{\ovp_{0a}}_{\Phi, \mathbb{T}}$ in the form
\[\class{\ovp_{0a}}_{\Phi, \mathbb{T}} = Y_{\textup{scrap}} \cup  Y^1 \cup \dots \cup Y^r,\]
where $r$ is some positive integer, and where
\[\# Y_{\textup{scrap}} \le  \exp^{(3)}\left(\tfrac{11}{16}\log^{(3)} H\right)^{-1/2}  \cdot \#\class{\ovp_{0a}}_{\Phi, \mathbb{T}}.\] 
 Here, the $Y^i$ are disjoint sets, and each $Y^i$ may be written in the form
\[Y^i = \{\ovp_{ai, 1}, \dots, \ovp_{ai, E_{\textup{pre-size}}}\}\]
so that
\[\beta_x \circ \pi_x\left(\Phi_{\ovp_{ai, j}}(\FrobF{F}{\ovp_{0b}}) -\Phi_{\ovp_{ai, 1}}(\FrobF{F}{\ovp_{0b}})\right) =(g(j,x) - g(1, x)) \cdot \pi_{\saaa}(v_a)(-1)\]
for $ j \le E_{\textup{pre-size}}$ and all $x \in X_{\textup{pre}}$. 
\end{prop}

\begin{proof}
In this proof, all classes and symbols are defined with respect to $\mathbb{T}$.

Given $k$ satisfying $1 \le j \le E_{\textup{pre-size}}$, and given $\ovp_a$ in $\class{\ovp_{0a}}$, we may use Proposition \ref{prop:variety} to find a prime $\ovp_b$ in the class of $\ovp_{0b}$ such that
\[ \Phi_{\ovp_{a}}\left(\FrobF{F}{\ovp_{b}}\right) = \Phi_{\ovp_{0a}}\left(\FrobF{F}{\ovp_{0b}}\right) +g(j) \cdot \pi_{\saaa}(v_a)(-1).\]
From Proposition \ref{prop:same_symbs}, given a prime $\ovp_{a}' \in \class{\ovp_a}$ satisfying
\[\symb{\ovp_a'}{\ovp_{0b}} = \symb{\ovp_a}{\ovp_b},\]
we have 
\[\Phi_{\ovp_{a}'}\left(\FrobF{F}{\ovp_{0b}}\right)  = \Phi_{\ovp_{a}}\left(\FrobF{F}{\ovp_{b}}\right).\]
The result then follows from Definition \ref{defn:Cheb-like}.
\end{proof}

We now can finish the proof of Theorem \ref{thm:higher_smallgrid}.  Take $g$ to be the function constructed in Lemma \ref{cor:TAR_comb}. Choose $\class{\ovp_{0a}}_{\Phi, \mathbb{T}}$ and $\ovp_{0b}$ satisfying the conditions of Proposition \ref{prop:force_random}, and apply this result to $g$. If $Y_i$ is one of the subsets of $\class{\ovp_{0a}}_{\Phi, \mathbb{T}}$ constructed in this proposition, we find that 
\begin{equation}
\label{eq:Yi0bZp}
Y_i \times \{\ovp_{0b}\} \times Z_{\text{pre}}
\end{equation}
either lies in $\class{x_0}$ or is disjoint from this grid class. In the former case, we can bound the test mean of $\text{ct}_x(w)$ over \eqref{eq:Yi0bZp} using Corollary \ref{cor:TAR_comb}.

Take $Y_{\text{ab scrap}}$ to be the set of $(\ovp_{a}, \ovp_{0b}) \in X_{\saaa}(Z_{\text{pre}}) \times X_{\sbbb}(Z_{\text{pre}})$ such that either
\begin{enumerate}
\item $\class{\ovp_a}_{\Phi, \mathbb{T}}$ is not dense or not Chebotarev-like at $\ovp_{0b}$ or
\item $\ovp_a$ is in the set $Y_{\text{scrap}}$ constructed in Proposition \ref{prop:force_random} for $\class{\ovp_a} \times \{\ovp_{0b}\}$.
\end{enumerate}
The size of this set has upper bound
\[2 \cdot \exp^{(3)}\left(\tfrac{11}{16}\log^{(3)} H\right)^{-1/2} \cdot \# (X_{\saaa}(Z_{\textup{pre}}) \times X_{\sbbb}(Z_{\textup{pre}}))\]
by Proposition \ref{prop:force_random} and Corollary \ref{cor:an}.

Taking $Y$ to be the set defined in Theorem \ref{thm:higher_smallgrid}, we can combine the estimates over sets of the form \eqref{eq:Yi0bZp} to prove that the test mean of $\text{ct}_x(w)$ over
\[Y \backslash Y \cap \left(Y_{\text{ab scrap}} \times Z_{\text{pre}}\right)\]
is at most $\tfrac{1}{3}(\log \log H)^{-1/4}$. We also find that
\[\# Y \ge C^{-1} \cdot \# (X_{\saaa}(Z_{\textup{pre}}) \times X_{\sbbb}(Z_{\textup{pre}}))\]
for sufficiently large $H$ and some $C >0$ depending on $(K/F, \Vplac_0)$. So we have
\[\# \left(Y_{\text{ab scrap}} \times Z_{\text{pre}}\right) \le 2C   \cdot \exp^{(3)}\left(\tfrac{11}{16}\log^{(3)} H\right)^{-1/2} \cdot \# Y.\]
In other words, the set $Y_{\text{ab scrap}} \times Z_{\text{pre}}$ has negligible cardinality compared to $Y$. The theorem follows.
\qed

\printbibliography

\end{document}